\documentclass{article}

\usepackage{xcolor}
\usepackage{amsmath,amstext,amsthm,amsfonts}
\usepackage[titletoc,title]{appendix}
\usepackage{makeidx}
\usepackage{amssymb,latexsym}
\usepackage{txfonts, pxfonts}
\usepackage[english]{babel}
\usepackage{hyperref}
\usepackage{geometry}

\usepackage[displaymath, mathlines, pagewise]{lineno}

\usepackage{tikz}
\usetikzlibrary{decorations.pathreplacing,angles,quotes}
\usetikzlibrary{arrows,shapes,positioning}

\numberwithin{equation}{section}
\theoremstyle{plain}
\newtheorem{theorem}{Theorem}[section]
\newtheorem{proposition}[theorem]{Proposition}
\newtheorem{lemma}[theorem]{Lemma}
\newtheorem{conjecture}[theorem]{Conjecture}
\newtheorem{corollary}[theorem]{Corollary}

\theoremstyle{definition}
\newtheorem{definition}[theorem]{Definition}

\newtheorem*{notation}{\textbf{Notation}}

\newtheorem*{convention}{\textbf{Convention}}
\newtheorem{remark}[theorem]{Remark}

\def \pa{\partial}
\def \Gx{G_{x_0}}

\def\e{\epsilon}

\def\N{\mathbb{N}}
\def\R{\mathbb{R}}
\def\Rn{{\mathbb{R}}^n_+}
\def\d{\partial}

\def\dvt{dv_{g(t)}}
\def\dv{dv}
\def\ds{d\sigma}
\def\D{\Delta}

\def\a{\alpha}
\def\b{\beta}
\def\l{\lambda}
\def\g{g_{\nu}}
\def\cesc{R_{g(t)}}
\def\cescbar{\overline{R}_{g(t)}}
\def\cmedia{H}
\def\cm{H_{g(t)}}

\def\cminfbar{\overline{R}_{\infty}}
\def\cmz{H_{g_0}}

\def\cez{R_{g_0}}
\def\cev{R_{\g}}

\def\bx{\bar{x}}
\def\bu{\bar{U}}

\def\u{u_{\nu}}
\def\uinf{u_{\infty}}
\def\uinu{\bar{u}_{(x_{i,\nu},\e_{i,\nu})}}
\def\ujnu{\bar{u}_{(x_{j,\nu},\e_{j,\nu})}}

\def\uknu{\bar{u}_{(x_{k,\nu},\e_{k,\nu})}}
\def\uk{\bar{u}_{(x_{k},\e_{k})}}

\def\w{w_{\nu}}
\def\crit {\frac{2n}{n-2}}
\def\critbordo{\frac{2(n-1)}{n-2}}
\def\conj {\frac{2n}{n+2}}
\def\ba{\begin{align}}
\def\ea{\end{align}}
\def\bp{\begin{proof}}
\def\ep{\end{proof}}
\def\Q{Q(S_+^n)}
\def\Y{Y(S^n)}

\def\s{\sigma}

\def\ubar{\bar{U}_{(x_0,\e)}}
\def\ubarrho{\bar{U}_{(x_0,\e)}}
\def\func:u{\bar{u}_{(x_0,\e)}}

\def\U{U_{\epsilon}}

\def\fp{\frac{\d}{\d t}}
\def\im{\int_M}
\def\idm{\int_{\partial M}}
\def\ct{\frac{4(n-1)}{n-2}}
\def\lp{\left(}
\def\rp{\right)}

\begin{document}

\title{Convergence of the Yamabe flow on manifolds with minimal boundary}
\author{\textsc{S\'ergio Almaraz\footnote{Supported by CNPq/Brazil grants 308231/2013-9, 471508/2013-6 and 467138/2014-1.} and  Liming Sun}}


\maketitle
\begin{abstract}
We study the Yamabe flow on compact Riemannian manifolds of dimensions greater than two with minimal boundary. Convergence to a metric with constant scalar curvature and minimal boundary is established in dimensions up to seven, and in any dimensions if the manifold is spin.

MSC classes: 53C21, 53C25.
\end{abstract}

\tableofcontents

\section{Introduction}\label{sec:intr}
Let $M^n$ be a closed manifold with dimension $n\geq 3$. In order to solve the Yamabe problem (see \cite{yamabe}), R. Hamilton introduced the Yamabe flow, which evolves Riemannian metrics on $M$ according to the equation
$$
\frac{\d}{\d t}g(t)=-(R_{g(t)}-\overline{R}_{g(t)})g(t)\,,
$$ 
where $R_g$ denotes the scalar curvature of the metric $g$ and $\overline{R}_g$ stands for the average 
$$
\displaystyle{\left(\int_M \dv_g\right)^{-1}}\int_M R_g\dv_g.
$$ 
Here, $\dv_g$ is the volume form of $(M,g)$. Although the Yamabe problem was solved using a different approach in \cite{aubin1,schoen1,trudinger}, the Yamabe flow is a natural geometric deformation to metrics of constant scalar curvature. The convergence of the Yamabe flow on closed manifolds was studied in \cite{chow, struwe-flow, ye}. This question was  solved in \cite{brendle-flow, brendle-invent} under the hypotheses of the positive mass theorem.

In this work, we study the convergence of the Yamabe flow on compact $n$-dimensional manifolds with boundary, when $n\geq 3$. For those manifolds, J. Escobar raised the question of existence of conformal metrics with constant scalar curvature which have the boundary as a minimal hypersurface. This problem was studied in \cite{brendle-chen,escobar2, mayer-ndiaye}; see also \cite{ambrosetti-li-malchiodi,han-li}. 
 
Let $(M^n,g_0)$ be a compact Riemannian manifold with boundary $\d M$ and dimension $n\geq 3$.
We consider the following conformal invariant defined in \cite{escobar2}:
\ba
Q(M)&=\inf_{g\in [g_0]}
\frac{\int_M R_g\dv_g+2\int_{\d M}\cmedia_g\ds_g}{\left(\int_{M}\dv_g\right)^{\frac{n-2}{n}}}\notag
\\
&=\inf_{\{\:u\in C^1(\bar{M}), u\nequiv 0\}}
\frac{\int_M \left(\frac{4(n-1)}{n-2}|d u|_{g_0}^2+R_{g_0}u^2\right)dv_{g_0}+\int_{\d M}2\cmedia_{g_0}u^2d\sigma_{g_0}}
{\left(\int_{M}|u|^\frac{2n}{n-2}dv_{g_0}\right)^{\frac{n-2}{n}}}\,,\notag
\end{align}
where $\cmedia_g$ and  $\ds_g$ denote respectively the trace of the 2nd fundamental form and the volume form of $\d M$, with respect to the metric $g$, and $[g_0]$ stands for the conformal class of the metric $g_0$.

We are interested in a formulation of the Yamabe flow for compact manifolds with minimal boundary proposed by S. Brendle in \cite{brendle-boundary}. This flow evolves a conformal family of metrics $g(t)$, $t\geq 0$, according to the equations
\begin{equation}\label{eq:evol}
\begin{cases}
\displaystyle\frac{\d}{\d t}g(t)=-(\cesc-\cescbar)g(t)\,,&\text{in}\:M\,,
\\
\cm=0\,,&\text{on}\:\d M\,.
\end{cases}
\end{equation}
\begin{theorem}[\cite{brendle-boundary}]\label{brendle:boundary:thm}
Suppose that:

(i) $Q(M)\leq 0$, \: or

(ii) $Q(M)>0$ and $M$ is locally conformally flat with umbilic boundary.

Then, for every initial metric $g(0)$ on M with minimal boundary, the flow \eqref{eq:evol} exists for all time $t\geq 0$ and converges to a  constant scalar curvature metric with minimal boundary.
\end{theorem}

Inspired by the ideas in \cite{brendle-flow, brendle-invent}, we handle the remaining cases of this problem.
Define
\begin{equation*}
\mathcal{Z}_M=
\{x_0\in M\backslash\d M\,;\:
\limsup_{x\to x_0}d_{g_0}(x,x_0)^{2-d}|W_{g_0}(x)|=0\}\,,
\end{equation*}
\begin{equation*}
\mathcal{Z}_{\d M}=
\{x_0\in\d M\,;\:
\limsup_{x\to x_0}d_{g_0}(x,x_0)^{2-d}|W_{g_0}(x)|
=\limsup_{x\to x_0}d_{g_0}(x,x_0)^{1-d}|\pi_{g_0}(x)|=0\}\,,
\end{equation*}
\begin{equation*}
\text{and}\quad\mathcal{Z}=\mathcal{Z}_{M}\cup \mathcal{Z}_{\d M}\,,
\end{equation*}
where $W_{g_0}$ denotes the Weyl tensor of $M$, $\pi_{g_0}$ the trace-free second fundamental form of $\d M$, and $d=\Big{[}\frac{n-2}{2}\Big{]}$.
Our first result is the following:
\begin{theorem}\label{first:thm}
Suppose that $(M^n,g_0)$ is not conformally diffeomorphic to the hemisphere $S_+^n$ and satisfies $Q(M)>0$. 
If 

(a) $\mathcal{Z}= \emptyset$, \: or

(b) $n\leq 7$, \: or

(c) $M$ is spin,
\\
then, for any initial metric $g(0)$ on M with minimal boundary, the flow \eqref{eq:evol} exists for all time $t\geq 0$ and converges to a  metric with constant scalar curvature and minimal boundary.
\end{theorem}

Since the round sphere $S^n$ minus a point is diffeomorphic to $\R^n$, which is spin, the following is an immediate consequence of Theorems \ref{brendle:boundary:thm} and \ref{first:thm}:

\begin{corollary}

If $M\subset S^n$ is a compact domain with smooth boundary, then the flow \eqref{eq:evol}, starting with any metric with minimal boundary, exists for all time $t\geq 0$ and converges to a metric with constant scalar curvature and minimal boundary.
\end{corollary}

Condition (a) in Theorem \ref{first:thm} is particularly satisfied if  the Weyl tensor and the trace-free second fundamental form are nonzero everywhere on $M\backslash \d M$ and $\d M$ respectively.
Conditions (b) and (c) allow us to make use  of the positive mass theorem in \cite{schoen2, schoen-yau, witten} and its corresponding version for manifolds with a non-compact boundary in \cite{almaraz-barbosa-lima}. 

Before stating our main result, from which Theorem \ref{first:thm} follows, we will briefly discuss those positive mass theorems.

\begin{definition}\label{def:asym}
Let $(N, g)$ be a Riemannian manifold with a (possibly empty) boundary $\d N$. 
We say that $N$ is {\it{asymptotically flat}} with order $p>0$, if there is a compact set $K\subset N$ and a diffeomorphism $f:N\backslash K\to \R^n\backslash \overline{B_1(0)}$ or $f:N\backslash K\to \Rn\backslash \overline{B^+_1(0)}$ such that, in the coordinate chart defined by $f$ (which we call the {\it  asymptotic coordinates} of $N$), we have
$$
|g_{ab}(y)-\delta_{ab}|+|y||g_{ab,c}(y)|+|y|^2|g_{ab,cd}(y)|=O(|y|^{-p})\,,
\:\:\:\:\text{as}\:\:|y|\to\infty\,,
$$
where $a,b,c,d=1,...,n$. Here, $\Rn=\{(y_1,...,y_n)\in\R^n\,;\:y_n\geq 0\}$, $\overline{B_1(0)}=\{y\in\R^n\,;\:|y|\leq 1\}$ and $\overline{B^+_1(0)}=\overline{B_1(0)}\cap \Rn$.
\end{definition}

Suppose the manifold $N^n$, with dimension $n\geq 3$,  is asymptotically flat with order $p>\frac{n-2}{2}$, as defined  above. Assume also that $R_g$ is integrable on $N$, and $\cmedia_g$  is integrable on $\d N$ if $\d N$ is noncompact. Let $(y_1,...,y_n)$ be the  asymptotic coordinates induced by the diffeomorphism $f$. 

If $f$ takes values in $\R^n\backslash \overline{B_1(0)}$ then $\d N$ is compact (or empty) and the limit
\begin{equation}\label{def:ADM:mass}
m_{ADM}(g):=\lim_{R\to\infty}\sum_{a,b=1}^{n}\int_{y\in\R^n,\, |y|=R}(g_{ab,b}-g_{bb,a})\frac{y_a}{|y|}\,\ds_{R}\notag
\end{equation}
exists and is called the {\it ADM mass} of $(N, g)$. Moreover, $m_{ADM}(g)$ is a geometric invariant in the sense that it does not depend on the asymptotic coordinates; see \cite{bartnik}.
\begin{conjecture}[Positive mass]\label{pmc}
If $R_g, \cmedia_g\geq 0$, then we have $m_{ADM}(g)\geq 0$ and the equality holds if and only if $N$ is isometric to $\R^n$. In particular, $\d N=\emptyset$ when the equality holds.
\end{conjecture}
As a consequence of \cite{schoen2, schoen-yau, witten} we have:
\begin{theorem}\label{pmt}
Conjecture \ref{pmc} holds true if $n\leq 7$ or if $N$ is spin.
\end{theorem}
The proof for $n\leq 7$ was obtained by Schoen and Yau in \cite{schoen2, schoen-yau}, and the one for spin manifolds by Witten in \cite{witten} when $M=\emptyset$. The boundary condition used in \cite{almaraz-barbosa-lima} can be used to extend Witten's result to the case $\d M\neq \emptyset$.

If $f$ takes values in $\Rn\backslash\overline{B^+_1(0)}$ then the limit
\begin{align}\label{def:mass}
m(g):=
\lim_{R\to\infty}\left\{
\sum_{a,b=1}^{n}\int_{y\in\Rn,\, |y|=R}(g_{ab,b}-g_{bb,a})\frac{y_a}{|y|}\,\ds_{R}
+\sum_{i=1}^{n-1}\int_{y\in\d\Rn,\, |y|=R}g_{ni}\frac{y_i}{|y|}\,\ds_{R}\right\}
\end{align}
exists, and we call it the {\it mass} of $(M, g)$. Moreover, $m(g)$ is a geometric invariant in the sense that it does not depend on the asymptotic coordinates; see \cite{almaraz-barbosa-lima}.
\begin{conjecture}[Positive mass with a noncompact boundary]\label{pmcb}
If $R_g$, $\cmedia_g\geq 0$, then we have $m(g)\geq 0$ and the equality holds if and only if $N$ is isometric to $\Rn$.
\end{conjecture}
In \cite{almaraz-barbosa-lima}, this conjecture is reduced to Conjecture \ref{pmc},  so we have the following result:
\begin{theorem}
Conjecture \ref{pmcb} holds true if $n\leq 7$ or if $N$ is spin.
\end{theorem}

The asymptotically flat manifolds  used in this paper are obtained as the generalized stereographic projections of the compact Riemannian manifold $(M,g_0)$ with nonempty boundary. Those stereographic projections are performed around points $x_0\in  M$ by means of Green functions $G_{x_0}$, with singularity at $x_0$. After choosing a new background metric $g_{x_0}\in [g_0]$ with better coordinates expansion around $x_0$ (see Section \ref{sec:testfunc}), we consider the asymptotically flat manifold $(M\backslash \{x_0\}, \bar{g}_{x_0})$, where $\bar{g}_{x_0}= G_{x_0}^{\frac{4}{n-2}}g_{x_0}$ satisfies $R_{\bar {g}_{x_0}}\equiv 0$ and $\cmedia_{\bar {g}_{x_0}}\equiv 0$. If $x_0\in \mathcal{Z}_{\d M}$, according to Proposition \ref{propo19} below, this manifold has asymptotic order $p>\frac{n-2}{2}$, so Conjecture \ref{pmcb} claims that $m(\bar{g}_{x_0})>0$ unless $M$ is conformally equivalent to the unit hemisphere.  If $x_0\in \mathcal{Z}_{M}$, this manifold has asymptotic order $p>\frac{n-2}{2}$ (see \cite[Proposition 19]{brendle-invent}), so Conjecture \ref{pmc} claims that $m_{ADM}(\bar{g}_{x_0})>0$.

Our main result, which implies Theorem \ref{first:thm}, is the following:

\begin{theorem}\label{main:thm}
Suppose that $(M^n,g_0)$ is not conformally diffeomorphic to the unit hemisphere $S_+^n$ and satisfies $Q(M)>0$.
Assume that $m_{ADM}(\bar{g}_{x_0})>0$ for all $x_0\in \mathcal{Z}_{M}$ and $m(\bar{g}_{x_0})>0$ for all $x_0\in \mathcal{Z}_{\d M}$. Then, for any initial metric $g(0)$ with minimal boundary, the flow \eqref{eq:evol} exists for all $t\geq 0$ and converges to a constant scalar curvature metric with minimal boundary.
\end{theorem}

The proof of Theorem \ref{main:thm} follows the arguments in \cite{brendle-flow}; see also \cite{almaraz5}. An essential step is the construction of a family of test functions around each point $x_0\in M$, whose energies are uniformly bounded by the Yamabe quotient $Y(S^n)$ if $x_0\in M\backslash \d M$, and by $\Q$ if $x_0\in \d M$.  If $x_0\in M\backslash \d M$, the test functions used are essentially the ones introduced by S. Brendle in \cite{brendle-invent} for the case of closed manifolds. If $x_0\in \d M$, the functions used here were obtained in \cite{brendle-chen} in the case of umbilic boundary, where the authors address the existence of solutions to the Yamabe problem for manifolds with boundary. In this paper, however, we estimate their energies without any assumption on the boundary.

An additional difficulty in controlling the energy of interior test functions by $\Y$ arises when their centers get close to the boundary (see Subsection \ref{sub:sec:deftestfunct:tubular}). In this case, the techniques in \cite{brendle-invent} cannot be directly adapted because the standard (and symmetric) bubble in $\R^n$, which represents the sphere metric and is essential in the construction of the test functions, does not satisfy the Neumann boundary condition unless it is centered on $\d\Rn$. However, here we are able to exploit the sign of this Neumann derivative, when centered in $\Rn\backslash \d\Rn$, to obtain the necessary estimates.

This paper is organized as follows. In Section \ref{sec:prelim}, we establish some preliminaries and prove the long-time existence of the flow. In Section
\ref{sec:testfunc}, we construct the necessary test functions and estimate their energy. In Section \ref{sec:blowup}, we make use of the decomposition theorem in \cite{pierotti-terracini} to carry out a blow-up analysis using the test functions. In Section \ref{sec:mainthm}, first we use the blow-up analysis to prove a result which is analogous to Proposition 3.3 of \cite{brendle-flow}. Then we use it to prove our main theorem by estimating the solution to the flow uniformly in $t\geq 0$.  

\bigskip
\noindent
{\bf{Acknowledgments.}}
The first author is grateful to the Princeton University Mathematics Department, where this work began during his short visit in 2015, and the hospitality of  Professor F. Marques. The second author would like to thank Professor YanYan Li for his continuous support, encouragement and motivation.
Both authors thank the anonymous referee for the thorough review and highly appreciate his/her comments and suggestions.

\section{Preliminary results and long-time existence}\label{sec:prelim}
\begin{notation}
In the rest of this paper, $M^n$ will denote a compact manifold of dimension $n\geq 3$ with boundary $\d M$,  and $g_0$ will denote a background Riemannian metric on $M$. 
We will denote by $B_r(x)$ the metric ball in $M$  of radius $r$ with center $x\in M$ (observe  that $B_r(x)$ intersects $\d M$ when $g_{g_0}(x,\d M)<r$). 

For any Riemannian metric $g$ on $M$, $\eta_g$ will denote the inward unit normal vector to $\d M$ with respect to $g$ and $\Delta_g$ the Laplace-Beltrami operator.

If $z_0\in \Rn$, we set $B_r^+(z_0)=\{z\in\Rn\,;\:|z-z_0|< r\}$,
$$
D_{r}(z_0)=B^+_{r}(z_0)\cap \d\Rn\,,
\:\:\:\:\text{and}\:\:\:\:\:
\d^+B^+_{r}(z_0)=\d B^+_{r}(z_0)\cap\Rn\,.
$$

Finally, for any $z=(z_1,..., z_n)\in \R^n$ we set $\bar{z}=(z_1,...,z_{n-1},0)\in \d\Rn\cong\R^{n-1}$. 
\end{notation} 
\begin{convention}
We assume that $(M,g_0)$ satisfies $Q(M)>0$. According to \cite[Lemma 1.1]{escobar2}, we can also assume that $R_{g_0}>0$ and $\cmz\equiv 0$, after a conformal change of the metric. Multiplying $g_0$ by a positive constant, we can suppose that $\int_{M}\dv_{g_0}=1$.
We will adopt the summation convention whenever confusion is not possible, and use indices $a,b,c,d=1,...,n$, and $i,j,k,l=1,...,n-1$.
\end{convention}

If $g=u^{\frac{4}{n-2}}g_0$ for some  positive smooth function $u$ on $M$, we know that 
\begin{equation}\label{eq:R:H}
\begin{cases}
\displaystyle R_g=u^{-\frac{n+2}{n-2}}\left(-\frac{4(n-1)}{n-2}\Delta_{g_0}u+R_{g_0}u\right)\,,&\text{in}\:M\,,
\\
\displaystyle \cmedia_g=u^{-\frac{n}{n-2}}\left(-\frac{2(n-1)}{n-2}\frac{\d}{\d \eta_{g_0}}u+\cmedia_{g_0} u\right)\,,&\text{on}\:\d M\,,
\end{cases}
\end{equation}
and the operators $L_{g}=\Delta_{g}-\frac{n-2}{4(n-1)}R_{g}$ and $B_g=\frac{\d}{\d \eta_{g}}-\frac{n-2}{2(n-1)}\cmedia_{g}$ satisfy
\begin{equation}\label{propr:L}
L_{u^{\frac{4}{n-2}}g_0}(u^{-1}\zeta)=u^{-\frac{n+2}{n-2}}L_{g_0}\zeta,
\end{equation}
\begin{equation}\label{propr:B}
B_{u^{\frac{4}{n-2}}g_0}(u^{-1}\zeta)=u^{-\frac{n}{n-2}}B_{g_0}\zeta\,,
\end{equation}
for any smooth function $\zeta$.

If $u(t)=u(\cdot,t)$ is a 1-parameter family of positive smooth functions on $M$ and $g(t)=u(t)^{\frac{4}{n-2}}g_0$ with $H_{g_0}\equiv 0$, then (\ref{eq:evol}) can be written as 
\begin{equation}\label{eq:evol:u}
\begin{cases}
 \displaystyle \fp u(t)=-\frac{n-2}{4}(\cesc-\cescbar)\,u(t),&\text{in}\:M,\\
 \displaystyle\frac{\d }{\d \eta_{g_0}}u(t)=0 \,,&\text{on}\:\d M.
\end{cases}
\end{equation}
The first equation of (\ref{eq:evol:u}) can also be written as 
$$
\fp u(t)^{\frac{n+2}{n-2}}=\frac{n+2}{4}\left(\frac{4(n-1)}{n-2}\Delta_{g_0}u-\cez u+\cescbar u^{\frac{n+2}{n-2}}\right).
$$

Short-time existence of solutions to the equations (\ref{eq:evol:u}) can be obtained by standard theory for quasilinear parabolic equations. Hence, the equations (\ref{eq:evol:u}) have a solution $u(t)$ defined for all $t$ in the maximal interval $[0, T_{max})$.

Taking $\d/\d\eta_{g_0}$ on both sides of the first equation of (\ref{eq:evol:u}) and using the second one, one gets $\d \cesc/\d \eta_{g_0}=0$ on $\d M$.
Hence the scalar curvature has evolution equations
\begin{equation}\label{eq:evol:R}
 \begin{cases}
  \displaystyle\frac{\d}{\d t}\cesc=(n-1)\D_{g(t)}\cesc+(\cesc-\cescbar)\cesc\,,&\text{in}\:M\,,\\
  \displaystyle \frac{\d}{\d \eta_{g_{(t)}}}\cesc=0\,,&\text{on}\:\d M\,,
 \end{cases}
\end{equation}
where the first equation comes from the well known first variation formula of scalar scalar curvature.

Observe that for all $t\geq 0$ we have
\begin{equation}\label{eq:evol:vol}
\fp \dvt=-\frac{n}{2}(\cesc-\cescbar)\,\dvt
\end{equation}
and 
\begin{equation}\label{eq:evol:Rbar}
 \fp \cescbar=-\frac{n-2}{2}\int_M (\cesc-\cescbar)^2dv_{g(t)}.
\end{equation}
In particular, $\cescbar$ is decreasing and one can easily derive that (\ref{eq:evol}) preserves the volume which we can normalize to 
$$
\int_M dv_{g(t)}=1,\quad \text{for all }t\in [0,T_{max}).
$$ 
So, $\cescbar\geq Q(M)>0$ for all $t\geq 0$. 
\begin{proposition}\label{mp}
We have $ \cesc\geq \min\,\{\inf_M R_{{g(0)}},0\}$, for all $t\in[0,T_{max}) $.
\end{proposition}
\bp
Following ($\ref{eq:evol:R}$), this is an application of maximum principle.
\ep
\begin{proposition}\label{Propo2.4}
 For each $T\in (0, T_{max})$, there exist $C(T),c(T)>0$ such that
\begin{equation}\label{Propo2.4:1}
\sup_M u(t)\leq C(T)\:\:\:\:\text{and}\:\:\:\:\inf_M u(t)\geq c(T),\:\:\:\:\text{for all}\:t\in[0,T].
\end{equation}
In particular, $T_{max}=\infty$.
\end{proposition}
\begin{proof}
Set $\sigma=1-\min\,\{\inf_M R_{{g(0)}},0\}=\max\{\sup_M(1-R_{g(0)}),1\}$. Then, by Proposition \ref{mp}, $\cesc+\sigma\geq 1$ for all $t\in [0,T_{max})$.
It follows from \eqref{eq:evol:u} and \eqref{eq:evol:Rbar} that
 $$\fp \log u(t)=\frac{n-2}{4}(\cescbar-\cesc)\leq \frac{n-2}{4}(\overline{R}_{g(0)}+\sigma).$$
 Then there exists $C(T)>0$ such that $\sup_M u(t)\leq C(T)$ for all $t\in [0,T]$.
 
 Defining $P=R_{g_0}+\sigma\left(\sup_{0\leq t\leq T}\sup_M u(t)\right)^{\frac{4}{n-2}}$ we obtain
 \begin{align*}
  -\frac{4(n-1)}{n-2}\Delta_{g_0}u(t)+Pu(t)\geq -\frac{4(n-1)}{n-2}\Delta_{g_0}u(t)+R_{g_0}u(t)+\sigma u(t)^{\frac{n+2}{n-2}}
 =(R_{g(t)}+\sigma)u(t)^{\frac{n+2}{n-2}}\geq 0
 \end{align*}
for all $0\leq t\leq T$. Then it follows from Proposition \ref{CorolA.3} in the Appendix that
$$
\inf_M\,u(t)\left(\sup_M\,u(t)\right)^{\frac{n+2}{n-2}}\geq c(T)\int_Mu(t)^{\frac{2n}{n-2}}dv_{g_0}=c(T),
$$
by our volume normalization. This proves the second equation of (\ref{Propo2.4:1}).

Now we can follow \cite[Proposition 2.6]{brendle-flow} to prove that if $0<\a<\min\{4/n,1\}$ then there is $\tilde{C}(T)$ such that 
$$
|u(x_1,t_1)-u(x_2,t_2)|\leq \tilde{C}(T)\big((t_1-t_2)^{\a/2}+d_{g_0}(x_1,x_2)^{\a}\big)
$$
for all $x_1,x_2\in M$ and $t_1,t_2\in[0,T]$ satisfying $0<t_1-t_2<1$.
Then standard regularity theory for parabolic equations can be used to prove that all higher order derivatives of $u$ are uniformly bounded on every fixed interval $[0,T]$. This implies the long-time existence of $u$.
\end{proof}

Set
\begin{equation}\label{eq:def:cminfbar}
 \cminfbar=\lim_{t\to\infty}\cescbar>0.
\end{equation}

Because $\d\cesc/\d\eta_{g(t)}=0$ holds on $\d M$, we can follow the proof of Corollary 3.2 in \cite{brendle-flow} line by line, making use of \eqref{eq:evol:R}, \eqref{eq:evol:vol} and \eqref{eq:evol:Rbar}, to obtain
\begin{corollary}\label{Corol3.2}
For any $1<p<n/2+1$ we have 
$$\lim_{t\to\infty}\int_{M}|\cesc-\cminfbar|^p\dv_{g(t)}=0\,.$$
\end{corollary}


\section{The test functions}\label{sec:testfunc}
In this section, we construct the test functions  to be used in the blow-up analysis of Section \ref{sec:blowup}. Those functions are perturbations of the symmetric functions $U_{\e}$ (see \eqref{eq:def:U} below), which represent the spherical metric on $\R^n$ and have maximum at the origin.

We will make use of the following coordinate systems:
\begin{definition}\label{def:fermi}
Fix $x_0\in\d M$ and geodesic normal coordinates for $\d M$ centered at $x_0$. 
Let $(y_1,...,y_{n-1})$ be the coordinates of  $x\in\d M$ and $\eta(x)$ be the inward unit vector normal to $\d M$ at $x$. 
For small $y_n\geq 0$, the point $\exp_{x}(y_n\eta(x))\in M$ is said to have {\it{Fermi coordinates}} $(y_1,...,y_n)$ (centered at $x_0$). 
\end{definition}
\begin{definition}\label{def:normal}
Let $g$ be any (smooth)  Riemannian metric on $M$.
Consider $\tilde M$ the double of $M$ along its boundary and extend $g$ to a (smooth) Riemannian metric $\tilde g$ on $\tilde M$. Fix $x_0\in M$ and let $\tilde{\psi}_{x_0}:B_{r}(0)\subset \R^n\to \tilde M$
be normal coordinates (with respect to $\tilde g$) centered at $x_0$. If $\tilde{B}_{x_0,r}=\tilde{\psi}_{x_0}^{-1}(\tilde{\psi}_{x_0}(B_{r}(0))\cap M)$, we define the {\it{extended normal coordinates}} (centered at $x_0$) 
$$
\psi_{x_0}:\tilde{B}_{x_0,r}\subset \R^n\to M
$$
as the restriction of $\tilde\psi_{x_0}$ to $\tilde{B}_{x_0,r}$.
\end{definition}
Observe that this definition depends on the metric $\tilde g$ chosen, but this does not harm our arguments in this section because we can fix the extension to $\tilde M$ of the background metric $g_0$.
\begin{convention}
We will refer to extended normal coordinates as {\it{normal coordinates}} for short.
\end{convention}
\begin{notation}
We set $\tilde{D}_{x_0,r}=\psi_{x_0}^{-1}(\psi_{x_0}(\tilde{B}_{x_0,r})\cap\d M)$ and
$\d^+\tilde{B}_{x_0,r}=\d \tilde{B}_{x_0,r}\backslash \tilde{D}_{x_0,r}\subset \d B_{r}(0)$.
\end{notation}

Set  $M_{t}=\{x\in M\,;\:\:d_{g_0}(x,\d M)\leq t\}$ and let $\delta_0>0$ be a small constant to be chosen later  (see Remark \ref{choosing:bubbles} below). In the next subsections we will define three types of test functions:
\begin{itemize}
\item
{\bf{Type A}} test functions ($\bar u_{A;(x_0,\e)}$): defined in Subsection \ref{sub:sec:deftestfunct:boundary} using Fermi coordinates centered at any $x_0\in \d M$ and with energy to be controlled by $\Q$.
\item
{\bf{Type B}} test functions ($\bar u_{B;(x_0,\e)}$): defined in Subsection \ref{sub:sec:deftestfunct:tubular} using normal coordinates centered at any $x_0\in M_{2\delta_0}\backslash \d M$ and with energy to be controlled by $\Y$.
\item
{\bf{Type C}} test functions ($\bar u_{C;(x_0,\e)}$): defined in Subsection \ref{sub:sec:deftestfunct:int} using normal coordinates centered at any $x_0\in M\backslash M_{\delta_0}$ and with energy to be controlled by $\Y$. 
\end{itemize}
We fix $P_0=P_0(M,g_0)>0$ small such that (extended) normal coordinates with center $x_0$ are defined in $\tilde{B}_{x_0, 2P_0}$ for all $x_0\in M\backslash \d M$, and Fermi coordinates with center at $x_0$ are defined in $B^+_{2P_0}(0)$ for all $x_0\in \d M$.
\begin{convention} 
In this section, we will use the normalization $\cminfbar=4n(n-1)$, without loss of generality.
\end{convention}


\subsection{The auxiliary function $\phi$ and some algebraic preliminaries}\label{subsec:algebraic}

Firstly we fix some notations. If $\e>0$, we define
\begin{equation}\label{eq:def:U}
\U(y)=\left(\frac{\e}{\e^2+|y|^2}\right)^{\frac{n-2}{2}}
\:\:\:\:\text{for}\:\:y\in\R^n\,.
\end{equation}
It is well known that $\U$ satisfies
\begin{align}\label{eq:Ue}
\begin{cases}
\Delta\U+n(n-2)\U^{\frac{n+2}{n-2}}=0\,,&\text{in}\:\Rn\,,
\\
\d_n\U=0\,,&\text{on}\:\d\Rn\,,
\end{cases}
\end{align}
and
\begin{equation}\label{eq:U:Q}
4n(n-1)\left(\int_{\Rn}\U(y)^{\crit}dy\right)^{\frac{2}{n}}=\Q\,.
\end{equation}

In this subsection, $\mathcal{H}$ will denote a symmetric trace-free 2-tensor on $\Rn$ with components $\mathcal{H}_{ab}$, $a,b=1,...,n$, satisfying
\begin{equation}\label{propr:H}
\begin{cases}
\mathcal{H}_{ab}(0)=0\,,&\text{for}\: a,b=1,...,n\,,
\\
\mathcal{H}_{an}(x)=0\,,&\text{for}\:x\in\Rn,\: a=1,...,n\,,
\\
\d_k\mathcal{H}_{ij}(0)=0\,,&\text{for}\: i,j,k=1,...,n-1\,,
\\
\sum_{j=1}^{n-1}x_j\mathcal{H}_{ij}(x)=0\,,&\text{for}\:x\in\d\Rn,\: i=1,...,n-1\,.
\end{cases}
\end{equation}
We will also assume that those components are of the form
\begin{equation}\label{forma:H}
\mathcal{H}_{ab}(x)=\sum_{1\leq |\a|\leq d}h_{ab,\a}x^{\a}
\:\:\:\:\:\:\text{for}\:x\in\Rn\,,
\end{equation}
where $d=\left[\frac{n-2}{2}\right]$ and each $\a$ stands for a multi-index. Obviously, the constants $h_{ab,\a}\in\R$ satisfy $h_{an,\a}=0$ for any $\a$, and $h_{ab,\a}=0$ for any $\a\neq (0,...,0,1)$ with $|\a|=1$, where  $a,b=1,...,n$.

Let $\chi:\R\to\R$ be a non-negative smooth function such that
$\chi|_{[0,4/3]}\equiv 1$ and $\chi|_{[5/3,\infty)}\equiv 0$. If $\rho>0$, we define 
\begin{equation}\label{def:eta}
\chi_{\rho}(x)=\chi\left(\frac{|x|}{\rho}\right)
\:\:\:\:\text{for}\:x\in \R^n\,.
\end{equation} 
Notice that $\d_n\chi_{\rho}=0$ on $\d\Rn$.

Let $V=V(\e, \rho, \mathcal{H})$ be the smooth vector field on $\Rn$ obtained in  \cite[Theorem A.4]{brendle-chen}, which satisfies
\begin{align}\label{eq:V}
\begin{cases}
\sum_{b=1}^{n}\d_b\left\{\U^{\crit}(\chi_{\rho}\mathcal{H}_{ab}-\d_aV_b-\d_bV_a+\frac{2}{n}(\text{div} V)\delta_{ab})\right\}=0\,,&\text{in}\:\Rn\,,
\\
\d_nV_i=V_n=0\,,&\text{on}\:\d\Rn\,,
\end{cases}
\end{align}
for $a=1,...,n$, and $i=1,...,n-1$, and 
\begin{equation}\label{est:V}
|\d^{\b}V(x)|\leq C(n,|\b|)\sum_{i,j=1}^{n-1}\sum_{|\a|=1}^{d}|h_{ij,\a}|(\e+|x|)^{|\a|+1-|\b|}
\end{equation} 
for any multi-index $\b$. Here $\delta_{ab}=1$ if $a=b$ and $\delta_{ab}=0$ if $a\neq b$.

We define symmetric trace-free 2-tensors $S$ and $T$ on $\Rn$ by 
\begin{equation}\label{def:S:T}
S_{ab}=\d_aV_b+\d_bV_a-\frac{2}{n}\d_cV_c\delta_{ab}\quad\quad\text{and}\quad\quad T=\mathcal{H}-S\,.
\end{equation}
(Recall that we are adopting the summation convention.) Observe that $T_{in}=S_{in}=0$ on $\d\Rn$ for $i=1,...,n-1$.
It follows from (\ref{eq:V}) that $T$ satisfies
\begin{equation}\label{eq:U:T:1}
\U\d_bT_{ab}+\frac{2n}{n-2}\d_b\U T_{ab}=0\,,
\:\:\:\:\text{in}\:\:B^+_{\rho}(0)\,,
\:\:\:\:\text{for}\:\:a=1,...,n\,.
\end{equation}
In particular,
\begin{equation}\label{eq:U:T:2}
\frac{n-2}{4(n-1)}\U\d_a\d_bT_{ab}+\d_a(\d_b\U T_{ab})=0\,,
\:\:\:\:\text{in}\:\:B^+_{\rho}(0)\,,
\end{equation}
where we have used 
$\U\d_a\d_b\U-\frac{n}{n-2}\d_a\U\d_b\U=\frac{1}{n}(\U\Delta\U-\frac{n}{n-2}|d\U|^2)\delta_{ab}$ in $\Rn$ for all $a,b=1,...,n$.

Next we define the auxiliary function $\phi=\phi_{\e,\rho,\mathcal{H}}$ 
by
\begin{equation}\label{eq:def:phi}
\phi=\d_a\U V_a+\frac{n-2}{2n}\U\d_aV_a\,.
\end{equation}
By a direct computation, we have
\begin{align}\label{eq:phi}
\begin{cases}
\Delta \phi+n(n+2)\U^{\frac{4}{n-2}}\phi=\frac{n-2}{4(n-1)}\U \d_b\d_a \mathcal{H}_{ab}+\d_b(\d_a\U \mathcal{H}_{ab}),&\text{in}\:B^+_{\rho}(0)\,,
\\
\d_n \phi=0, &\text{on}\:\d\Rn\,.
\end{cases}
\end{align}
By the estimate (\ref{est:V}), $\phi$ satisfies 
\begin{equation}\label{est:phi}
|\phi(x)|\leq C\e^{\frac{n-2}{2}}\sum_{i,j=1}^{n-1}\sum_{|\a|=1}^{d}|h_{ij,\a}|(\e+|x|)^{|\a|+2-n}
\end{equation}
and 
\begin{equation}\label{est:lapl:phi}
\left|\Delta \phi(x)+n(n+2)\U^{\frac{4}{n-2}}\phi(x)\right|
\leq C\e^{\frac{n-2}{2}}\sum_{i,j=1}^{n-1}\sum_{|\a|=1}^{d}|h_{ij,\a}|(\e+|x|)^{|\a|-n}\,,
\end{equation}
for all $x\in\Rn$.

Observe that if $n=3$ then $d=0$, in which case $\mathcal{H}\equiv 0$ and $\phi\equiv 0$.

\begin{convention} 
In the rest of Subsection \ref{subsec:algebraic} we will assume that $n\geq 4$.
\end{convention}

We define algebraic Schouten tensor and algebraic Weyl tensor by   
$$A_{ac}=\d_c\d_e\mathcal{H}_{ae}+\d_a\d_e\mathcal{H}_{ce}-\d_e\d_e\mathcal{H}_{ac}-\frac{1}{n-1}\d_e\d_f\mathcal{H}_{ef}\delta_{ac}$$
and
\ba
Z_{abcd}=\d_b\d_d\mathcal{H}_{ac}-\d_b\d_c\mathcal{H}_{ad}+\d_a\d_c\mathcal{H}_{db}-\d_a\d_d\mathcal{H}_{bc}
+\frac{1}{n-2}\left(A_{ac}\delta_{bd}-A_{ad}\delta_{bc}
+A_{bd}\delta_{ac}-A_{bc}\delta_{db}\right)\,.\notag
\end{align}
We also set
\ba\label{eq:def:Q}
Q_{ab,c}=\U \d_cT_{ab}-\frac{2}{n-2}\d_a\U T_{bc}-\frac{2}{n-2}\d_b\U T_{ac}
+\frac{2}{n-2}\d_d\U T_{ad}\delta_{bc}+\frac{2}{n-2}\d_d\U T_{bd}\delta_{ac}\,.
\end{align}
\begin{lemma}\label{lemma1}
If the tensor $\mathcal{H}$ satisfies
\ba
\begin{cases}\notag
Z_{abcd}=0,&\text{in}\:\Rn\,,
\\
\d_n\mathcal{H}_{ij}=0,&\text{on}\:\d\Rn\,,
\end{cases}
\end{align}
then $\mathcal{H}=0$ in $\Rn$.
\end{lemma}
\bp
Observe that the hypothesis $\d_n\mathcal{H}_{ij}=0$ on $\d\Rn$ implies that $h_{ij,\a}=0$ for $\a=(0,...,0,1)$. In this case, the expression (\ref{forma:H}) can be written as 
$$
\mathcal{H}_{ab}(x)=\sum_{|\a|=2}^{d}h_{ab,\a}x^{\a}\,.
$$
Now the result is just Proposition 2.3 in \cite{brendle-chen}.
\ep
\begin{proposition}\label{propo2}
Set $U_r=B_{r/4}(0,...,0,\frac{3r}{2})\subset \Rn$. Then there exists $C=C(n)>0$ such that
$$
\sum_{i,j=1}^{n-1}\sum_{|\a|=1}^{d}|h_{ij,\a}|^2r^{2|\a|-4+n}
\leq C\int_{U_r}Z_{abcd}Z_{abcd}
+Cr^{-1}\int_{D_{\frac{5r}{3}}(0)\backslash D_{\frac{4r}{3}}(0)}\d_n \mathcal{H}_{ij}\d_n \mathcal{H}_{ij}\,,
$$
for all $r>0$.
\end{proposition}
\bp
If $r=1$,  observe that the square roots of both sides of the inequality are norms in $\mathcal{H}$, due to Lemma \ref{lemma1}. The general case follows by scaling. 
\ep
\begin{lemma}\label{lemma3}
There exists $C=C(n)>0$ such that
$$
\e^{n-2}r^{6-2n}\int_{U_r}Z_{abcd}Z_{abcd}
\leq 
\frac{C}{\theta}\int_{B^+_{2r}(0)\backslash B^+_{r}(0)}Q_{ab,c}Q_{ab,c}
+\theta\e^{n-2}\sum_{i,j=1}^{n-1}\sum_{|\a|=1}^{d}|h_{ij,\a}|^2r^{2|\a|+2-n}
$$
for all $0<\theta<1$ and all $r\geq \e$.
\end{lemma}
\bp
This follows from the third formula in the proof of Proposition 2.5 in \cite{brendle-chen}, by means of Young's inequality. Observe that, in our calculations, we are using the range $1\leq|\a|\leq d$ in the summation formulas, instead of the range $2\leq|\a|\leq d$ used in \cite{brendle-chen}. 
\ep
\begin{lemma}\label{lemma4}
There exists $C=C(n)>0$ such that
\ba
\e^{n-2}r^{5-2n}\int_{D_{\frac{5r}{3}}(0)\backslash D_{\frac{4r}{3}}(0)}\d_n \mathcal{H}_{ij}\d_n \mathcal{H}_{ij}
\leq 
C\theta\e^{n-2}\sum_{i,j=1}^{n-1}\sum_{|\a|=1}^{d}|h_{ij,\a}|^2r^{2|\a|+2-n}
+\frac{C}{\theta}\int_{B^+_{2r}(0)\backslash B^+_{r}(0)}Q_{ij,n}Q_{ij,n}\notag
\end{align}
for all $0<\theta<1$ and all $r\geq \e$. 
\end{lemma}
\bp
Let $\chiup:\R\to\R$ be a non-negative smooth function such that
$\chiup(t)=1$ for $t\in [4/3,5/3]$ and $\chiup(t)=0$ for $t\notin [1,2]$. For $r>0$ and $x\in \Rn$ we define $\chiup_r(x)=\chiup (|x|/r)$.
Observe that 
$\d_nS_{ij}=-\frac{1}{n-1}\d_nS_{nn}\delta_{ij}$ on $\d\Rn$. On the other hand, \eqref{eq:U:T:1} gives $\d_nS_{nn}=-\d_nT_{nn}=0$. Thus, $\d_nS_{ij}=0$ and 
$\d_n\mathcal{H}_{ij}=\d_nT_{ij}=\U^{-1}Q_{ij,n}$ on $\d\Rn$.
Integration by parts gives
\ba\label{lemma4:2}
\int_{\d\Rn}\U^{\critbordo}\d_n\mathcal{H}_{ij}\d_n\mathcal{H}_{ij}\chiup_r
&=\int_{\d\Rn}\U^{\frac{2}{n-2}}Q_{ij,n}Q_{ij,n}\chiup_r
=-\int_{\Rn}\d_n\big{(}\U^{\frac{2}{n-2}}Q_{ij,n}Q_{ij,n}\chiup_r\big{)}
\\
&=-\int_{\Rn}\d_n(\U^{\frac{2}{n-2}}Q_{ij,n}\chiup_r)Q_{ij,n}-\int_{\Rn}\U^{\frac{2}{n-2}}\d_nQ_{ij,n}Q_{ij,n}\chiup_r\,.\notag
\end{align}
By using Young's inequality, the result now follows from the inequalities 
$$
\U^{\frac{2(n-1)}{n-2}}\d_n\mathcal{H}_{ij}\d_n\mathcal{H}_{ij}\chiup_r\geq C^{-1}\e^{n-1}r^{2-2n}\d_n\mathcal{H}_{ij}\d_n\mathcal{H}_{ij}\chiup_r
$$ 
and
$$
|\d_n(\U^{\frac{2}{n-2}}Q_{ij,n}\chiup_r)|+|\U^{\frac{2}{n-2}}\d_nQ_{ij,n}\chiup_r|\leq C\e^{\frac{n}{2}}\sum_{i,j=1}^{n-1}\sum_{|\a|=1}^{d}|h_{ij,\a}|r^{|\a|-2-n}\,.
$$
\ep
\begin{proposition}\label{propo5}
There exists $\l=\l(n)>0$ such that
$$
\l\e^{n-2}\sum_{i,j=1}^{n-1}\sum_{|\a|=1}^{d}|h_{ij,\a}|^2\int_{B^+_{\rho}(0)}(\e+|x|)^{2|\a|+2-2n}dx
\leq \frac{1}{4}\int_{B^+_{\rho}(0)}Q_{ab,c}Q_{ab,c}dx
$$
for all $\rho\geq 2\e$.
\end{proposition}
\bp
This follows from Proposition \ref{propo2}, Lemma \ref{lemma3}, and Lemma \ref{lemma4}.
\ep


\subsection{Type A test functions ($\bar u_{A;(x_0,\e)}$)}\label{sub:sec:deftestfunct:boundary}
In this subsection we use the same test functions as in \cite{brendle-chen} but we need to do some changes when estimating their energy by $\Q$ because the boundary does not need to be umbilical in our case.

For $\rho\in (0, P_0/2]$, the Fermi coordinates centered at $x_0\in \d M$ define a smooth map $\psi_{x_0}:B_{\rho}^+(0)\subset \Rn\to M$. We will sometimes omit the symbols $\psi_{x_0}$ in order to simplify our notations, identifying $\psi_{x_0}(x)\in M$ with $x\in B^+_{\rho}(0)$. In those coordinates, we have the properties $g_{ab}(0)=\delta_{ab}$ and $g_{nb}(x)=\delta_{nb}$, for any $x\in B^+_{\rho}(0)$ and $a,b=1,...,n$. If we write $g=\exp(h)$, where $\exp$ denotes the matrix exponential, then the symmetric 2-tensor $h$ satisfies the following properties:
\begin{equation}\notag
\begin{cases}
h_{ab}(0)=0\,,&\text{for}\: a,b=1,...,n\,,
\\
h_{an}(x)=0\,,&\text{for}\:x\in B^+_{\rho}(0),\: a=1,...,n\,,
\\
\d_kh_{ij}(0)=0\,,&\text{for}\: i,j,k=1,...,n-1\,,
\\
\sum_{j=1}^{n-1}x_jh_{ij}(x)=0\,,&\text{for}\:x\in D_{\rho}(0),\: i=1,...,n-1\,.
\end{cases}
\end{equation}  
The last two properties follow from the fact that Fermi coordinates are normal on the boundary.

According to  \cite[Proposition 3.1]{marques-weyl}, for each $x_0\in\d M$ we can find a conformal metric $g_{x_0}=f_{x_0}^{\frac{4}{n-2}}g_0$, with $f_{x_0}(x_0)=1$, and Fermi coordinates centered at $x_0$ such that $\text{det}(g_{x_0})(x)=1+O(|x|^{2d+2})$, where $d=\big[\frac{n-2}{2}\big]$. In particular, 
if we write $g_{x_0}=\exp(h_{x_0})$, we have $\text{tr}(h_{x_0})(x)=O(|x|^{2d+2})$. Moreover, $\cmedia_{g_{x_0}}$, the trace of the second fundamental form of $\d M$, satisfies 
\begin{equation}\label{est:H}
\cmedia_{g_{x_0}}(x)=-\frac{1}{2}g^{ij}\d_ng_{ij}(x)=-\frac{1}{2}\d_n(\log \text{det} (g_{x_0}))(x)=O(|x|^{2d+1})\,.
\end{equation}

Since $M$ is compact, we can assume that $1/2\leq f_{x_0}\leq 3/2$ for any $x_0\in \d M$, choosing $P_0$ smaller if necessary.
\begin{notation}
In order to simplify our notations, in the coordinates above, we will write $g_{ab}$ and $g^{ab}$ instead of $(g_{x_0})_{ab}$ and $(g_{x_0})^{ab}$ respectively, and $h_{ab}$ instead of $(h_{x_0})_{ab}$.
\end{notation}

In this subsection, we denote by 
\begin{equation}\notag 
\mathcal{H}_{ab}(x)=\sum_{1\leq |\a|\leq d}h_{ab,\a}x^{\a}
\end{equation}
the Taylor expansion of order $d$ associated with the function $h_{ab}(x)$. Thus, we have $h_{ab}(x)=\mathcal{H}_{ab}(x)+O(|x|^{d+1})$.  Observe that $\mathcal{H}$ is a symmetric trace-free 2-tensor on $\Rn$, which satisfies the properties (\ref{propr:H}) and has the form (\ref{forma:H}). 
Then we can use the function $\phi=\phi_{\e,\rho,\mathcal{H}}$ 
(see formula (\ref{eq:def:phi})) and the results obtained in Subsection \ref{subsec:algebraic}. 

Recall the definitions of $\U$ in (\ref{eq:def:U}), $\chi_{\rho}$ in (\ref{def:eta}), and $\cminfbar$ in (\ref{eq:def:cminfbar}). Define
\ba\label{def:test:func}
\ubar(x)
=&\left(\frac{4n(n-1)}{\cminfbar}\right)^{\frac{n-2}{4}}\chi_{\rho}(\psi_{x_0}^{-1}(x))\big(\U(\psi_{x_0}^{-1}(x))+\phi(\psi_{x_0}^{-1}(x))\big)
\\
&\hspace{0.1cm}+\left(\frac{4n(n-1)}{\cminfbar}\right)^{\frac{n-2}{4}}\e^{\frac{n-2}{2}}\big(1-\chi_{\rho}(\psi_{x_0}^{-1}(x))\big)G_{x_0}(x),\notag
\end{align}
for $x\in M$. Here, $G_{x_0}$ is the Green's function of the conformal Laplacian $L_{g_{x_0}}=\Delta_{g_{x_0}}-\frac{n-2}{4(n-1)}R_{g_{x_0}}$, with pole at $x_0\in\d M$, satisfying the boundary condition 
\begin{equation}\label{eq:G:bordo}
\frac{\d}{\d\eta_{g_{x_0}}}G_{x_0}-\frac{n-2}{2(n-1)}\cmedia_{g_{x_0}}G_{x_0}=0,
\end{equation}
on $\partial M \backslash\{x_0\}$, and the normalization $\lim_{|y|\to 0}|y|^{n-2}G_{x_0}(\psi_{x_0}(y))=1$. This function, obtained in Proposition \ref{green:point}, satisfies 
\ba\label{estim:G}
|G_{x_0}(\psi_{x_0}(y))-|y|^{2-n}|&\leq
C\sum_{i,j=1}^{n-1}\sum_{|\a|=1}^{d}|h_{ij,\a}||y|^{|\a|+2-n}+
\begin{cases}
C|y|^{d+3-n},\:\:\:\text{if}\:n\geq 5,
\\
C(1+|\log|y||),\:\:\:\text{if}\:n=3,4,
\end{cases}
\\
\left|\frac{\d}{\d y_b}(G_{x_0}(\psi_{x_0}(y))-|y|^{2-n})\right|
&\leq
C\sum_{i,j=1}^{n-1}\sum_{|\a|=1}^{d}|h_{ij,\a}||y|^{|\a|+1-n}+C|y|^{d+2-n}\,,\notag
\end{align}
for all $b=1,...,n$.

We define the test function 
\begin{equation}\label{def:test:func:u}
\bar u_{A;(x_0,\e)}=f_{x_0}\ubar\,.
\end{equation}
Observe that this function depends also on the radius $\rho$ above, which will be fixed later in Section \ref{sec:blowup}. Such constant will also be referred to as $\rho_A$ in order to avoid confusion with test functions of the other subsections. 

Our main result in this subsection is the following estimate for the energy of $\bar u_{A;(x_0,\e)}$:
\begin{proposition}\label{Propo:energy:test}
Under the hypotheses of Theorem \ref{main:thm}, there exists $P_1=P_1(M,g_0)>0$ such that
\ba
&\frac{\int_M\left\{\frac{4(n-1)}{n-2}|d \bar u_{A;(x_0,\e)}|_{g_0}^2+R_{g_0}\bar u_{A;(x_0,\e)}^2\right\}\dv_{g_0}}{\left(\int_{M}\bar u_{A;(x_0,\e)}^{\crit}\dv_{g_0}\right)^{\frac{n-2}{n}}}\notag
\\
&\hspace{0.5cm}=\frac{\int_M\left\{\frac{4(n-1)}{n-2}|d \ubar|_{g_{x_0}}^2
+R_{g_{x_0}}\ubar^2\right\}\dv_{g_{x_0}}
+\int_{\d M}2\cmedia_{g_{x_0}}\ubar^2\ds_{g_{x_0}}}{\left(\int_{M}\ubar^{\crit}\dv_{g_{x_0}}\right)^{\frac{n-2}{n}}}\notag
\\
&\hspace{0.5cm}\leq \Q\notag
\end{align}
for all $x_0\in\d M$ and $0<2\e<\rho_A<P_1$.
\end{proposition}

Let $\l$ be the constant obtained in Proposition \ref{propo5}.
\begin{proposition}\label{propo6}
There exist $C, P_1>0$, depending only on $(M, g_0)$,  such that
\ba
&\int_{B^+_{\rho}(0)}\left\{\frac{4(n-1)}{n-2}|d(\U+\phi)|_{g_{x_0}}^2+R_{g_{x_0}}(\U+\phi)^2\right\}dx
+\int_{D_{\rho}(0)}2\cmedia_{g_{x_0}}(\U+\phi)^2d\sigma\notag
\\
&\hspace{0.5cm}\leq
4n(n-1)\int_{B^+_{\rho}(0)}\U^{\frac{4}{n-2}}(\U^2+\frac{n+2}{n-2}\phi^2)dx
+\int_{\d^+B^+_{\rho}(0)}\left\{\frac{4(n-1)}{n-2}\U\d_a\U+\U^2\d_bh_{ab}-\d_b\U^2h_{ab}\right\}
\frac{x_a}{|x|}\ds_{\rho}\notag
\\
&\hspace{1.5cm}-\frac{\l}{2}\sum_{i,j=1}^{n-1}\sum_{|\a|=1}^{d}|h_{ij,\a}|^2\e^{n-2}\int_{B^+_{\rho}(0)}(\e+|x|)^{2|\a|+2-2n}dx
+C\sum_{i,j=1}^{n-1}\sum_{|\a|=1}^{d}|h_{ij,\a}|\e^{n-2}\rho^{|\a|+2-n}
+C\e^{n-2}\rho^{2d+4-n}\notag
\end{align}
for all $0<2\e\leq \rho\leq P_1$.
\end{proposition}
\bp  
Following the steps  in \cite[Proposition 3.6]{brendle-chen} we obtain
\ba
\int_{B^+_{\rho}(0)}&\left\{\frac{4(n-1)}{n-2}|d(\U+\phi)|_{g_{x_0}}^2+R_{g_{x_0}}(\U+\phi)^2\right\}dx
+\int_{D_{\rho}(0)}2\cmedia_{g_{x_0}}(\U+\phi)^2d\sigma\notag
\\
&\hspace{0.5cm}\leq
\int_{B^+_{\rho}(0)}\frac{4(n-1)}{n-2}|d\U|^2dx+\int_{B^+_{\rho}(0)}\frac{4(n-1)}{n-2}n(n+2)\U^{\frac{4}{n-2}}\phi^2dx\notag
\\
&\hspace{1cm}+\int_{\d^+B^+_{\rho}(0)}\left(\U^2\d_bh_{ab}-\d_b\U^2h_{ab}\right)
\frac{x_a}{|x|}\ds_{\rho}
-\frac{1}{4}\int_{B^+_{\rho}(0)}Q_{ab,c}Q_{ab,c}dx\notag
\\
&\hspace{1cm}+\frac{\l}{2}\sum_{i,j=1}^{n-1}\sum_{|\a|=1}^{d}|h_{ij,\a}|^2\e^{n-2}\int_{B^+_{\rho}(0)}(\e+|x|)^{2|\a|+2-2n}dx\notag
\\
&\hspace{1cm}+C\sum_{i,j=1}^{n-1}\sum_{|\a|=1}^{d}|h_{ij,\a}|\e^{n-2}\rho^{|\a|+2-n}
+C\e^{n-2}\rho^{2d+4-n}\,.\notag
\end{align}
The result follows  by making use of Proposition \ref{propo5} and
$$
|d\U|^2=\d_a(\U\d_a\U)-\U\Delta\U=\d_a(\U\d_a\U)+n(n-2)\U^{\crit}\,.
$$
\ep

As in \cite[p. 1006]{brendle-chen}, we define the flux integral
\ba\label{def:I}
\mathcal{I}(x_0,\rho)
=\frac{4(n-1)}{n-2}&\int_{\d^+B^+_{\rho}(0)}(|x|^{2-n}\d_aG_{x_0}-\d_a|x|^{2-n}G_{x_0})\frac{x_a}{|x|}\ds_{\rho}
\\
-&\int_{\d^+B^+_{\rho}(0)}|x|^{2-2n}(|x|^{2}\d_bh_{ab}-2nx_bh_{ab})\frac{x_a}{|x|}\ds_{\rho}\,,\notag
\end{align}
for $\rho>0$ sufficiently small. 
\begin{proposition}\label{propo7}
There exists $P_1=P_1(M, g_0)>0$ such that
\ba
&\int_M\left\{\frac{4(n-1)}{n-2}|d\ubar|_{g_{x_0}}^2+R_{g_{x_0}}\ubar^2\right\}\dv_{g_{x_0}}
+\int_{\d M}2\cmedia_{g_{x_0}}\ubar^2\ds_{g_{x_0}}\notag
\\
&\hspace{2cm}\leq
\Q\left\{\int_{M}\ubar^{\crit}\dv_{g_{x_0}}\right\}^{\frac{n-2}{n}}
-\e^{n-2}\mathcal{I}(x_0,\rho)\notag
\\
&\hspace{2.5cm}-\frac{\l}{4}\sum_{i,j=1}^{n-1}\sum_{|\a|=1}^{d}|h_{ij,\a}|^2\e^{n-2}
\int_{B^+_{\rho}(0)}(\e+|x|)^{2|\a|+2-2n}dx\notag
\\
&\hspace{2.5cm}+C\sum_{i,j=1}^{n-1}\sum_{|\a|=1}^{d}|h_{ij,\a}|\e^{n-2}\rho^{|\a|+2-n}
+C\e^{n-2}\rho^{2d+4-n}
+C\e^{n}\rho^{-n}\notag
\end{align}
for all $0<2\e\leq \rho\leq P_1$. 
\end{proposition}
\bp
As in \cite[Proposition 15]{brendle-invent}, we get
\ba\label{est:propo8}
&4n(n-1)\int_{B^+_{\rho}(0)}\U^{\frac{4}{n-2}}\left(\U^2+\frac{n+2}{n-2}\phi^2\right)dx
\\
&\leq \Q\left(\int_{B^+_{\rho}(0)}(\U+\phi)^{\crit}dx\right)^{\frac{n-2}{n}}
+\sum_{i,j=1}^{n-1}\sum_{|\a|=1}^{d}|h_{ij,\a}|\rho^{|\a|-n}\e^{n}+C\sum_{i,j=1}^{n-1}\sum_{|\a|=1}^{d}|h_{ij,\a}|^2\e^{n-1}\int_{B^+_{\rho}(0)}(\e+|x|)^{2|\a|+2-2n}dx\notag
\end{align}
for all $0<2\e\leq \rho\leq P_1$ and $P_1$ sufficiently small. Now, with Proposition \ref{propo6} at hand, our proof is analogous to the one in \cite[Proposition 4.1]{brendle-chen}
\ep
\begin{corollary}\label{corol8}
There exist $P_1,\: \theta,\: C_0>0$, depending only on $(M,g_0)$, such that
\ba
&\int_M\left\{\frac{4(n-1)}{n-2}|d\ubar|_{g_{x_0}}^2+R_{g_{x_0}}\ubar^2\right\}\dv_{g_{x_0}}
+\int_{\d M}2\cmedia_{g_{x_0}}\ubar^2\ds_{g_{x_0}}\notag
\\
&\hspace{2cm}\leq
\Q\left\{\int_{M}\ubar^{\crit}\dv_{g_{x_0}}\right\}^{\frac{n-2}{n}}
-\e^{n-2}\mathcal{I}(x_0,\rho)-\theta\e^{n-2}\int_{B^+_{\rho}(0)}|W_{g_0}(x)|^2(\e+|x|)^{6-2n}dx\notag
\\
&\hspace{2.5cm}-\theta\e^{n-2}\int_{D_{\rho}(0)}|\pi_{g_0}(x)|^2(\e+|x|)^{5-2n}\ds
+C_0\e^{n-2}\rho^{2d+4-n}
+C_0\left(\frac{\e}{\rho}\right)^{n-2}\frac{1}{|\log(\rho/\e)|}\notag
\end{align}
for all $0<2\e\leq \rho\leq P_1$. Here, we denote by $W_{g_0}$ the Weyl tensor of $(M,g_0)$ and by $\pi_{g_0}$ the trace-free 2nd fundamental form of $\d M$.
\end{corollary}
\bp
Similar to \cite[Corollary 3.10]{almaraz5}.
\ep
Recall that we denote by $\mathcal{Z}_{\d M}$ the set of all points $x_0\in\d M$ such that
$$
\limsup_{x\to x_0}d_{g_0}(x,x_0)^{2-d}|W_{g_0}(x)|
=\limsup_{x\to x_0}d_{g_0}(x,x_0)^{1-d}|\pi_{g_0}(x)|=0\,.
$$
\begin{proposition}\label{propo18}
The functions $\mathcal{I}(x_0,\rho)$ converge uniformly to a continuous function $I:\mathcal{Z}_{\d M}\to\R$ as $\rho\to 0$.
\end{proposition}
\bp
As in \cite[Proposition 3.11]{almaraz5} we can prove that 
\begin{equation*}
\sup_{x_0\in\mathcal{Z}_{\d M}}|\mathcal{I}(x_0,\rho)-\mathcal{I}(x_0,\tilde\rho)|\leq 
\begin{cases}
C\rho^{2d+4-n}&\text{if}\:n\geq 5,
\\
C\rho^{2d+4-n}|\log\rho|&\text{if}\:n=3,4,
\end{cases}
\end{equation*}
for all $0<\tilde\rho<\rho$. The result follows.
\ep

The following proposition, which is \cite[Proposition 3.12]{almaraz5} 
\footnote{
In \cite[Propositions 3.11 and 3.12]{almaraz5} a $\log\rho$ must be included in the arguments for dimensions $3$ and $4$, when the Green function has $\log$ in its expansion; see \eqref{estim:G}.},
relates $\mathcal{I}(x_0)$ with the mass defined by (\ref{def:mass}):
\begin{proposition}\label{propo19} 
Let $x_0\in\mathcal{Z}_{\d M}$ and consider inverted coordinates $y=x/|x|^2$, where $x=(x_1,...,x_n)$ are Fermi coordinates centered at $x_0$. If we define the metric $\bar{g}=G_{x_0}^{\frac{4}{n-2}}g_{x_0}$ on $M\backslash\{x_0\}$, then the following statements hold:

{\bf{(i)}} $(M\backslash\{x_0\},\bar{g})$ is an asymptotically flat manifold with order  $p>\frac{n-2}{2}$ (in the sense of Definition \ref{def:asym}), and  satisfies $R_{\bar{g}}\equiv 0$ and $\cmedia_{\bar{g}}\equiv 0$.

{\bf{(ii)}} We have 
$$
\mathcal{I}(x_0)=\lim_{R\to\infty}
\left\{\int_{\d^+B^+_R(0)}\frac{y_a}{|y|}\frac{\d\bar{g}}{\d y_b}
\left(\frac{\d}{\d y_a},\frac{\d}{\d y_b}\right)\ds_{R}
-\int_{\d^+B^+_R(0)}\frac{y_a}{|y|}\frac{\d\bar{g}}{\d y_a}
\left(\frac{\d}{\d y_b},\frac{\d}{\d y_b}\right)\ds_{R}
\right\}\,.
$$

In particular, $\mathcal{I}(x_0)$ is the mass $m(\bar{g})$ of $(M\backslash\{x_0\},\bar{g})$.
\end{proposition}
\bp[Proof of Proposition \ref{Propo:energy:test}]
Once we have proved Corollary \ref{corol8}, and Propositions \ref{propo18} and \ref{propo19}, this proof follows the same lines as \cite[Proposition 3.7]{almaraz5}. 
\ep

We now prove some further results for later use.
\begin{proposition}\label{Propo1:notes5}\footnote{
The $(\epsilon^2+|x|^2)^{-\frac{1}{2}}$ term in this proposition is necessary only in dimension $3$, when $d=0$ and so $\mathcal{H}=0$. On the other hand, the $|\log \rho|$  term is necessary only in dimensions $3$ and $4$, because of \eqref{estim:G}. The same terms are also necessary in the first inequality of \cite[ Proposition 3.13]{almaraz5}, but this does not affect any other results in that paper because weaker estimates similar to the ones obtained in Subsection \ref{sub:sec:further} are also enough to \cite{almaraz5}.}
For $x\in M$ and $\e<\rho$,  
 \begin{align*} 
  &\left|\ct\Delta_{g_{x_0}} \bu_{(x_0,\e)}-R_{g_{x_0}}\bu_{(x_0,\e)}+\cminfbar\bu_{(x_0,\e)}^{\frac{n+2}{n-2}}\right|(x)
\\
  &\hspace{2cm}\leq C\left(\frac{\epsilon}{\epsilon^2+|x|^2}\right)^{\frac{n-2}{2}}(\epsilon^2+|x|^2)^{-\frac{1}{2}}1_{B^+_{2\rho}(0)}(x)+C\left(\frac{\epsilon}{\epsilon^2+d_{g_{x_0}}(x,x_0)^2}\right)^{\frac{n+2}{2}}1_{M\backslash B^+_{\rho}(0)}(x)
\\
&\hspace{2.5cm}+C(\e^{\frac{n+2}{2}}\rho^{-2-n}+\e^{\frac{n-2}{2}}\rho^{1-n}|\log \rho|)1_{B^+_{2\rho}(0)\backslash B^+_{\rho}(0)}(x).
 \end{align*}
\end{proposition}
\begin{proof}
Note that after scaling, we are assuming $\cminfbar=4n(n-1)$. Then
 \begin{align*}
  \Delta_{g_{x_0}}\bu_{(x_0,\e)}&-\frac{n-2}{4(n-1)}R_{g_{x_0}}\bu_{(x_0,\e)}+\frac{n-2}{4(n-1)}\cminfbar\bu_{(x_0,\e)}^{\frac{n+2}{n-2}}\\
  &=(\Delta_{g_{x_0}} \chi_{\rho})(\U+\phi-\e^\frac{n-2}{2}|x|^{2-n})+2\langle d\chi_{\rho} , d(\U+\phi-\e^\frac{n-2}{2}|x|^{2-n})\rangle_{g_{x_0}}\\
  &\hspace{0.3cm}-(\Delta_{g_{x_0}} \chi_{\rho})\e^{\frac{n-2}{2}}(G_{x_0}-|x|^{2-n})-2\epsilon^{\frac{n-2}{2}}\langle d\chi_{\rho},d(G_{x_0}-|x|^{2-n})\rangle_{g_{x_0}}\\
  &\hspace{0.3cm}+\chi_{\rho}\left(\Delta_{g_{x_0}} (\U+\phi)-\frac{n-2}{4(n-1)}R_{g_{x_0}}(\U+\phi)+n(n-2)(\U+\phi)^{\frac{n+2}{n-2}}\right)\\
  &\hspace{0.3cm}+n(n-2)\left(\left(\chi_{\rho}(\U+\phi)+(1-\chi_{\rho})\e^{\frac{n-2}{2}}G_{x_0}\right)^{\frac{n+2}{n-2}}-\chi_{\rho}(\U+\phi)^{\frac{n+2}{n-2}}\right)\\
  &=I_1+I_2+I_3+I_4
 \end{align*}
where $I_i$, i=1,2,3,4, denote the corresponding row.

To estimate $I_1$, notice that for $|x|\geq \rho>\e$ we have
\begin{align}\label{eq:I_1}
 \big|(\e^2+|x|^2)^{\frac{2-n}{2}}-|x|^{2-n}\big|\leq C\e^2|x|^{-n}
\end{align}
and, equivalently, $|\U-\e^{\frac{n-2}{2}}|x|^{2-n}|\leq C\e^{\frac{n+2}{2}}|x|^{-n}$.
Then $I_1$ can be estimated as 
\[|I_1|\leq C(\e^{\frac{n+2}{2}}\rho^{-2-n}+\e^{\frac{n-2}{2}}\rho^{1-n})1_{B^+_{2\rho}(0)\backslash B^+_{\rho}(0)}.\]
Recall the properties (\ref{estim:G}) of $G_{x_0}$. Then $|I_2|\leq C\e^{\frac{n-2}{2}}\rho^{1-n}|\log\rho|1_{B^+_{2\rho}(0)\backslash B^+_{\rho}(0)}$.

In order to estimate $I_3$, first observe that 
\begin{align*}
 I_3=
&\chi_{\rho}\left((\Delta_{g_{x_0}}-\Delta) \U-\d_i(H_{ij}\d_j\U)-\frac{n-2}{4(n-1)}R_{g_{x_0}}\U+\frac{n-2}{4(n-1)}\d_i\d_jH_{ij}\U\right)\\
&+\chi_{\rho}\left((\Delta_{g_{x_0}}-\Delta) \phi-\frac{n-2}{4(n-1)}R_{g_{x_0}}\phi\right)\\
&+\chi_{\rho}\left(n(n-2)(\U+\phi)^{\frac{n+2}{n-2}}-n(n-2)\U^{\frac{n+2}{n-2}}-n(n+2)\U^{\frac{4}{n-2}}\phi\right),
\end{align*}
where we have used (\ref{eq:Ue}) and (\ref{eq:phi}).
Using \cite[inequality (3.20)]{almaraz5},
$$
|(\Delta_{g_{x_0}}-\Delta) \U+\d_i(H_{ij}\d_j\U)|+|R_{g_{x_0}}\U-\d_i\d_jH_{ij}\U|
\leq 
C\e^{\frac{n-2}{2}}(\e+|x|)^{1-n},
$$
$$
|(\Delta_{g_{x_0}}-\Delta) \phi+\d_i(H_{ij}\d_j\phi)|+|R_{g_{x_0}}\phi-\d_i\d_jH_{ij}\phi|
\leq C\e^{\frac{n-2}{2}}(\e+|x|)^{2-n}
$$
and
$$
\Big|(\U+\phi)^{\frac{n+2}{n-2}}-\U^{\frac{n+2}{n-2}}-\frac{n+2}{n-2}\U^{\frac{4}{n-2}}\phi\Big|
\leq C\U^{\frac{n+2}{n-2}}(\phi\U^{-1})^2
\leq  C\e^{\frac{n+2}{2}}(\e+|x|)^{-n}.
$$
This leads to
\[|I_3|\leq C\left(\frac{\e}{\e^2+|x|^2}\right)^{\frac{n-2}{2}}(\epsilon^2+|x|^2)^{-\frac{1}{2}}1_{B^+_{2\rho}(0)}.\]

Finally we consider $I_4$, using the elementary inequality 
$$
|a^{\frac{n+2}{n-2}}-b^{\frac{n+2}{n-2}}|\leq Cb^{\frac{4}{n-2}}|a-b|+C|a-b|^{\frac{n+2}{n-2}},
$$
which holds for any $a,b>0$, and where $C=C(n)$. Letting $a=\chi_{\rho}(\U+\phi)+(1-\chi_{\rho})\e^{\frac{n-2}{2}}G_{x_0}$ and $b=\chi_{\rho}^{\frac{n-2}{n+2}}(\U+\phi)$, and applying the bound  (\ref{estim:G}) for $G_{x_0}$, one gets the estimate
\[|I_4|\leq C\left(\frac{\epsilon}{\epsilon^2+d_{g_{x_0}}(x,x_0)^2}\right)^{\frac{n+2}{2}}1_{M\backslash B^+_{\rho}(0)}.\]
Combining all the estimates above, we get the conclusion.
\end{proof}
\begin{proposition}\label{Propo2:notes5}
For $x\in \d M$,
 \[\left|\frac{2(n-1)}{n-2}\frac{\partial}{\partial \eta_{g_{x_0}}}\bu_{(x_0,\e)}-H_{g_{x_0}}\bu_{(x_0,\e)}\right|(x)
\leq C\rho\left(\frac{\epsilon}{\epsilon^2+|\bx|^2}\right)^{\frac{n-2}{2}}1_{D_{2\rho}(0)}(x).\]
\end{proposition}
\begin{proof}
Observe that
 \begin{align*}
  \frac{\partial}{\partial\eta_{g_{x_0}}}\bu_{(x_0,\e)}-\frac{n-2}{2(n-1)}H_{g_{x_0}}\bu_{(x_0,\e)}=&\chi_{\rho}\frac{\partial}{\partial\eta_{g_{x_0}}}(\U+\phi)+\frac{n-2}{2(n-1)}\chi_{\rho} H_{g_{x_0}}(\U+\phi)\\
  &+(1-\chi_{\rho})\e^{\frac{n-2}{2}}\left(\frac{\partial}{\partial\eta_{g_{x_0}}}G_{x_0}-\frac{n-2}{2(n-1)}H_{g_{x_0}}G_{x_0}\right).
 \end{align*}
Recall that we were using Fermi coordinates, thus $\eta_{g_{x_0}}=\partial_n$. The first and third terms are zero by the equations \eqref{eq:Ue} and \eqref{eq:phi} while the middle one can be bounded as 
\[|\chi_{\rho} H_{g_{x_0}}(\U+\phi)|\leq C\rho\left(\frac{\e}{\e^2+|\bx|^2}\right)^{\frac{n-2}{2}}1_{D_{2\rho}(0)}.\]
\end{proof}



\subsection{Type B test functions ($\bar u_{B;(x_0,\e)}$)}\label{sub:sec:deftestfunct:tubular}

In this case the test functions we use are essentially the same as in \cite{brendle-invent}. However, when trying to control their energy by $\Y$, due to the proximity to the boundary, the argument in that paper cannot be directly applied. We are able to overcome this difficulty by exploiting the sign of $\d_n \U(0)$ (see the definition in \eqref{eq:def:U}). Since all the argument is local, we do not make use of the positive mass theorem in this subsection.

Fix $x_0\in M_{2\delta_0}\backslash  \d M$ and let $\psi_{x_0}: \tilde{B}_{x_0, 2\rho}\subset\R^n\to  M$ be normal coordinates centered at $x_0$  (see Definition \ref{def:normal}) where $0<\rho\leq P_0$. 
We will sometimes omit the symbols $\psi_{x_0}$ in order to simplify our notations, identifying $\psi_{x_0}(x)\in M$ with $x\in \tilde{B}_{x_0, 2\rho}$. In those coordinates, we have the properties $g_{ab}(0)=\delta_{ab}$ and $\d_cg_{ab}(0)=0$, for $a,b,c=1,...,n$. If we write $g=\exp(h)$, where $\exp$ denotes the matrix exponential, then the symmetric 2-tensor $h$ satisfies the following properties:
\begin{equation}\notag
\begin{cases}
h_{ab}(0)=0\,,&\text{for}\: a,b=1,...,n\,,
\\
\d_ch_{ab}(0)=0\,,&\text{for}\: a,b,c=1,...,n\,,
\\
\sum_{b=1}^{n}x_b h_{ab}(x)=0\,,&\text{for}\:x\in \tilde{B}_{x_0,\rho},\: a=1,...,n\,.
\end{cases}
\end{equation}  

According to  \cite{lee-parker},  we can find a conformal metric $g_{x_0}=f_{x_0}^{\frac{4}{n-2}}g_0$, with $f_{x_0}(x_0)=1$, such that $\text{det}(g_{x_0})(x)=1+O(|x|^{2d+2})$ in normal coordinates  centered at $x_0$, again written $\psi_{x_0}: \tilde{B}_{x_0, 2\rho}\to  M$ for simplicity. 
We can suppose that $1/2\leq f_{x_0}\leq 3/2$ .

\begin{notation}
In order to simplify notations, in the coordinates above, we will write $g_{ab}$ and $g^{ab}$ instead of $(g_{x_0})_{ab}$ and $(g_{x_0})^{ab}$ respectively, $h_{ab}$ instead of $(h_{x_0})_{ab}$, and $\eta^{a}$ instead of $(\eta_{g_{x_0}})^{a}$.
We denote by $\nu=\nu_{x_0}$ the unit normal vector to $\tilde{D}_{x_0,\rho}$ with respect to the Euclidean metric $\delta_{ab}$, pointing the same way as $\eta_{g_0}$ and $\eta_{g_{x_0}}$, and write $\nu=\nu^a\d_a$ and $\eta=\eta^a\d_a$.
\end{notation}
Set $\delta=d_{g_{x_0}}(x_0,\d M)$. If $\tilde{x}_0\in \d M$ is chosen such that $d_{g_{x_0}}(x_0,\tilde{x}_0)=\delta$ then we can assume that  $\psi_{x_0}$ takes $(-\delta,0,\cdots, 0)\in \mathbb{R}^n$ to $\tilde{x}_0$ and thus both $\eta_{g_{x_0}}$ and $\nu_{x_0}$ coincide at $\tilde{x}_0$ with the coordinate vector $\partial_n$. So, there exists $C_0=C_0(M,g_0)>2$ such that
\begin{align}\label{eq:eta_property}
 |\eta^a(x)-\delta_{an}|\leq C_0|\bx|, \:\:\:\:\text{and}
\end{align}
\begin{align}\label{eq:nu_property}
 |\nu^a(x)-\delta_{an}|\leq C_0|\bx|, \:\:\:\:\text{for all}\:x\in \tilde{D}_{x_0,2\rho},
\end{align}
where $x=(x_1,\cdots,x_n)=(\bar{x},x_n)\in \R^n$. 
We will also assume that $\tilde{D}_{x_0, 2\rho}$ is the graph of a smooth function $\gamma=\gamma_{x_0}$ so that 
$$
\tilde{D}_{x_0, 2\rho}=\{x=(\bar{x},\gamma(\bar{x}))\:|\:|x|<2\rho\}.
$$ 
We can write $\gamma(\bar{x})=-\delta+O(|\bar{x}|^2)$ and choose $C_0$ larger if necessary such that
\begin{align}\label{eq:f_property}
 |\gamma(\bx)+\delta|\leq C_0|\bar{x}|^2, \:\:\:\:\text{for all}\:x\in \tilde{D}_{x_0, 2\rho}.
\end{align}
See Figure 1.
\begin{figure}[t]
\begin{center}
 \begin{tikzpicture}
 \begin{scope}[shift={(-6,0)}, scale=0.7]
 \draw [domain=-33:228] plot ({2cm*cos(\x)}, {2cm*sin(\x)});
 \draw[black,fill=red] (1.68,-1.08) circle (.1ex);
 \draw[black,fill=red] (-1.35,-1.46) circle (.1ex);
 \draw (-2,-2) to [out=50,in=210] (-1.35,-1.46) to [out=30, in=180] (0,-1);
 \draw [very thick] (-1.35,-1.46) to [out=30, in=180] (0,-1);
 \draw [very thick] (0,-1) to[out=0,in=200] (1.68,-1.08);
 \draw (0,-1) to [out=0,in=200] (1.68,-1.08) to [out=20,in=170](3,-1);
 \draw [<-, thin, dotted] (4,0) node [right]{$\bar{x}$} -- (-4,0);
 \draw [->,thin, dotted] (0,-4) --(0,4) node [left] {$x_n$};
 \node at (-4,4) {$\mathbb{R}^n$};
 \node [below left] at (0,-1.1) {$-\delta$};
 \draw [fill=black] (0,-1.0) circle (2.5pt);
 \node at (1,0.5) {$\tilde B_{x_0,\rho}$};
 \node at (-3,1) {$\partial^+\tilde B_{x_0,\rho}$};
 \node at (1,-2) {${\tilde D}_{x_0,\rho}$};
 \draw [->] (0.8,-1.6) to (0.7,-1.35);
 \draw [->] (-2.1,1) to (-1.8,1.11);
 \node at(-2.5,-2.5) {$\gamma(\bar{x})$};
 \end{scope}
 \begin{scope}[shift={(-0.8,-2.5)},scale=0.7]
 \node (A) at (-4,5) {};
 \node (B) at (-2,5) {};
 \draw (-2,4) to [out=-20,in=190](0, 3) to [out=10,in=180] (1,2.5) to [out=0,in=200](2,3.5) to (3,4);
 \draw (0,3) to [out=90,in=180] (1.2,4.5) to [out=0,in=90] (2,3.5);
  \path[->,font=\scriptsize]
 (A) edge [right, bend left] node[above]{$\psi_{x_0}$} (B);
 \node at (4,4) {${\partial} M$};
 \node at (3,6) {$M$};
 \draw[black,fill=red] (1.1,3.2) circle (.2ex) node [above] {$x_0$};
 \draw[->] (-.5,3) to (-.4,3.5);
 \node at (-.5,4.0) {$\eta_{{g_{x_0}}}$};
 \end{scope}
\end{tikzpicture}
\end{center}
\caption{Some notations.}
\end{figure}
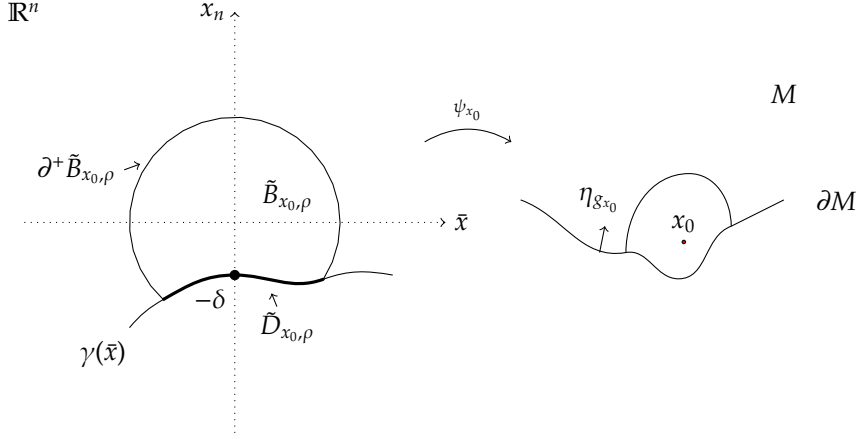

In this subsection, we denote by 
\begin{equation}\notag 
\mathcal{H}_{ab}(x)=\sum_{2\leq |\a|\leq d}h_{ab,\a}x^{\a}
\end{equation}
the Taylor expansion of order $d=\big[\frac{n-2}{2}\big]$ associated with the function $h_{ab}(x)$. Thus, $h_{ab}(x)=\mathcal{H}_{ab}(x)+O(|x|^{d+1})$.  
We define $\phi$, $S$, $T$ and $Q_{ab,c}$ as in Subsection \ref{subsec:algebraic} (see (\ref{eq:def:phi}), (\ref{def:S:T}) and (\ref{eq:def:Q})), except for the fact that, as in \cite{brendle-invent}, the whole construction is done in $\R^n$ instead of $\Rn$. Then the first equation of \eqref{eq:phi} and the estimates \eqref{est:phi} and \eqref{est:lapl:phi} also hold, with $2\leq |\a|\leq d$ replacing  $1\leq |\a|\leq d$.
\begin{lemma}\label{corol10}
There exists $\l=\l(n)>0$ such that
$$
\lambda\e^{n-2}\sum_{a,b=1}^{n}\sum_{|\a|=2}^{d}|h_{ab,\a}|^2\int_{B_{\rho}(0)}(\e+|x|)^{2|\a|+2-2n}dx\leq \frac 14 \int_{B_\rho(0)}Q_{ab,c}Q_{ab,c}
$$
for all $\rho\geq 2\e$.
\end{lemma}
\bp
See \cite[Corollary 10]{brendle-invent}.
\ep

Recall the definitions of $\U$ in (\ref{eq:def:U}), $\chi_{\rho}$ in (\ref{def:eta}), and $\cminfbar$ in (\ref{eq:def:cminfbar}). Set
\begin{align}
\ubarrho(x)
=&\left(\frac{4n(n-1)}{\cminfbar}\right)^{\frac{n-2}{4}}\chi_{\rho}(\psi_{x_0}^{-1}(x))\big(\U(\psi_{x_0}^{-1}(x))+\phi(\psi_{x_0}^{-1}(x))\big)\notag
\\
&\hspace{2cm}+\left(\frac{4n(n-1)}{\cminfbar}\right)^{\frac{n-2}{4}}\e^{\frac{n-2}{2}}\big(1-\chi_{\rho}(\psi_{x_0}^{-1}(x))\big)G_{x_0}(x)\,,\notag
\end{align}
for $x\in M$. 
Here, $G_{x_0}$ is the Green's function of the conformal Laplacian $L_{g_{x_0}}$ with pole at $x_0\in M\backslash\d M$, satisfying the boundary condition \eqref{eq:G:bordo} and the normalization $\lim_{|y|\to 0}|y|^{n-2}G_{x_0}(\psi_{x_0}(y))=1/2$. This function is obtained in Proposition \ref{green:tubular} and satisfies, for some $C=C(M,g_0)$,
\ba\label{estim:G:simpl}
|G_{x_0}(\psi_{x_0}(y))-|y|^{2-n}|
&\leq
\begin{cases}
C|y|^{3-n}+C\delta|y|^{1-n}&\text{if}\:n\geq 4,
\\
C(1+|\log |y||)+C\delta|y|^{1-n}&\text{if}\:n=3,
\end{cases}
\\
\Big|\frac{\d}{\d y_b}(G_{x_0}(\psi_{x_0}(y))-|y|^{2-n})\Big|
&\leq
C|y|^{2-n}+C\delta|y|^{-n}\,,\notag
\end{align}
for all $b=1,...,n$ and $\psi_{x_0}(y)\in M_{\tilde{\delta}}$ for some small $\tilde{\delta}=\tilde{\delta}(M,g_0)$.

Define the test function
\begin{equation}\label{def:test:func:u:tubular}
\bar u_{B;(x_0,\e)}=f_{x_0}\ubarrho.
\end{equation}
Observe that this function also depends on the radius $\rho$ above, which will be fixed later in Section \ref{sec:blowup}. Such constant will also be referred to as $\rho_B$ in order to avoid confusion with test functions of the other subsections. 

The main result of this subsection is the following:
\begin{proposition}\label{Propo:energy:test:tubular}
Under the hypothesis of Theorem \ref{main:thm}, there exist positive $P_2$ and $C_B$, depending only on $(M,g_0)$, such that for any $\rho_B\leq P_2$ one can choose $\delta_0<C_B\rho_B^2$ satisfying
\ba
&\frac{\int_M\left\{\frac{4(n-1)}{n-2}|d \bar u_{B;(x_0,\e)}|_{g_0}^2+R_{g_0}\bar u_{B;(x_0,\e)}^2\right\}\dv_{g_0}}{\left(\int_{M}\bar u_{B;(x_0,\e)}^{\crit}\dv_{g_0}\right)^{\frac{n-2}{n}}}\notag
\\
&\hspace{0.5cm}=\frac{\int_M\left\{\frac{4(n-1)}{n-2}|d \ubar|_{g_{x_0}}^2
+R_{g_{x_0}}\ubar^2\right\}\dv_{g_{x_0}}
+\int_{\d M}2\cmedia_{g_{x_0}}\ubar^2\ds_{g_{x_0}}}{\left(\int_{M}\ubar^{\crit}\dv_{g_{x_0}}\right)^{\frac{n-2}{n}}}\notag
\\
&\hspace{0.5cm}\leq \Y\notag
\end{align}
for all $x_0\in M_{2\delta_0}\backslash \d M$ and $0<\e<C_B^{-1}d_{g_0}(x_0, \d M)$.
\end{proposition}

We will prove several lemmas before proceeding to the proof of Proposition \ref{Propo:energy:test:tubular}.
\begin{lemma}\label{Lemma1:notes3}
 If $|\bar{x}|\leq 1/(2C_0)$, then for $\e>0$ and $0<\delta<1$ we have
\begin{equation}\label{Lemma1:notes3:1}
\frac{1}{2C_0}(\e^2+|\bar{x}|^2+\delta^2)<\e^2+|\bx|^2+\gamma(\bx)^2<2(\e^2+|\bar{x}|^2+\delta^2).
\end{equation}
\end{lemma}
\begin{proof}
First assume $\delta\geq C_0|\bx|^2$. Since $|\gamma(\bx)|\geq \delta-C_0|\bx|^2\geq 0$ by (\ref{eq:f_property}), Cauchy's inequality implies
 \begin{align*}
  \gamma(\bx)^2\geq \left(\delta-C_0|\bx|^2\right)^2\geq \delta^2-2C_0\delta|\bx|^2\geq \frac 12\delta^2-2 C_0^2|\bx|^4.
 \end{align*}
So,
\[\e^2+|\bx|^2+\gamma(\bx)^2\geq \e^2+(1-2C_0^2|\bx|^2)|\bx|^2+\frac 12\delta^2,\]
and our assumption $|\bx|^2\leq 1/(4C_0^2)$ gives 
\[\e^2+|\bx|^2+\gamma(\bx)^2\geq \e^2+\frac 12|\bx|^2+\frac 12\delta^2>\frac 12(\e^2+|\bx|^2+\delta^2).\]

If $\delta<C_0|\bx|^2$ we have
\[|\bx|^2+\gamma(\bx)^2+\e^2> \frac{\delta^2}{2C_0}+\frac{|\bx|^2}{2}+\e^2> \frac{1}{2C_0}(\delta^2+|\bx|^2+\e^2).\]
so the left side of (\ref{Lemma1:notes3:1}) is proved. 

As for the right side, notice that
\[\gamma(\bx)^2\leq (\delta+C_0|\bx|^2)^2\leq 2\delta^2+2C_0^2|\bar{x}|^4.\]
Consequently,
\[\e^2+|\bx|^2+\gamma(\bx)^2\leq \e^2+(1+2C_0^2|\bx|^2)|\bx|^2+2\delta^2<2(\e^2+|\bx|^2+\delta^2),\]
because our assumption on $|\bar{x}|$  implies $1+2C_0^2|\bx|^2\leq 2$.
\end{proof}
\begin{lemma}\label{Lemma3:notes3}
 If $0<\rho<1/C_0$ and $0<\delta\leq \rho/4$ then
$$
\sqrt{ |\bx|^2+\gamma(\bx))^2}<\rho,\:\:\:\:\text{for all}\:|\bx|\leq\rho/2.
$$
\end{lemma}
\begin{proof}
 From our assumption it is easy to get $\delta/\rho+C_0\rho/4< 1/2$.
 Since $$|\gamma(\bx)|\leq \delta+C_0|\bx|^2\leq \delta+C_0\rho^2/4,$$ 
 we have
 \[|\bx|^2+\gamma(\bx)^2\leq \frac{\rho^2}{4}+\Big(\delta+\frac{C_0\rho^2}{4}\Big)^2< \frac{\rho^2}{4}+\left(\frac{\rho}{2}\right)^2=\frac{\rho^2}{2}.\]
\end{proof}
 \begin{lemma}\label{Lemma5:notes3}
 If $0<\rho\leq 1/C_0$ and $0<\delta<1$ then
 \[\sqrt{|\bx|^2+\gamma(\bx)^2}>\delta/\sqrt{C_0}\,,\quad\quad\text{for all }|\bx|< \rho.\]
\end{lemma}
\begin{proof}
First assume $\delta\geq C_0|\bx|^2$. Then
 $|\gamma(\bx)|\geq \delta-C_0|\bx|^2\geq 0$,
which yields
 \begin{align*}
  \gamma(\bx)^2&\geq(\delta-C_0|\bx|^2)^2
\\
&\geq \delta^2-2\delta C_0|\bx|^2+C_0^2|\bx|^4
=\frac{\delta^2}{2}-C_0^2|\bx|^4.
 \end{align*}
Therefore, by the assumption $|\bx|<\rho\leq 1/C_0$, we have
\[|\bx|^2+\gamma(\bx)^2\geq (1-C_0^2|\bx|^2)|\bx|^2+\delta^2/2\geq \delta^2/2> \delta^2/C_0,\]
because $C_0> 2$.

If $\delta<C_0|\bx|^2$, since $0<\delta<1$, we have $\delta^2<\delta<C_0|\bx|^2$. Obviously
\[|\bx|^2+\gamma(\bx)^2>\delta^2/C_0,\]
proving the result.
\end{proof}
\begin{lemma}\label{Lemma2:notes3}
 There exists $C=C(n)$ such that
 \[\int_{\{\bar x\in \R^{n-1}|\:|\bx|\leq\rho\}}(\e^2+|\bx|^2+\delta^2)^{2-n}d\bx\leq C\rho\delta^{2-n},\quad \text{ for } 0<\delta\leq \rho.\]
\end{lemma}
\begin{proof}
Just observe that
 \begin{align*}
  \int_{|\bx|\leq\rho}(\e^2+|\bx|^2+\delta^2)^{2-n}d\bx
&\leq \int_{|\bx|\leq\rho}(|\bx|^2+\delta^2)^{2-n}d\bx\\
  &\leq \sqrt{2}\rho\int_{\mathbb{R}^{n-1}}(|\bx|^2+\delta^2)^\frac{3-2n}{2}d\bx
  = \sqrt{2}\rho\delta^{2-n}\int_{\mathbb{R}^{n-1}}(|\bar{y}|^2+1)^\frac{3-2n}{2}d\bar{y}.
 \end{align*}
\end{proof}
\begin{lemma}\label{Lemma4:notes3}
 There exist $\tilde{c}, K, P_2>0$, depending only on  $(M,g_0)$, such that
 \[\frac{4(n-1)}{n-2}\int_{\tilde{D}_{x_0, \rho}}\U\d_\nu\U d\sigma\geq \tilde{c}\e^{n-2}\delta^{2-n}\]
 when $0<\e<\delta<K\rho$ and $\rho<P_2$.
\end{lemma}
\begin{proof}
Observe that  $\U\partial_a \U=-(n-2)\e^{n-2}(\e^2+|x|^2)^{1-n}x_a$
 and, on $\tilde D_{x_0,\rho}$,
 \[\U\d_\nu\U =\U\nu^a\partial_a\U=\U\partial_n\U+\U(\nu^a-\delta_{an})\partial_a\U.\]
 Using \eqref{eq:nu_property} and Lemma \ref{Lemma1:notes3}, we have 
 \begin{align}
 |\U(\nu^a-\delta_{an})\partial_a\U|(x)
&\leq (n-2)C\e^{n-2}(\e^2+|\bx|^2+\gamma(\bx)^2)^{2-n}\notag\\
 &\leq (2C_0)^{n-2}(n-2)C\e^{n-2}(\e^2+|\bx|^2+\delta^2)^{2-n}\notag
 \end{align}
 when $x=(\bar x,\gamma(\bar x))\in\tilde D_{x_0,\rho}$ with $|\bx|\leq (2C_0)^{-1}$. Hence if $\rho\leq (2C_0)^{-1}$ and $0<\delta\leq\rho$, then
 \[
\int_{\tilde{D}_{x_0, \rho}}\U\d_\nu\U d\sigma\geq \int_{\tilde{D}_{x_0, \rho}}\U\partial_n \U d\sigma-C\rho\left(\frac{\e}{\delta}\right)^{n-2},
\]
where we used Lemma \ref{Lemma2:notes3}. 

In order to estimate from below the r.h.s. of this last inequality, we see that
 \begin{align*}
  \U\partial_n\U(x)
&=-(n-2)\e^{n-2}(\e^2+|\bx|^2+\gamma(\bx)^2)^{1-n}\gamma(\bx)\\
  &\geq (n-2)\e^{n-2}(\e^2+|\bx|^2+\gamma(\bx)^2)^{1-n}(\delta-C_0|\bx|^2)\\
  &\geq (n-2)\e^{n-2}\delta(\e^2+|\bx|^2+\gamma(\bx)^2)^{1-n}-(n-2)C_0\e^{n-2}(\e^2+|\bx|^2+\gamma(\bx)^2)^{2-n}\\
  &\geq (n-2)2^{1-n}\e^{n-2}\delta(\e^2+|\bx|^2+\delta^2)^{1-n}-C\e^{n-2}(\e^2+|\bx|^2+\delta^2)^{2-n}
 \end{align*}
for $x=(\bar x,\gamma(\bar x))\in\tilde D_{x_0,\rho}$ with $|\bar x|\leq (2C_0)^{-1}$, where we used Lemma \ref{Lemma1:notes3} in the last step.

Assume $0<\rho<(2C_0)^{-1}$ and $0<\delta\leq \rho/4$. According to Lemma \ref{Lemma3:notes3}, 
\[\big\{(\bx,\gamma(\bx))\:\:\big|\:\:|\bx|\leq\rho/2\big\}\subset \tilde{D}_{x_0, \rho}.\]
Then
\begin{align*}
 \int_{\tilde{D}_{x_0, \rho}}\U\partial_n \U d\sigma
&\geq (n-2)2^{1-n}\e^{n-2}\delta\int_{|\bx|\leq \rho/2}(\e^2+|\bx|^2+\delta^2)^{1-n}d\bx\\
 &\hspace{0.2cm}-C\e^{n-2}\int_{|\bx|< \rho}(\e^2+|\bx|^2+\delta^2)^{2-n}d\bx\\
&=I-II.
\end{align*}
Notice that
\begin{align*}
 \delta\int_{|\bx|\leq \rho/2}(\e^2+|\bx|^2+\delta^2)^{1-n}d\bx&=\delta^{2-n}\int_{|\bar{y}|\leq \rho/2\delta}\left(\left(\frac{\e}{\delta}\right)^2+|\bar{y}|^2+1\right)^{1-n}d\bar{y}\\
 &\geq 2^{1-n}\delta^{2-n}\int_{|\bar{y}|\leq \rho/2\delta}(|\bar{y}|^2+1)^{1-n}d\bar{y}
\end{align*}
for $0<\e<\delta$, because $(\e/\delta)^2+|\bar{y}|^2+1<2(|\bar{y}|^2+1)$.

Set $\a(n)=\int_{\mathbb{R}^{n-1}}(|\bar{y}|^2+1)^{1-n}d\bar{y}$ and observe that
\begin{align*}
 \int_{|\bar{y}|\leq \rho/2\delta}(|\bar{y}|^2+1)^{1-n}d\bar{y}
=\a(n)-\int_{|\bar{y}|> \rho/2\delta}(|\bar{y}|^2+1)^{1-n}d\bar{y}
\geq \a(n)-C\left(\frac{\delta}{\rho}\right)^{n-1}.
\end{align*}
Hence,
$$
 I\geq (n-2)2^{2-2n}\a(n)\left(\frac{\e}{\delta}\right)^{n-2}-C\left(\frac{\delta}{\rho}\right)^{n-1}\left(\frac{\e}{\delta}\right)^{n-2}.
$$
On the other hand, $II\leq C\rho\left(\e/\delta\right)^{n-2}$, by Lemma \ref{Lemma2:notes3}.

Putting things together, we obtain
$$
 \int_{\tilde{D}_{x_0, \rho}}\U\d_\nu\U d\sigma
\geq  (n-2)2^{2-2n}\big(\a(n)-C(\delta/\rho)^{n-1}-C\rho\big)\left(\e/\delta\right)^{n-2},
$$
from which the result follows.
\end{proof}
\begin{proposition}\label{Propo1:notes4}
 There exists $P_2=P_2(M,g_0)>0$ such that if $0<\delta\leq \rho\leq P_2$
 \begin{align*}
&\int_{\tilde{B}_{x_0\rho}}\left\{\frac{4(n-1)}{n-2}|d(\U+\phi)|^2+R_{g_{x_0}}(\U+\phi)^2\right\}dx
\\
&\hspace{1.5cm}\leq
\frac{4(n-1)}{n-2}\int_{\tilde{B}_{x_0,\rho}}|d\U|^2dx+\int_{\tilde{B}_{x_0,\rho}}\frac{4(n-1)}{n-2}n(n+2)\U^{\frac{4}{n-2}}\phi^2dx\notag
\\
&\hspace{1.5cm}+\frac{\lambda}{2}\sum_{a,b=1}^{n}\sum_{|\a|=2}^{d}|h_{ab,\a}|^2\e^{n-2}\int_{\tilde{B}_{x_0, \rho}}(\e+|x|)^{2|\alpha|+2-2n}dx
-\frac 14\int_{\tilde{B}_{x_0,\rho}}Q_{ab,c}Q_{ab,c}\,dx+C\rho\left(\frac{\e}{\delta}\right)^{n-2}+C\rho\left(\frac{\e}{\rho}\right)^{n-2}\notag
 \end{align*}
for all $\e\in(0,\rho/2]$. Here, $\l$ is the constant obtained in Lemma \ref{corol10}.
\end{proposition}
\begin{proof}
 As in \cite[Proposition 3.6]{brendle-chen}, we can choose $0<P_2<1$ such that
 \begin{align*}
  &\int_{\tilde{B}_{x_0,\rho}}\left\{\frac{4(n-1)}{n-2}|d(\U+\phi)|^2+R_{g_{x_0}}(\U+\phi)^2\right\}dx
\\
&\leq \frac{4(n-1)}{n-2}\int_{\tilde{B}_{x_0\rho}}|d\U|^2dx+\int_{\tilde{B}_{x_0, \rho}}\frac{4(n-1)}{n-2}n(n+2)\U^{\frac{4}{n-2}}\phi^2dx
\\
&\hspace{1cm}+\int_{\d^+\tilde{B}_{x_0, \rho}}\left(\U^2\d_bh_{ab}-\d_b\U^2h_{ab}\right)
\frac{x_a}{|x|}\ds_{\rho}
-\frac{1}{4}\int_{\tilde{B}_{x_0, \rho}}Q_{ab,c}Q_{ab,c}dx
\\
&\hspace{1cm}+\frac{\l}{2}\sum_{a,b=1}^{n}\sum_{|\a|=2}^{d}|h_{ab,\a}|^2\e^{n-2}\int_{\tilde{B}_{x_0, \rho}}(\e+|x|)^{2|\a|+2-2n}dx
\\
&\hspace{1cm}+C\sum_{a,b=1}^{n}\sum_{|\a|=2}^{d}|h_{ab,\a}|\e^{n-2}\rho^{|\a|+2-n}
+C\e^{n-2}\rho^{2d+4-n}+\int_{\tilde{D}_{x_0, \rho}}\Psi d\sigma
 \end{align*}
holds for all $0<2\e\leq \rho\leq P_2$, where 
\begin{align*}
 \Psi=-\frac{8(n-1)}{n-2}\left(\partial_a \U\phi+\frac{(n-2)^2}{2}\U^\crit V_a\right)\nu^a
 -\U^2\partial_bh_{ab}\nu^a+2\U(\partial_b \U)h_{ab}\nu^a+\U^2\mathcal H_{ab}\partial_c\mathcal H_{ab}\nu^b
-\nu^a\xi_a
\end{align*}
comes from integration by parts. Here, $\xi_a$ is a 1-tensor controlled by $$|\xi_a(x)|\leq C\sum_{a,b=1}^{n}\sum_{|\a|=2}^{d}|h_{ab,\a}|^2\e^{n-2}(\e+|x|)^{3+2|\a|-2n}.$$
It is easy to estimate the following term on $\tilde{D}_{x_0, \rho}$
\begin{align}\label{Propo1:notes4:1}
 |\U^\crit V_a|(x)\leq 
C\e^{n}(\e^2+|\bx|^2+\gamma(\bx)^2)^{1-n}
 \leq 
C\e^{n-2}(\e^2+|\bx|^2+\gamma(\bx)^2)^{2-n},
\end{align}
and all the other terms in $\Psi$ can also be estimated by the r.h.s. of \eqref{Propo1:notes4:1}.

Choosing $P_2$ possibly smaller, from Lemmas \ref{Lemma1:notes3} and \ref{Lemma2:notes3} we get
\begin{equation}\label{Propo1:notes4:2}
\int_{\tilde{D}_{x_0,\rho}}\Psi d\sigma
\leq C\left(\frac{\e}{\delta}\right)^{n-2}\rho,
\end{equation}
for $0<\delta\leq \rho$, from which the result follows.
\end{proof}
\begin{proposition}\label{Propo2:notes4}
There exist $P_2, C>0$, depending only on $(M,g_0)$, such that
 \begin{align*}
&\int_{\tilde{B}_{x_0, \rho}}\left\{\frac{4(n-1)}{n-2}|d(\U+\phi)|^2+R_{g_{x_0}}(\U+\phi)^2\right\}dx
\\
&\hspace{1cm}\leq \Y\left(\int_{\tilde{B}_{x_0, \rho}}(\U+\phi)^{\frac{2n}{n-2}}dx\right)^{\frac{n-2}{n}}
-(\tilde{c}-C\rho-C(\delta/\rho)^{n-2})\left(\frac{\e}{\delta}\right)^{n-2}
\\
&\hspace{1.5cm} -\frac{\lambda}{4}\sum_{a,b=1}^{n}\sum_{|\a|=2}^{d}|h_{ab,\a}|^2\e^{n-2}\int_{\tilde{B}_{x_0, \rho}}(\e+|x|)^{2|\alpha|+2-2n}dx
\end{align*}
for all $0< \rho\leq P_2$ and $0<\e<\delta<K\rho$, where $K$ and $\tilde{c}$ are the constants obtained in Lemma \ref{Lemma4:notes3}.
\end{proposition}
\begin{proof}
This result is a consequence of Proposition \ref{Propo1:notes4} and Lemma \ref{corol10}. 
Observe that
\begin{align}\label{Propo2:notes5:3}
\frac{4(n-1)}{n-2}&\int_{\tilde{B}_{x_0, \rho}}|d\U|^2dx+\int_{\tilde{B}_{x_0, \rho}}\frac{4(n-1)}{n-2}n(n+2)\U^{\frac{4}{n-2}}\phi^2dx
\\
&=\int_{\tilde{B}_{x_0, \rho}}\frac{4(n-1)}{n-2}\left(n(n-2)\U^\crit+n(n+2)\U^{\frac{4}{n-2}}\phi^2\right)dx\notag
\\
&\hspace{0.2cm}-\int_{\tilde{D}_{x_0, \rho}}\frac{4(n-1)}{n-2}\U\d_\nu\U d\sigma+\int_{\partial^+\tilde{B}_{x_0, \rho}}\frac{4(n-1)}{n-2}\U\partial_a\U\frac{x_a}{|x|}d\sigma\notag
\\
&\leq \int_{\tilde{B}_{x_0, \rho}}4n(n-1)\U^{\frac{4}{n-2}}(\U^2+\frac{n+2}{n-2}\phi^2)\,dx\notag
\\
&\hspace{0.2cm}-\int_{\tilde{D}_{x_0, \rho}}\frac{4(n-1)}{n-2}\U\d_\nu\U d\sigma+C\left(\frac{\e}{\rho}\right)^{n-2}.\notag
\end{align}

We shall handle the first two terms of the r.h.s. of \eqref{Propo2:notes5:3} separately.
As in \cite[ Proposition 14]{brendle-invent}, we have
\begin{align*}
\Big(\U^2+\frac{n+2}{n-2}\phi^2\Big)^{\frac{n}{n-2}}-(\U+\phi)^{\frac{2n}{n-2}}+\frac{2n}{n-2}\U^{\frac{n+2}{n-2}}\phi
\leq 
C\sum_{a,b=1}^{n}\sum_{|\a|=2}^{d}|h_{ab,\a}|^2\e^n(\e+|x|)^{2|\a|+2-2n}
\end{align*}
and
\begin{align*}
\int_{\tilde{B}_{x_0, \rho}}\frac{2n}{n-2}\U^{\frac{n+2}{n-2}}\phi\, dx
\geq\int_{\tilde{B}_{x_0, \rho}}\d_a(\U^{\frac{2n}{n-2}}V_a)\, dx
&=\int_{\d^+\tilde{B}_{x_0, \rho}}\U^{\frac{2n}{n-2}}V_a\frac{x_a}{|x|}\, d\sigma
-\int_{\tilde{D}_{x_0, \rho}}\U^{\frac{2n}{n-2}}V_a\nu^a\, d\sigma
\\
&\geq
-C\rho^{1-n}\e^n-C\rho\left(\frac{\e}{\delta}\right)^{n-2}.
\end{align*}
Here, in the last step we estimated the integral on $\tilde{D}_{x_0,\rho}$ by \eqref{Propo1:notes4:1} and Lemmas \ref{Lemma1:notes3} and \ref{Lemma2:notes3}. So,
\begin{align}\label{Propo2:notes5:1}
\int_{\tilde{B}_{x_0, \rho}}&4n(n-1)\U^{\frac{4}{n-2}}(\U^2+\frac{n+2}{n-2}\phi^2)\,dx
\leq \Y\left(\int_{\tilde{B}_{x_0, \rho}}(\U^2+\frac{n+2}{n-2}\phi^2)^{\frac{n}{n-2}}dx\right)^{\frac{n-2}{n}}
\\
&\leq \Y\left(\int_{\tilde{B}_{x_0, \rho}}(\U+\phi)^\crit dx\right)^{\frac{n-2}{n}}
+C\rho\left(\frac{\e}{\rho}\right)^{n}+C\rho\left(\frac{\e}{\delta}\right)^{n-2}\notag
\\
&+C\sum_{a,b=1}^{n}\sum_{|\a|=2}^{d}|h_{ab,\a}|^2\e^n\int_{\tilde{B}_{x_0, \rho}}(\e+|x|)^{2|\a|+2-2n}dx.\notag
\end{align}

Recall that Lemma \ref{Lemma4:notes3} says
\begin{equation}\label{Propo2:notes5:2}
-\int_{\tilde{D}_{x_0, \rho}}\frac{4(n-1)}{n-2}\U\d_\nu\U d\sigma\leq -\tilde{c}\left(\frac{\e}{\delta}\right)^{n-2}
\end{equation}
if $0<\epsilon<\delta<K\rho$ and $0<\rho<P_2$, for $P_2$ small enough. 

Now it follows from Lemma \ref{corol10} that
$$
\lambda\e^{n-2}\sum_{a,b=1}^{n}\sum_{|\a|=2}^{d}|h_{ab,\a}|^2\int_{\tilde{B}_{x_0, \rho}(0)}(\e+|x|)^{2|\a|+2-2n}dx\leq \frac 14 \int_{B_\rho(0)}Q_{ab,c}Q_{ab,c}\,dx.
$$
We claim that we can choose $P_2>0$ possibly smaller such that 
\[ \int_{B_\rho(0)\backslash \tilde{B}_{x_0, \rho}}Q_{ab,c}Q_{ab,c}\,dx\leq C\rho^2\left(\frac{\epsilon}{\delta}\right)^{n-2}\]
for all $\rho<P_2$.
In fact, from Lemma \ref{Lemma5:notes3} we can choose $P_2$ small such that $$B_\rho(0)\backslash \tilde{B}_{x_0, \rho}\subset B_{\rho}(0)\backslash B_{\delta/\sqrt{C_0}}(0)$$ 
for any $\rho<P_2$. Then using $Q_{ab,c}Q_{ab,c}\leq C\e^{n-2}(\e+|x|)^{4-2n}$ we get
\begin{align*}
 \int_{B_{\rho}(0)\backslash \tilde{B}_{x_0, \rho}}Q_{ab,c}Q_{ab,c}\,dx&\leq C\e^{n-2}\int_{B_\rho(0)\backslash \tilde{B}_{x_0, \rho}}(\e+|x|)^{4-2n}dx\\
 &\leq C\e^{n-2}\rho^2\int_{\mathbb{R}^n\backslash B_{\delta/\sqrt{C_0}}}(\e+|x|)^{2-2n}dx
\leq C\e^{n-2}\rho^2\delta^{2-n}.
\end{align*}

In particular,
\begin{align}\label{eq:handQ}
\lambda\e^{n-2}\sum_{a,b=1}^{n}\sum_{|\a|=2}^{d}|h_{ab,\a}|^2\int_{\tilde{B}_{x_0, \rho}}&(\e+|x|)^{2|\a|+2-2n}dx
\leq \frac 14 \int_{\tilde{B}_{x_0, \rho}}Q_{ab,c}Q_{ab,c}dx+C\rho^2\left(\frac{\e}{\delta}\right)^{n-2}.
\end{align}

Now the result follows from Proposition \ref{Propo1:notes4} and estimates \eqref{Propo2:notes5:3}, \eqref{Propo2:notes5:1}, \eqref{Propo2:notes5:2} and \eqref{eq:handQ}.
\end{proof}
\begin{proposition}\label{Propo3:notes4}
 There exist $P_2$ and $K$ such that
 \begin{align*}
  &\int_{M}\left\{\frac{4(n-1)}{n-2}|d\bar{U}_{(x_0,\e)}|_{g_{x_0}}^2+R_{g_{x_0}}\bar{U}_{(x_0,\e)}^2\right\}dv_{g_{x_0}}
+\int_{\d M}2\cmedia_{g_{x_0}}\bar{U}_{(x_0,\e)}^2d\sigma_{g_{x_0}}
\\
&\leq \Y\left(\int_M \bar{U}_{(x_0,\e)}^\crit dv_{g_{x_0}}\right)^{\frac{n-2}{n}}-\frac{\lambda}{4}\sum_{a,b=1}^{n}\sum_{|\a|=2}^{d}|h_{ab,\a}|^2\e^{n-2}\int_{\tilde{B}_{x_0, \rho}}(\e+|x|)^{2|\alpha|+2-2n}dx-\frac{\tilde{c}}{2}\left(\frac{\epsilon}{\delta}\right)^{n-2}.
 \end{align*}
for all $0<\e<\delta<K\rho$ and $0<\rho<P_2$.
\end{proposition}
\begin{proof}
We have
 \begin{align*}
  \int_{M\backslash \tilde{B}_{x_0, \rho}}\left\{\frac{4(n-1)}{n-2}|d\bar{U}_{(x_0,\e)}|_{g_{x_0}}^2+R_{g_{x_0}}\bar{U}_{(x_0,\e)}^2\right\}dv_{g_{x_0}}
+\int_{\partial M\backslash \tilde{D}_{x_0, \rho}}2\cmedia_{g_{x_0}}\bar{U}_{(x_0,\e)}^2d\sigma_{g_{x_0}}\leq C\left(\frac{\e}{\rho}\right)^{n-2}.
 \end{align*}
As in the proof of Proposition \ref{Propo1:notes4},
$$
\int_{\tilde{D}_{x_0, \rho}}2\cmedia_{g_{x_0}}\bar{U}_{(x_0,\e)}^2d\sigma_{g_{x_0}}
\leq C\rho\left(\frac{\e}{\delta}\right)^{n-2}.
$$
The result now follows from Proposition \ref{Propo2:notes4} and the fact that $\det(g_{x_0})(x)=1+O(|x|^{2d+2})$.
\end{proof}

\bp[Proof of Proposition \ref{Propo:energy:test:tubular}]
Let $P_2$ and $K$ be as in  Proposition  \ref{Propo3:notes4}. Choose $P_2$ maybe smaller such that $P_2<K$. Given $\rho_B\leq P_2$ choose $K'\leq \rho_B$ and $\delta'_0\in (0,K'\rho_B)$. Observe that, in particular, one has $\delta'_0<\rho_B^2$ and $\delta'_0<K\rho_B$. By Proposition  \ref{Propo3:notes4}, the inequality we want to prove holds for all $0<\e<\delta<\delta'_0$ and $0<\rho=\rho_B\leq P_2$, where $\delta=d_{g_{x_0}}(x_0,\d M)$.

Now choose $C_B=C_B(M,g_0)$ such that $C_B^{-1}\delta \leq d_{g_0}(x_0,\d M)\leq C_B\delta$,
and take any $\delta_0<C_B\delta'_0$. Then, because $\delta'_0<\rho_B^2$, we have 
$$
\delta_0<C_B\rho_B^2.
$$
For any $\e<C_B^{-1}d_{g_0}(x_0,\d M)$ we have $\e<C_B^{-1}d_{g_0}(x_0, \d M)<\delta<\delta'_0$
and the inequality in Proposition \ref{Propo:energy:test:tubular} holds.
\ep

We finally prove some results for later use.
\begin{proposition}\label{Propo3:notes5}
For $x\in M$ , $\e<\rho$ and $\delta\leq C\rho^2$,
 \begin{align*}
 &\left|\ct\Delta_{g_{x_0}} \bu_{(x_0,\e)}-R_{g_{x_0}}\bu_{(x_0,\e)}+\cminfbar\bu_{(x_0,\e)}^{\frac{n+2}{n-2}}\right|(x)
\\
 &\hspace{0.5cm}\leq 
C\rho^2\left(\frac{\epsilon}{\epsilon^2+|x|^2}\right)^{\frac{n-2}{2}}1_{\tilde B_{x_0, \rho}}(x)
+C\left(\frac{\epsilon}{\epsilon^2+|x|^2}\right)^{\frac{n+2}{2}}1_{M\backslash\tilde B_{x_0,\rho}}(x)
\\
&\hspace{0.7cm}+C(\e^{\frac{n+2}{2}}\rho^{-2-n}+\e^{\frac{n-2}{2}}\rho^{1-n}|\log\rho|)1_{\tilde  B_{x_0,2\rho}\backslash\tilde  B_{x_0,\rho}}(x).
 \end{align*}
\end{proposition}
\bp
The proof goes like that of Proposition \ref{Propo1:notes5} with $I_1,I_2,I_3,I_4$ being the same. 
Observing that we are using normal coordinates, we have
\[|I_3|\leq C\rho^2\left(\frac{\e}{\e^2+|x|^2}\right)^{\frac{n-2}{2}}1_{\tilde{B}_{x_0, 2\rho}}.\]
Using \eqref{estim:G:simpl} we obtain $|I_2|\leq C\e^{\frac{n-2}{2}}\rho^{1-n}|\log\rho| 1_{\tilde{B}_{x_0, 2\rho}\backslash \tilde{B}_{x_0, \rho}}+C\e^{\frac{n-2}{2}}\delta\rho^{-1-n} 1_{\tilde{B}_{x_0, 2\rho}\backslash \tilde{B}_{x_0, \rho}}$, the $|\log\rho|$ being necessary only in dimension $n=3$.

With the same estimate for $I_1$ and $I_4$ as in Proposition \ref{Propo1:notes5}, we get the result.
\ep
\begin{proposition}\label{Propo4:notes5} For $x\in \d M$, $\e<\rho$  and $\delta\leq C\rho^2$,
 \begin{align*}
 \Big|\frac{2(n-1)}{n-2}\frac{\partial}{\partial\eta_{g_{x_0}}}\bu_{(x_0,\e)}&-H_{g_{x_0}}\bu_{(x_0,\e)}\Big|(x)
\\
&\leq C\frac{\delta}{\e}\left(\frac{\e}{\e^2+|x|^2}\right)^{\frac n2}1_{\tilde{D}_{x_0,2\rho}}(x)+
C\left(\frac{\epsilon}{\epsilon^2+|x|^2}\right)^{\frac{n-2}{2}}1_{\tilde{D}_{x_0,2\rho}}(x)
\\
&+C(\e^{\frac{n+2}{2}}\rho^{-1-n}+\e^{\frac{n-2}{2}}\rho^{2-n}|\log\rho|)1_{\tilde{D}_{x_0,2\rho}\backslash\tilde{D}_{x_0,\rho}}(x).
 \end{align*}
\end{proposition}
\begin{proof}
Observe that, on $\d M$,
\begin{align*}
 \frac{\partial \bu_{(x_0,\e)}}{\partial\eta_{g_{x_0}}}-\frac{n-2}{2(n-1)}H_{g_{x_0}}\bu_{(x_0,\e)}
=&\frac{\partial \chi_{\rho}}{\partial \eta_{g_{x_0}}}(\U+\phi-\e^{\frac{n-2}{2}}|x|^{2-n})+\frac{\partial \chi_{\rho}}{\partial \eta_{g_{x_0}}}\e^{\frac{n-2}{2}}(|x|^{2-n}-G_{x_0})
\\
 &+\chi_{\rho}\frac{\partial}{\partial \eta_{g_{x_0}}}(\U+\phi)-\frac{n-2}{2(n-1)}\chi_{\rho} H_{g_{x_0}}(\U+\phi)
\\
&+(1-\chi_{\rho})\e^{\frac{n-2}{2}}\left(\frac{\partial G_{x_0}}{\partial\eta_{g_{x_0}}}-\frac{n-2}{2(n-1)}H_{g_{x_0}}G_{x_0}\right),
\end{align*}
where the last term is zero by the definition of $G_{x_0}$.
Set
\begin{align*}
 &J_1=\frac{\partial \chi_{\rho}}{\partial \eta_{g_{x_0}}}(\U+\phi-\e^{\frac{n-2}{2}}|x|^{2-n}), &J_2=\frac{\partial \chi_{\rho}}{\partial \eta_{g_{x_0}}}\e^{\frac{n-2}{2}}(|x|^{2-n}-G_{x_0}),\\
 &J_3=\chi_{\rho}\frac{\partial\U}{\partial \eta_{g_{x_0}}}, &J_4=\chi_{\rho}\left(\frac{\partial \phi}{\partial \eta_{g_{x_0}}}- \frac{n-2}{2(n-1)}H_{g_{x_0}}(\U+\phi)\right).
\end{align*}
Recall (\ref{eq:I_1}) to bound
\begin{align*}
 |J_1|&\leq \Big|\frac{\partial \chi_{\rho}}{\partial \eta_{g_{x_0}}}\Big|\big(|\U-\e^{\frac{n-2}{2}}|x|^{2-n}\big|+|\phi|\big)
\leq C(\e^{\frac{n+2}{2}}\rho^{-1-n}+\e^{\frac{n-2}{2}}\rho^{3-n})1_{\tilde D_{x_0,2\rho}\backslash \tilde D_{x_0,\rho}}.
\end{align*}
For $J_2$, we use the properties \eqref{estim:G:simpl} of the Green function and the hypothesis $\delta\leq C\rho^2$ to obtain
\[|J_2|
\leq 
\e^{\frac{n-2}{2}}\Big|\frac{\partial \chi_{\rho}}{\partial \eta_{g_{x_0}}}\Big|\big||x|^{2-n}-G_{x_0}\big|
\leq C\e^{\frac{n-2}{2}}\rho^{2-n}|\log\rho|1_{\tilde D_{x_0,2\rho}\backslash \tilde D_{x_0,\rho}}.\]
In order to estimate $J_3$, let us calculate $ \d\U/\d\eta_{g_{x_0}}$. Suppose $x=(\bar x, \gamma(\bar x))\in \tilde D_{x_0,\rho}$, then
\begin{align}\label{ineq:Un}
 \d\U/\d\eta_{g_{x_0}}(x)&=-(n-2)\e^{\frac{n-2}{2}}(\e^2+|x|^2)^{-\frac n2}x_a\eta^a(x)
 \\
 &=-(n-2)\e^{\frac{n-2}{2}}(\e^2+|x|^2)^{-\frac n2}(\gamma(\bx)+(\eta^a(x)-\delta_{an})x_a).\notag
\end{align}
Recall the properties (\ref{eq:f_property}) and (\ref{eq:eta_property}) of $\gamma$ and $\eta_{g_{x_0}}$. So,
\begin{align*}
\big|\d\U/\d\eta_{g_{x_0}}\big|(x)
\leq C\e^{\frac{n-2}{2}}(\e^2+|x|^2)^{-\frac n2}(\delta+C|\bx|^2)
\leq C\frac{\delta}{\e}\left(\frac{\e}{\e^2+|x|^2}\right)^{\frac n2}+C\left(\frac{\e}{\e^2+|x|^2}\right)^{\frac{n-2}{2}}
\end{align*}
for $x\in \tilde D_{x_0,\rho}$.
Consequently,
\[|J_3|\leq C\frac{\delta}{\e}\left(\frac{\e}{\e^2+|x|^2}\right)^{\frac n2}1_{\tilde D_{x_0,2\rho}}+C\left(\frac{\e}{\e^2+|x|^2}\right)^{\frac{n-2}{2}}1_{\tilde D_{x_0,2\rho}}.\]
Easily we can get
\[
|J_4|\leq C\chi_{\rho}\big(\Big|\frac{\partial \phi}{\partial \eta_{g_{x_0}}}\Big|+\U+|\phi|\big)
\leq 
C\left(\frac{\e}{\e^2+|x|^2}\right)^{\frac{n-2}{2}}1_{\tilde D_{x_0,2\rho}}.
\]
Combining all the results, we get the conclusion.
\end{proof}
\begin{proposition}\label{Propo5:notes5}
For $x\in \d M$, $\e<\rho$  and $\delta\leq C\rho^2$,
 \begin{align*}
 \Big(\frac{2(n-1)}{n-2}&\frac{\partial}{\partial \eta_{g_{x_0}}}\bu_{(x_0,\e)}-H_{g_{x_0}}\bu_{(x_0,\e)}\Big)(x)
\\
&\geq 
-C\left(\frac{\e}{\e^2+|x|^2}\right)^{\frac{n-2}{2}}1_{\tilde{D}_{x_0,2\rho}}(x)
-C(\e^{\frac{n+2}{2}}\rho^{-1-n}+\e^{\frac{n-2}{2}}\rho^{2-n}|\log\rho|)1_{\tilde{D}_{x_0,2\rho}\backslash \tilde{D}_{x_0, \rho}}(x).
 \end{align*}
\end{proposition}
\begin{proof}
 By (\ref{ineq:Un}) we have
\begin{align*}
\chi_{\rho}\d\U/\d\eta_{g_{x_0}}
\geq \chi_{\rho}(n-2)(\e^2+|x|^2)^{-\frac{n}{2}}(\delta-C|\bx|^2)
\geq  -C\e^{\frac{n-2}{2}}(\e^2+|x|^2)^{\frac {2-n}{2}}1_{\tilde{D}_{x_0,2\rho}}.
\end{align*}
Now the result follows as in Proposition \ref{Propo4:notes5}.
\end{proof}


\subsection{Type C test functions ($\bar u_{C;(x_0,\e)}$)}\label{sub:sec:deftestfunct:int}
Our test functions in this case are the ones in \cite{brendle-invent}, which are controlled by $\Y$ the same way as in that paper.

Recall that we assume that the background metric $g_0$ on $M$ satisfies $\cmz\equiv 0$ on $\d M$.
Fix $x_0\in M\backslash M_{\delta_0}$ and let $\psi_{x_0}: B_{2\rho}(0)\subset\R^n\to B_{2\rho}(x_0)\subset M$ be normal coordinates centered at $x_0$, where $\rho$ is small such that $0<\rho\leq \delta_0/4$.
As in Subsection \ref{sub:sec:deftestfunct:tubular},  we choose a conformal metric $g_{x_0}=f_{x_0}^{\frac{4}{n-2}}g_0$ such that $\text{det}(g_{x_0})(x)=1+O(|x|^{2d+2})$ in normal coordinates centered at $x_0$, still denoted by $\psi_{x_0}$. 
We assume $f_{x_0}\equiv 1$ in $M\backslash B_{2\rho}(x_0)$, which implies $H_{g_{x_0}}\equiv 0$ on $\d M$.

Define $\phi$ as in Subsection \ref{sub:sec:deftestfunct:tubular} and set
\ba\label{def:test:func:int}
\ubarrho(x)
=&\left(\frac{4n(n-1)}{\cminfbar}\right)^{\frac{n-2}{4}}\chi_{\rho}(\psi_{x_0}^{-1}(x))\big(\U(\psi_{x_0}^{-1}(x))+\phi(\psi_{x_0}^{-1}(x))\big)
\\
&\hspace{0.1cm}+\left(\frac{4n(n-1)}{\cminfbar}\right)^{\frac{n-2}{4}}\e^{\frac{n-2}{2}}\big(1-\chi_{\rho}(\psi_{x_0}^{-1}(x))\big)G_{x_0}(x)\,\notag
\end{align}
for $x\in M$. Here, $G_{x_0}$ is the Green's function of the conformal Laplacian $L_{g_{x_0}}=\Delta_{g_{x_0}}-\frac{n-2}{4(n-1)}R_{g_{x_0}}$, with pole at $x_0\in M\backslash M_{\delta_0}$, boundary condition (\ref{eq:G:bordo})
and the normalization $\lim_{|y|\to 0}|y|^{n-2}G_{x_0}(\psi_{x_0}(y))=1$. This function, obtained in Proposition \ref{green:point}, satisfies 
\ba
|G_{x_0}(\psi_{x_0}(y))-|y|^{2-n}|&\leq
C\sum_{i,j=1}^{n-1}\sum_{|\a|=1}^{d}|h_{ij,\a}||y|^{|\a|+2-n}+
\begin{cases}
C|y|^{d+3-n},\:\:\:\text{if}\:n\geq 5,
\\
C(1+|\log|y||),\:\:\:\text{if}\:n=3,4,
\end{cases}
\\
\left|\frac{\d}{\d y_b}(G_{x_0}(\psi_{x_0}(y))-|y|^{2-n})\right|
&\leq
C\sum_{i,j=1}^{n-1}\sum_{|\a|=1}^{d}|h_{ij,\a}||y|^{|\a|+1-n}+C|y|^{d+2-n}\,,\notag
\end{align}
for some $C=C(M,g_0,\delta_0)$ for all $b=1,...,n$ and $x_0\in M\backslash M_{\delta_0}$.

We define the test function
\begin{equation}\label{def:test:func:u:int}
\bar u_{C;(x_0,\e)}=f_{x_0}\ubarrho.
\end{equation}
Observe that this function also depends on the radius $\rho$ above, which will be fixed later in Section \ref{sec:blowup}. Such constant will also be referred to as $\rho_C$ in order to avoid confusion with test functions of the other subsections. 

For later use we observe that $\frac{\d}{\d\eta_{g_0}}\ubarrho=B_{g_0}\bar u_{C;(x_0,\e)}=B_{g_{x_0}}\ubarrho=0$ on $\d M$.

Our main result in this subsection is the following:
\begin{proposition}\label{Propo:energy:test:int}
Under the hypothesis of Theorem \ref{main:thm}, there exists $P_3=P_3(M,g_0, \delta_0)$ such that
\ba
&\frac{\int_M\left\{\frac{4(n-1)}{n-2}|d \bar u_{C;(x_0,\e)}|_{g_0}^2+R_{g_0}\bar u_{C;(x_0,\e)}^2\right\}\dv_{g_0}}{\left(\int_{M}\bar u_{C;(x_0,\e)}^{\crit}\dv_{g_0}\right)^{\frac{n-2}{n}}}\notag
\\
&\hspace{0.5cm}=\frac{\int_M\left\{\frac{4(n-1)}{n-2}|d \ubar|_{g_{x_0}}^2
+R_{g_{x_0}}\ubar^2\right\}\dv_{g_{x_0}}
+\int_{\d M}2\cmedia_{g_{x_0}}\ubar^2\ds_{g_{x_0}}}{\left(\int_{M}\ubar^{\crit}\dv_{g_{x_0}}\right)^{\frac{n-2}{n}}}\notag
\\
&\hspace{0.5cm}\leq \Y\notag
\end{align}
for all $x_0\in M\backslash M_{\delta_0}$ and $0<2\e<\rho_C<P_3$.
\end{proposition}
\bp
Choose $P_3$ small such that for any $x_0\in M\backslash M_{\delta_0}$ we have $d_{g_{x_0}}(x_0, \d M)> 2P_3$.
Choosing $P_3$ smaller if necessary (also depending on $\delta_0$ because of the above estimates for $G_{x_0}$) the result is Corollary 3 and Proposition 19 in \cite{brendle-invent} with some obvious modifications, by making use of Theorem \ref{pmt}.
\ep

For later use we state the following result, which is proved as  Proposition \ref{Propo3:notes5}:
\begin{proposition}\label{Propo3':notes5}
We can choose $P_3=P_3(M,g_0,\delta_0)$ maybe smaller such that there is $C=C(M,g_0)$ satisfying
 \begin{align*}
 &\left|\ct\Delta_{g_{x_0}} \bu_{(x_0,\e)}-R_{g_{x_0}}\bu_{(x_0,\e)}+\cminfbar\bu_{(x_0,\e)}^{\frac{n+2}{n-2}}\right|
\\
 &\hspace{0.5cm}\leq C\rho^2\left(\frac{\epsilon}{\epsilon^2+|x|^2}\right)^{\frac{n-2}{2}}1_{B_{2\rho}(0)}+C\left(\frac{\epsilon}{\epsilon^2+d_{g_{x_0}}(x,x_0)^2}\right)^{\frac{n+2}{2}}1_{M\backslash B_{\rho}(0)}
+C(\e^{\frac{n+2}{2}}\rho^{-2-n}+\e^{\frac{n-2}{2}}\rho^{3/4-n}|\log\rho|)1_{B_{2\rho}(0)\backslash B_{\rho}(0)}
 \end{align*}
for all $x_0\in M\backslash M_{\delta_0}$ and $\e<\rho\leq P_3$.
\end{proposition}
\bp
As in Proposition \ref{Propo3:notes5}, the proof follows the lines of Proposition \ref{Propo1:notes5}, but the term $I_2$ is estimated by
$|I_2|\leq C\e^{\frac{n-2}{2}}\rho^{1-n}|\log\rho|$, where $C$ depends on $\delta_0$.
Choose $P_3<C^{-4}$.
\ep


\subsection{Further estimates}\label{sub:sec:further}

The results of this subsection are consequences of what was proved in Subsections \ref{sub:sec:deftestfunct:boundary}, \ref{sub:sec:deftestfunct:tubular} and \ref{sub:sec:deftestfunct:int}.

In this subsection, unless otherwise stated, if $x_0\in \d M$, $x_0\in M_{\delta_0}\backslash \d M$ or $x_0\in M\backslash M_{2\delta_0}$, $\bar{u}_{(x_0,\e)}$ will stand for $\bar u_{A;(x_0,\e)}$, $\bar u_{B;(x_0,\e)}$ or $\bar u_{C;(x_0,\e)}$, respectively. If $x_0\in M_{2\delta_0}\backslash M_{\delta_0}$, $\bar{u}_{(x_0,\e)}$ will stand for $\bar u_{B;(x_0,\e)}$ and $\bar u_{C;(x_0,\e)}$, the results below holding for either. By the "radius" $\rho$ of  $\bar{u}_{(x_0,\e)}$, we mean $\rho_A$, $\rho_B$ or $\rho_C$, if $\bar{u}_{(x_0,\e)}=\bar u_{A;(x_0,\e)}$, $\bar{u}_{(x_0,\e)}=\bar u_{B;(x_0,\e)}$ or $\bar{u}_{(x_0,\e)}=\bar u_{C;(x_0,\e)}$, respectively.

We observe that whenever $\bar{u}_{(x_0,\e)}=\bar u_{B;(x_0,\e)}$ we have $d_{g_0}(x_0,\d M)\leq \delta_0\leq C\rho^2$, according to Proposition \ref{Propo:energy:test:tubular}, because $x_0\in M_{\delta_0}\backslash \d M$ in this case. Hence, we can make use of Propositions \ref{Propo3:notes5}, \ref{Propo4:notes5} and \ref{Propo5:notes5}.

\begin{corollary}\label{Corol1:notes5}
There exists $C=C(M,g_0)$ such that, for $\e<\rho$,
\begin{align*}
 &\left|\ct\Delta_{g_{0}} \bar{u}_{(x_0,\e)}-R_{g_{0}}\bar{u}_{(x_0,\e)}+\cminfbar\bar{u}_{(x_0,\e)}^{\frac{n+2}{n-2}}\right|
\\
 &\hspace{0.5cm}
\leq C\rho^{-1/2}\left(\frac{\epsilon}{\epsilon^2+d_{g_0}(x,x_0)^2}\right)^{\frac{n-2}{2}}(\epsilon^2+d_{g_0}(x,x_0)^2)^{-\frac{1}{2}}1_{B_{4\rho}(x_0)}
+C\left(\frac{\epsilon}{\epsilon^2+d_{g_0}(x,x_0)^2}\right)^{\frac{n+2}{2}}1_{M\backslash B_{\rho/2}(x_0)}.
\end{align*}
\end{corollary}
\bp
It is a consequence of Propositions \ref{Propo1:notes5}, \ref{Propo3:notes5} and \ref{Propo3':notes5}.
\ep
\begin{corollary}\label{Corol2:notes5}
\footnote{For types A and B test functions in dimensions $n\geq 5$, the coefficient $\rho^{1/2}$ in this inequality can be improved to $\rho$. Indeed, $\rho$ was worsen to $\rho^{1/2}$ due to the $\log \rho$ terms in Propositions \ref{Propo1:notes5} and \ref{Propo3:notes5}, which are necessary only for $n=3$ or $4$, as observed in the footnote in Proposition \ref{Propo1:notes5}.}
There exists $C=C(M,g_0)$ such that, if $\rho$ is the radius of $\bar{u}_{(x_2,\e_2)}$ and $\e_1\leq \e_2<\rho$, we have
\begin{align*}
 \int_M \bar{u}_{(x_1,\e_1)}&\left|\ct\Delta_{g_0} \bar{u}_{(x_2,\e_2)}-R_{g_0}\bar{u}_{(x_2,\e_2)}+\cminfbar\bar{u}_{(x_2,\e_2)}^{\frac{n+2}{n-2}}\right|dv_{g_0}
\\
&\leq C\left(\rho^{1/2}+\frac{\e^2_2}{\rho^2}\right)\left(\frac{\e_1\e_2}{\e^2_2+d_{g_0}(x_1,x_2)^2}\right)^{\frac{n-2}{2}}.
 \end{align*}
\end{corollary}
\begin{proof}
As in \cite[Lemma B.5]{brendle-flow} we get
\begin{align}\label{Corol2:notes5:1}
\int_{\{d_{g_0}(y,x_2)\geq \rho/2\}}
&\lp\frac{\e_1}{\e_1^2+d_{g_0}(x_1,y)^2}\rp^{\frac{n-2}{2}}
\lp\frac{\e_2}{\e_2^2+d_{g_0}(x_2,y)^2}\rp^{\frac{n+2}{2}}
dv_{g_0}
\leq C\frac{\e_2^2}{\rho^2}\lp\frac{\e_1\e_2}{\e_2^2+d_{g_0}(x_1,x_2)^2}\rp^{\frac{n-2}{2}}.
\end{align}
We claim that
\begin{align}\label{Corol2:notes5:2}
\int_{\{d_{g_0}(y,x_2)\leq 4\rho\}}
&\lp\frac{\e_1}{\e_1^2+d_{g_0}(x_1,y)^2}\rp^{\frac{n-2}{2}}
\lp\frac{\e_2}{\e_2^2+d_{g_0}(x_2,y)^2}\rp^{\frac{n-2}{2}}(\e_2^2+d_{g_0}(x_2,y)^2)^{-\frac{1}{2}}
dv_{g_0}
\\
&\leq C\rho\lp\frac{\e_1\e_2}{\e_2^2+d_{g_0}(x_1,x_2)^2}\rp^{\frac{n-2}{2}}.\notag
\end{align}

Set 
$$A=\{2d_{g_0}(x_1,y)\leq \e_2+d_{12}\}\cap\{d_{g_0}(y,x_2)\leq 4\rho\}$$ and $$B=\{2d_{g_0}(x_1,y)\geq \e_2+d_{12}\}\cap\{d_{g_0}(y,x_2)\leq 4\rho\}$$ 
where $d_{12}=d_{g_0}(x_1,x_2)$.
Observe that on $A$ we have
$$
\e_2+d_{g_0}(y,x_2)\geq \e_2+d_{12}-d_{g_0}(y,x_1)\geq \frac{1}{2}(\e_2+d_{12})\geq d_{g_0}(y,x_1)
$$
$$
\text{and}\:\:\:\:\:\: d_{g_0}(y,x_1)\leq \frac{1}{2}(\e_2+d_{12})\leq \e_2+d_{g_0}(y,x_2)\leq 5\rho.
$$
Then
\begin{align}\label{Corol2:notes5:3}
\int_A
&\lp\frac{\e_1}{\e_1^2+d_{g_0}(x_1,y)^2}\rp^{\frac{n-2}{2}}
\lp\frac{\e_2}{\e_2^2+d_{g_0}(x_2,y)^2}\rp^{\frac{n-2}{2}}(\e_2^2+d_{g_0}(x_2,y)^2)^{-\frac{1}{2}}
dv_{g_0}
\\
&\leq C\lp\frac{\e_1\e_2}{\e_2^2+d_{12}^2}\rp^{\frac{n-2}{2}}
\int_{\{d_{g_0}(y,x_1)\leq 5\rho\}}
(\e_1^2+d_{g_0}(x_1,y)^2)^{\frac{2-n}{2}}d_{g_0}(x_1,y)^{-1}
dv_{g_0}\notag
\\
&\leq C\lp\frac{\e_1\e_2}{\e_2^2+d_{12}^2}\rp^{\frac{n-2}{2}}
\int_{\{d_{g_0}(y,x_1)\leq 5\rho\}}d_{g_0}(x_1,y)^{1-n}
dv_{g_0}\notag
\end{align}

On the other hand,
\begin{align}\label{Corol2:notes5:4}
\int_B
&\lp\frac{\e_1}{\e_1^2+d_{g_0}(x_1,y)^2}\rp^{\frac{n-2}{2}}
\lp\frac{\e_2}{\e_2^2+d_{g_0}(x_2,y)^2}\rp^{\frac{n-2}{2}}(\e_2^2+d_{g_0}(x_2,y)^2)^{-\frac{1}{2}}
dv_{g_0}
\\
&\leq C\lp\frac{\e_1\e_2}{\e_2^2+d_{12}^2}\rp^{\frac{n-2}{2}}
\int_{\{d_{g_0}(y,x_2)\leq 4\rho\}}d_{g_0}(x_2,y)^{1-n}
dv_{g_0}.\notag
\end{align}
The estimate \eqref{Corol2:notes5:2} follows from \eqref{Corol2:notes5:3} and \eqref{Corol2:notes5:4} observing that the integrals on the right sides of those inequalities are bounded by $C\rho$.

The result now follows from  \eqref{Corol2:notes5:1}, \eqref{Corol2:notes5:2} and Corollary \ref{Corol1:notes5}.
\end{proof}
\begin{corollary}\label{Corol3:notes5}
\footnote{
Similarly to the footnote in Corollary \ref{Corol2:notes5}, for types A and B test functions the coefficient $\rho^{1/2}$ can be improved to $\rho$ if $n\geq 5$.
}
There exists $C=C(M,g_0)$ such that, if $\rho$ is the radius of $\bar{u}_{(x_2,\e_2)}$ and $\e_1\leq \e_2<\rho$,
$$
\int_{\partial M}\bar{u}_{(x_1,\e_1)}\frac{\partial}{\partial \eta_{g_0}}\bar{u}_{(x_2,\e_2)}d\sigma_{g_0}
\geq -C\left(\rho^{1/2}+\frac{\e_2}{\rho}\right)\lp\frac{\e_1\e_2}{\e_2^2+d_{g_0}(x_1,x_2)^2}\rp^{\frac{n-2}{2}}.
$$
\end{corollary}
\begin{proof} 
Observe that the above integral vanishes when $\bar{u}_{(x_2,\e_2)}$ is a type C test function. For type B test functions we obtain
\begin{align*}
\frac{\d}{\d\eta_{g_{x_2}}}\bar U_{(x_2,\e_2)}-\frac{n-2}{2(n-1)}H_{g_{x_2}}\bar U_{(x_2,\e_2)}
\geq 
-C\left(\frac{\e}{\e^2+|x|^2}\right)^{\frac{n-2}{2}}\rho^{-1/2}1_{\tilde D_{x_2,2\rho}}
-C\left(\frac{\e}{\e^2+|x|^2}\right)^{\frac{n}{2}}1_{\tilde D_{x_2,2\rho}\backslash\tilde D_{x_2,\rho}}
\end{align*}
from Proposition \ref{Propo5:notes5}. Then, using \eqref{propr:B} and \eqref{def:test:func:u:tubular}, we estimate
\begin{align*}
\frac{\partial}{\partial \eta_{g_0}}\bar{u}_{(x_2,\e_2)}
\geq& -C\rho^{-1/2}\lp\frac{\e_2}{\e_2^2+d_{g_0}(x_2,y)^2}\rp^{\frac{n-2}{2}}1_{\{d_{g_0}(y,x_2)\leq 4\rho\}\cap \partial M}
\\
&-C\lp\frac{\e_2}{\e_2^2+d_{g_0}(x_2,y)^2}\rp^{\frac{n}{2}}1_{\{d_{g_0}(y,x_2)\geq \rho/2\}\cap \partial M}.
\end{align*}
The same (actually a better) estimate as above can be obtained for type A test functions by means of Proposition \ref{Propo2:notes5}.

As in \cite[p.274-275]{brendle-flow} we can prove
\begin{align*}
\int_{\{d_{g_0}(y,x_2)\leq 4\rho\}\cap \partial M}
&\lp\frac{\e_1}{\e_1^2+d_{g_0}(x_1,y)^2}\rp^{\frac{n-2}{2}}
\lp\frac{\e_2}{\e_2^2+d_{g_0}(x_2,y)^2}\rp^{\frac{n-2}{2}}
d\sigma_{g_0}
\leq C\rho\lp\frac{\e_1\e_2}{\e_2^2+d_{g_0}(x_1,x_2)^2}\rp^{\frac{n-2}{2}}
\end{align*}
and
\begin{align*}
\int_{\{d_{g_0}(y,x_2)\geq \rho/2\}\cap \partial M}
\lp\frac{\e_1}{\e_1^2+d_{g_0}(x_1,y)^2}\rp^{\frac{n-2}{2}}
\lp\frac{\e_2}{\e_2^2+d_{g_0}(x_2,y)^2}\rp^{\frac{n}{2}}
d\sigma_{g_0}
\leq C\frac{\e_2}{\rho}\lp\frac{\e_1\e_2}{\e_2^2+d_{g_0}(x_1,x_2)^2}\rp^{\frac{n-2}{2}}.
\end{align*}

The result now follows.
\end{proof}
\begin{corollary}\label{Corol6:notes5}
For $\e<\rho$ we have 
 \begin{align*}
  &\left(\int_M \Big|\ct\Delta_{g_0}\bar{u}_{(x_0,\e)}-R_{g_0}\bar{u}_{(x_0,\e)}+\cminfbar\bar{u}_{(x_0,\e)}^{\frac{n+2}{n-2}}\Big|^{\frac{2n}{n+2}}dv_{g_0}\right)^{\frac{n+2}{2n}}\leq C\left(\frac{\e}{\rho}\right)^{ \frac{n+2}{2}}
+C
\begin{cases}
\e\rho^{-1/2}\quad &n\geq 5,
\\
\e\rho^{-1/2}|\log(\rho/\e)|\quad &n=4,
\\
\e^{1/2}\quad &n=3.
\end{cases}
 \end{align*}
\end{corollary}
\bp
The result follows easily from Corollary \ref{Corol1:notes5}.
\ep
\begin{corollary}\label{Corol5:notes5}
If  $\bar{u}_{(x_0,\e)}=\bar u_{B;(x_0,\e)}$ we have
 \begin{align*}
  \lp\int_{\partial M}\Big|\frac{2(n-1)}{n-2}\frac{\partial}{\partial \eta_{g_0}}\bar{u}_{(x_0,\e)}-H_{g_0}\bar{u}_{(x_0,\e)}\Big|^{\frac{2(n-1)}{n}}d\sigma_{g_0}\rp^{\frac{n}{2(n-1)}}\leq\begin{cases}
        \displaystyle C\lp\frac{\e}{\delta}\rp^{\frac{n-2}{2}}|\log \rho|+\frac{\e}{\rho}\quad &n\geq 5,\\
        \displaystyle C\lp\frac{\e}{\delta}\rp|\log \rho|+\frac{\e}{\rho}|\log(\rho/\e)|\quad &n=4,\\
        \displaystyle C\lp\frac{\e}{\delta}\rp^{1/2}|\log \rho|+C\lp\frac{\e}{\rho}\rp^{1/2}\quad &n=3,
       \end{cases}
 \end{align*}
for $\e<\rho$, where $\delta=d_{g_0}(x_0,\d M)$.
\end{corollary}
\bp
From Proposition \ref{Propo4:notes5}, on $\d M$ we have
\begin{align*}
\Big|\frac{2(n-1)}{n-2}\frac{\partial}{\partial \eta_{g_{0}}} \bar{u}_{(x_0,\epsilon)}-H_{g_{0}}\bar{u}_{(x_0,\epsilon)}\Big|
&\leq C\frac{\delta }{\e}\left(\frac{\e}{\e^2+d_{g_0}(x,x_0)^2}\right)^{\frac{n}{2}}1_{\{d_{g_0}(x,x_0)\leq 4\rho\}}
\\
&+C\rho^{-1}\left(\frac{\e}{\e^2+d_{g_0}(x,x_0)^2}\right)^{\frac{n-2}{2}}1_{\{d_{g_0}(x,x_0)\leq 4\rho\}}.
 \end{align*}
Using $\delta\leq C\rho^2$, which in particular implies $\delta\leq C\rho$, the first term on the right side above is estimated by 
$ C(\delta/\e)^{(n-2)/2}(\e+d_{g_0}(x,x_0))^{-n/2}\displaystyle1_{\{d_{g_0}(x,x_0)\leq 4\rho\}}$, and the result follows easily.
\ep


\section{Blow-up analysis}\label{sec:blowup}
In this section, we carry out the blow-up analysis for sequences of solutions to the equations (\ref{eq:evol:u}) that will be necessary for the proof of Theorem \ref{main:thm}. Although the analysis goes along the lines of \cite[Sections 4, 5 and 6]{brendle-flow}, here we have to consider the possibility of both interior and boundary blow-up points, thus differing from the situation in \cite[Section 4]{almaraz5}. As we will see in Proposition \ref{Propo4.1} below, type A test functions are used to approximate solutions near boundary blow-up points. As for interior blow-up points, we make use of type B test functions if those points accumulate on the boundary, and type C ones otherwise.
\begin{remark}\label{choosing:bubbles}
Before proceeding to the blow-up analysis, we observe that one can choose $\rho_A$, $\rho_B$ and $\rho_C$ in Propositions \ref{Propo:energy:test},  \ref{Propo:energy:test:tubular} and \ref{Propo:energy:test:int} in such a way that the inequalities of those propositions hold the three at the same time. To that end, choose $\delta_0$ according to a small $\rho_B$ in Proposition \ref{Propo:energy:test:tubular} and then $\rho_C$ according to $\delta_0$ in Proposition \ref{Propo:energy:test:int}. Moreover, observe that given $C=C(M,g_0)$ one can always assume $\rho_A, \rho_B, \rho_C \leq C$. This last remark will be used in the proofs of  Propositions \ref{Propo5.6} and \ref{Propo6.15} below.
\end{remark}

Let $u(t)$, $t\geq 0$, be the solution of (\ref{eq:evol:u}) obtained in Section \ref{sec:prelim}, and let
$\{t_{\nu}\}_{\nu=1}^{\infty}$ be any sequence satisfying $\lim_{\nu\to\infty}t_{\nu}=\infty$. We set $\u=u(t_{\nu})$ and $\g=g(t_{\nu})=\u^{\frac{4}{n-2}}g_0$. Then 
$$
\int_{M}\u^{\crit}dv_{g_0}=\int_{M}dv_{\g}=1\,,\:\:\:\:\text{for all}\:\nu\,.
$$
It follows from Corollary \ref{Corol3.2} that
$$ 
\int_{M}
\left|\frac{4(n-1)}{n-2}\Delta_{g_0}\u-\cez\u+\cminfbar\u^{\frac{n+2}{n-2}}\right|^{\conj}
\dv_{g_{0}}
=\int_{M}|\cev-\cminfbar|^{\conj}dv_{\g}
\to 0
$$
as $\nu\to\infty$.

The next proposition is an application of the decomposition result in \cite{pierotti-terracini}, which plays the same role here as \cite{struwe} did in \cite[Proposition 4.1]{brendle-flow}.
\begin{proposition}\label{Propo4.1}
After passing to a subsequence, there exist an integer $m\geq 0$, a smooth function $\uinf\geq 0$, and a sequence of $m$-tuplets  $\{(x_{k,\nu}^*,\e_{k,\nu}^*)_{1\leq k\leq m}\}_{\nu=1}^{\infty}$,  such that:

(i) The function $\uinf$ satisfies
\begin{equation}\notag
\begin{cases}
\frac{4(n-1)}{n-2}\D_{g_0}\uinf-\cez\uinf+\cminfbar\uinf^{\frac{n+2}{n-2}}=0\,,&\text{in}\:M\,,
\\
\d \uinf/\d\eta_{g_0}=0\,,&\text{on}\:\d M\,.
\end{cases} 
\end{equation}

(ii) For all $i\neq j$,
$$
\lim_{\nu\to\infty}\left\{\frac{\e_{i,\nu}^*}{\e_{j,\nu}^*}+\frac{\e_{j,\nu}^*}{\e_{i,\nu}^*}
+\frac{d_{g_0}(x_{i,\nu}^*,x_{j,\nu}^*)^2}{\e_{i,\nu}^*\e_{j,\nu}^*}\right\}=\infty\,.
$$

(iii) There are  integers $m_1, m_2$, with  $0\leq m_1\leq m_2\leq m$, such that $x_{k,\nu}^* \in \d M$ for $1\leq k\leq m_1$, $x_{k,\nu}^* \in M_{3\delta_0/2}\backslash \d M$ for $m_1+1\leq k\leq m_2$, $x_{k,\nu}^* \in M\backslash M_{3\delta_0/2}$ for $m_2+1\leq k\leq m$, and  
$$
\lim_{\nu\to\infty}d_{g_0}(x_{k,\nu}^*, \d M)/\e_{k,\nu}^*=\infty\:\:\:\:\:\: \text{if}\:\:k\geq m_1+1\,.
$$ 

(iv) If 
\begin{equation}\label{def:ubar}
\bar{u}_{(x_{k,\nu}^*,\e_{k,\nu}^*)}=
\begin{cases}
\bar{u}_{A; (x_{k,\nu}^*,\e_{k,\nu}^*)} &\text{if}\:k\leq m_1,
\\
\bar{u}_{B; (x_{k,\nu}^*,\e_{k,\nu}^*)} &\text{if}\:m_1+1\leq k\leq m_2,
\\
\bar{u}_{C; (x_{k,\nu}^*,\e_{k,\nu}^*)} &\text{if}\:k\geq m_2+1,
\end{cases}
\end{equation}
(see equations (\ref{def:test:func:u}), (\ref{def:test:func:u:tubular}) and  (\ref{def:test:func:u:int})) then
$$
\lim_{\nu\to\infty}
\big{\|}\u-\uinf-\sum_{k=1}^{m}\bar{u}_{(x_{k,\nu}^*,\e_{k,\nu}^*)}\big{\|}_{H^1(M)}= 0\,.
$$
\end{proposition}
\bp
By modifying the arguments in \cite[Section 3]{pierotti-terracini} to the case of Riemannian manifolds, we can prove the existence of $\uinf$ and $\bar{u}_{(x_{k,\nu}^*,\e_{k,\nu}^*)}$ satisfying (i) and (iv) except for, instead of using equations (\ref{def:ubar}), the $\bar{u}_{(x_{k,\nu}^*,\e_{k,\nu}^*)}$ are defined by
$$
\bar{u}_{(x_{k,\nu}^*,\e_{k,\nu}^*)}(x)=
\left(\frac{4n(n-1)}{\cminfbar}\right)^{\frac{n-2}{4}}(\e_{k,\nu}^*)^{-\frac{n-2}{2}}\chi_{\rho}\big(\psi_{x_{k,\nu}^*}^{-1}(x)\big)\, u\big((\e_{k,\nu}^*)^{-1}\psi_{x_{k,\nu}^*}^{-1}(x)\big).
$$
Here, $\psi_{x_{k,\nu}^*}$ are coordinates centered at $x_{k,\nu}^*$ and $u$ satisfies 
\begin{equation}\label{eq:u:int}
\Delta u+n(n-2)u^{\frac{n+2}{n-2}}=0\quad\text{in}\:\R^n
\end{equation}
if $\lim_{\nu\to\infty}d_{g_0}(x_{k,\nu}^*, \d M)/\e_{k,\nu}^*=\infty$, and 
\begin{equation}\label{eq:u:bd}
\begin{cases}
\Delta u+n(n-2)u^{\frac{n+2}{n-2}}=0&\text{in}\:\{y=(y_1,...,y_n)\,|\:y_n\geq t\},
\\
\frac{\d}{\d y_n}u=0&\text{on}\:\{y=(y_1,...,y_{n-1},t)\},
\end{cases}
\end{equation}
for some $t\in \R$ if $d_{g_0}(x_{k,\nu}^*, \d M)/\e_{k,\nu}^*$ is bounded.

Rearrange the indices and choose $m_1$ such that $k\geq m_1+1$ should \eqref{eq:u:int} holds and $k\leq m_1$ should \eqref{eq:u:bd} holds.

As in \cite[Lemma 3.3]{druet-hebey-robert}, we can prove that $u\geq 0$ and also that (ii) holds. The classification results in \cite{caffarell-gidas-spruck, li-zhu} (regularity was established in \cite{cherrier}) imply that $u(y)=\U(y-z)$ (see \eqref{eq:def:U}), for some $z=(z_1,...,z_n)\in \R^n$ (with $z_n=t$ if $k\leq m_1$). 

The points $x_{k,\nu}^*$ are now redefined as $\psi_{x_{k,\nu}^*}(z)$.\footnote{To see that changing the centers $x_{j,\nu}^*$ as above does not change the limit in (ii), we consider, for fixed $j$, new centers $\bar x_{j,\nu}^*$ satisfying  $d_{g_0}(x_{j,\nu}^*, \bar x_{j,\nu}^*)/\e_{j,\nu}^*\leq C$ (the term $\e_{j,\nu}^*$ in the quotient comes from the rescaling).
If the limit in (ii) holds with $\e_{j,\nu}^*/\e_{i,\nu}^*\to\infty$, that relation does not change after replacing the centers. So, let us assume $\e_{j,\nu}^*/\e_{i,\nu}^*\leq C$ without loss of generality. The triangle inequality gives
\begin{align*}
d_{g_0}(x_{i,\nu}^*, \bar x_{j,\nu}^*)^2
&\geq
\Big(d_{g_0}(x_{i,\nu}^*, x_{j,\nu}^*)-d_{g_0}(x_{j,\nu}^*, \bar x_{j,\nu}^*)\Big)^2 \geq\frac{1}{2}d_{g_0}(x_{i,\nu}^*, x_{j,\nu}^*)^2-Cd_{g_0}(x_{j,\nu}^*, \bar x_{j,\nu}^*)^2.
\end{align*}
Hence,
\begin{align*}
\frac{d_{g_0}(x_{i,\nu}^*, \bar x_{j,\nu}^*)^2}{\e_{i,\nu}^*\e_{j,\nu}^*}
&\geq 
\frac{1}{2}\frac{d_{g_0}(x_{i,\nu}^*, x_{j,\nu}^*)^2}{\e_{i,\nu}^*\e_{j,\nu}^*}
-C\frac{\e_{j,\nu}^*}{\e_{i,\nu}^*}\left(\frac{d_{g_0}(x_{j,\nu}^*, \bar x_{j,\nu}^*)}{\e_{j,\nu}^*}\right)^2
\geq 
\frac{1}{2}\frac{d_{g_0}(x_{i,\nu}^*, x_{j,\nu}^*)^2}{\e_{i,\nu}^*\e_{j,\nu}^*}-C\,,
\end{align*}
so that (ii) still holds with  $\bar x_{j,\nu}^*$ replacing $x_{j,\nu}^*$.}
This establishes (iii).

For each pair $(x_{k,\nu}^*,\e_{k,\nu}^*)$, one can check that the  difference between each function obtained above and the corresponding one defined by (\ref{def:ubar}) converges to zero in $H^1(M)$. This proves (iv).
\ep
\begin{proposition}\label{Propo4.2}
If $\uinf(x)=0$ for some $x\in M$, then $\uinf\equiv 0$.
\end{proposition}
\bp
This is just a consequence of the maximum principle.
\ep

Define the functionals
$$
E(u)=\frac{\frac{4(n-1)}{n-2}\int_M|du|_{g_0}^2\dv_{g_0}+\int_{M}\cez u^2dv_{g_0}}{\left(\int_{M}u^{\frac{2n}{n-2}}dv_{g_0}\right)^{\frac{n-2}{n}}}
$$
and
$$
F(u)=\frac{\frac{4(n-1)}{n-2}\int_M|du|_{g_0}^2\dv_{g_0}+\int_{M}\cez u^2dv_{g_0}}{\int_{M}u^{\frac{2n}{n-2}}dv_{g_0}}\,.
$$
Observe that $\cminfbar=F(\uinf)$. Hence,
\ba
1&=\lim_{\nu\to\infty}\int_{M}u_{\nu}^{\crit}dv_{g_0}
=\lim_{\nu\to\infty}\left\{\int_{M}\uinf^{\crit}dv_{g_0}
+\sum_{k=1}^{m}\int_{M}\bar{u}_{(x^*_{k,\nu},\e^*_{k,\nu})}^{\crit}dv_{g_0}\right\}.\notag
\end{align}
The right side of this equation is $(E(\uinf)/\cminfbar)^{n/2}+m_1(\Q/\cminfbar)^{n/2}+(m-m_1)(\Y/\cminfbar)^{n/2}$ if $\uinf>0$ and $m_1(\Q/\cminfbar)^{n/2}+(m-m_1)(\Y/\cminfbar)^{n/2}$ if $\uinf\equiv 0$.
Thus, 
\begin{align}\label{eq:cminbar}
&\cminfbar=\left(E(\uinf)^{n/2}+m_1\Q^{n/2}+(m-m_1)\Y^{n/2}\right)^{2/n}\quad\text{if}\:\uinf>0,
\\
&\text{and}\quad\cminfbar=\left(m_1\Q^{n/2}+(m-m_1)\Y^{n/2}\right)^{2/n}\quad\text{if}\:\uinf\equiv 0.\notag
\end{align}

\subsection{The case $\uinf\equiv 0$}

We set 
\ba\label{def:triplet}
\mathcal{A}_{\nu}=\Big{\{}(x_k,\e_k,\a_k)_{k=1,...,m}&\in (M\times\R_+\times\R_+)^m\,,\:\text{such that}
\\
&x_k\in \d M\:\text{if}\:k\leq m_1\,,\:x_k\in M\backslash\d M\:\text{if}\:k\geq m_1+1,\notag
\\
&\:\:d_{g_0}(x_{k}, x^*_{k,\nu})\leq \e^*_{k,\nu}\,,\:\frac{1}{2}\leq\frac{\e_k}{\e^*_{k,\nu}}\leq 2
\,,\:\frac{1}{2}\leq\a_k\leq 2\Big{\}}\,.\notag
\end{align}
For each $\nu$, we can choose a triplet $(x_{k,\nu},\e_{k,\nu},\a_{k,\nu})_{k=1,...,m}\in\mathcal{A}_{\nu}$ such that
\ba
&\int_M\frac{4(n-1)}{n-2}\big| d(u_{\nu}-\sum_{k=1}^{m}\a_{k,\nu}\uknu)\big|_{g_0}^2\dv_{g_0}
+\int_{ M}\cez\big(u_{\nu}-\sum_{k=1}^{m}\a_{k,\nu}\uknu\big)^2\dv_{g_0}\notag
\\
&\hspace{0.5cm}\leq\int_M\frac{4(n-1)}{n-2}\big{|}d(u_{\nu}-\sum_{k=1}^{m}\a_{k}\uk)\big{|}_{g_0}^2\dv_{g_0}
+\int_{M}\cez\big(u_{\nu}-\sum_{k=1}^{m}\a_{k}\uk\big)^2\dv_{g_0}\notag
\end{align}
for all $(x_k,\e_k,\a_k)_{k=1,...,m}\in\mathcal{A}_{\nu}$. Here, $\bar{u}_{(x_{k,\nu},\e_{k,\nu})}=\bar{u}_{A;(x_{k,\nu},\e_{k,\nu})}$ and $\bar{u}_{(x_{k},\e_{k})}=\bar{u}_{A;(x_{k},\e_{k})}$  if $k\leq m_1$,  $\bar{u}_{(x_{k,\nu},\e_{k,\nu})}=\bar{u}_{B;(x_{k,\nu},\e_{k,\nu})}$ and $\bar{u}_{(x_{k},\e_{k})}=\bar{u}_{B;(x_{k},\e_{k})}$ if $m_1+1\leq k\leq m_2$, and $\bar{u}_{(x_{k,\nu},\e_{k,\nu})}=\bar{u}_{C;(x_{k,\nu},\e_{k,\nu})}$ and $\bar{u}_{(x_{k},\e_{k})}=\bar{u}_{C;(x_{k},\e_{k})}$ if $k\geq m_2+1$; see (\ref{def:test:func:u}), (\ref{def:test:func:u:tubular}) and (\ref{def:test:func:u:int}).
\begin{proposition}\label{Propo5.0}
If $k\geq m_1+1$, then $\lim_{\nu\to\infty} d_{g_0}(x_{k,\nu}, \d M)/\e_{k, \nu}=\infty$.
\end{proposition}
\bp
It follows from the triangle inequality and (\ref{def:triplet}) that
$$
\frac{d_{g_0}(x_{k,\nu}, \d M)}{\e_{k,\nu}}
\geq \frac{d_{g_0}(x_{k,\nu}, \d M)}{2\e^*_{k,\nu}}
\geq \frac{d_{g_0}(x^*_{k,\nu}, \d M)}{2\e^*_{k,\nu}}-\frac{1}{2}.
$$
Now the right side goes to infinity as $\nu\to\infty$ by (iii) of Proposition \ref{Propo4.1}.
\ep
\begin{proposition}\label{Propo5.1}
We have: 

(i) For all $i\neq j$, 
$$
\lim_{\nu\to\infty}\left\{\frac{\e_{i,\nu}}{\e_{j,\nu}}+\frac{\e_{j,\nu}}{\e_{i,\nu}}
+\frac{d_{g_0}(x_{i,\nu},x_{j,\nu})^2}{\e_{i,\nu}\e_{j,\nu}}\right\}=\infty\,.
$$

(ii) We have
$$
\lim_{\nu\to\infty}
\big{\|}\u-\sum_{k=1}^{m}\a_{k,\nu}\uknu\big{\|}_{H^1(M)}= 0\,.
$$
\end{proposition}
\bp
This is a simple consequence of Proposition \ref{Propo4.1} and the definition of  $(x_{k,\nu},\e_{k,\nu},\a_{k,\nu})$; see \cite[Propostion 5.1]{brendle-flow} for details.
\ep
\begin{proposition}\label{Propo5.2}
We have 
$$
d_{g_0}(x_{k,\nu}, x^*_{k,\nu})\leq o(1)\e^*_{k,\nu}\,,\:\:\:
\frac{\e_{k,\nu}}{\e^*_{k,\nu}}=1+o(1)\,,\:\:\:
\text{and}\:\:\:\a_{k,\nu}=1+o(1)\,,
$$
for all $k=1,...,m$. In particular, $(x_{k, \nu},\e_{k, \nu},\a_{k, \nu})_{k=1,...,m}$ is an interior point of $\mathcal{A}_{\nu}$ for $\nu$ sufficiently large.
\end{proposition}
\bp
It follows from Propositions \ref{Propo4.1} and \ref{Propo5.1} that 
\begin{align*}
\Big{\|}\sum_{k=1}^{m}&\a_{k,\nu}\uknu-\sum_{k=1}^{m}\bar{u}_{(x_{k,\nu}^*,\e_{k,\nu}^*)}\Big{\|}_{H^1(M)}
\leq
\Big{\|}\u-\sum_{k=1}^{m}\bar{u}_{(x_{k,\nu}^*,\e_{k,\nu}^*)}\Big{\|}_{H^1(M)}
+
\Big{\|}\u-\sum_{k=1}^{m}\a_{k,\nu}\uknu\Big{\|}_{H^1(M)}=o(1).
\end{align*}
Now the result follows.
\ep
\begin{notation}
We write $u_{\nu}=v_{\nu}+w_{\nu}$, where
\begin{equation}\label{def:v:w:1}
v_{\nu}=\sum_{k=1}^{m}\a_{k,\nu}\uknu\:\:\:\:\text{and}\:\:\:\:
w_{\nu}=u_{\nu}-\sum_{k=1}^{m}\a_{k,\nu}\uknu\,.
\end{equation}
\end{notation}
Observe that by Proposition \ref{Propo5.1} we have
\begin{equation}\label{propr:w:1}
\int_M\frac{4(n-1)}{n-2}|dw_{\nu}|_{g_0}^2\dv_{g_0}
+\int_{M}\cez w_{\nu}^2\dv_{g_0}=o(1)\,.
\end{equation}
Set 
$$
C_{\nu}=\left(\int_{\d M}|\w|^{\critbordo}\ds_{g_0}\right)^{\frac{n-2}{2(n-1)}}
+\left(\int_{M}|\w|^{\crit}\dv_{g_0}\right)^{\frac{n-2}{2n}}.
$$ 
\begin{proposition}\label{Propo5.3}
Fix $\rho\leq P_0$. Let $\psi_{k,\nu}:\Omega_{k,\nu}=B^+_{\rho}(0)\subset \Rn\to M$ be Fermi coordinates centered at $x_{k,\nu}$ if $1\leq k\leq m_1$, and let $\psi_{k,\nu}:\Omega_{k,\nu}=\tilde{B}_{x_{k,\nu}, \rho}\subset R^n\to M$ be normal coordinates centered at $x_{k,\nu}$ if $m_1+1\leq k\leq m$ (see Definitions \ref{def:fermi} and \ref{def:normal}). We have:
\\
(i) $\displaystyle\big{|} \int_{M} \uknu^{\frac{n+2}{n-2}}\,\w\,\dv_{g_0}\big{|}
\leq o(1)\,C_{\nu}$\,.
\\
(ii) $\displaystyle\big{|} \int_{\Omega_{k,\nu}} \uknu^{\frac{n+2}{n-2}}
\frac{\e_{k,\nu}^2-|\psi_{k,\nu}^{-1}(x)|^2}{\e_{k,\nu}^2+|\psi_{k,\nu}^{-1}(x)|^2}\,\w\,\dv_{g_0}\big{|}
\leq o(1)\,C_{\nu}$\,.
\\
(iii) $\displaystyle\big{|} \int_{\Omega_{k,\nu}} \uknu^{\frac{n+2}{n-2}}
\frac{\e_{k,\nu}\psi_{k,\nu}^{-1}(x)}{\e_{k,\nu}^2+|\psi_{k,\nu}^{-1}(x)|^2}\,\w\,\dv_{g_0}\big{|}
\leq o(1)\,C_{\nu},\:\:\:\:\text{if}\:m_1+1\leq k\leq m$,

\vspace{0.1cm}
and
$\displaystyle\big{|} \int_{\Omega_{k,\nu}} \uknu^{\frac{n+2}{n-2}}
\frac{\e_{k,\nu}\overline{\psi_{k,\nu}^{-1}(x)}}{\e_{k,\nu}^2+|\psi_{k,\nu}^{-1}(x)|^2}\,\w\,\dv_{g_0}\big{|}
\leq o(1)\,C_{\nu},\:\:\:\:\text{if}\: k\leq m_1,$
\\
where we are denoting $\bar{y}=(y_1,...,y_{n-1})$ for any $y=(y_1,...,y_n)\in \R^n$.
\end{proposition}
\bp
 It follows from the definition of $(x_{k,\nu},\e_{k,\nu},\a_{k,\nu})$ that
$$
\int_M\left(\frac{4(n-1)}{n-2}\langle d\uknu,dw_{\nu}\rangle_{g_0}+\cez\uknu w_\nu\right)\dv_{g_0}
=0\,.
$$
Integrating by parts, we obtain 
\ba
\int_M&\left(\frac{4(n-1)}{n-2}\Delta_{g_0}\uknu-\cez\uknu\right)w_\nu dv_{g_0}+\int_{\d M}\frac{4(n-1)}{n-2}\frac{\d}{\d\eta_{g_0}}\uknu
w_{\nu}\,\ds_{g_0}=0.\notag
\end{align}
We claim that
$$
\Big\|\frac{4(n-1)}{n-2}\Delta_{g_0}\uknu-\cez\uknu+\cminfbar\uknu^{\frac{n+2}{n-2}}\Big\|_{L^{\frac{2n}{n+2}}(M)}=o(1),
$$
and
$$
\:\:\:\:
\Big\|\frac{\d}{\d\eta_{g_0}}\uknu\Big\|_{L^{\frac{2(n-1)}{n}}(\d M)}=o(1).
$$
The first statement follows from Corollary \ref{Corol6:notes5}. As for the second one, observe first that 
$$
\d\uknu /\d\eta_{g_0}=0
$$ 
on $\d M$ if $\uknu=\bar u_{C;(x_{k,\nu}, \e_{k,\nu})}$. If $\uknu=\bar u_{A;(x_{k,\nu}, \e_{k,\nu})}$ this statement follows easily from Proposition \ref{Propo2:notes5} and \eqref{eq:R:H}, and  if $\uknu=\bar u_{B;(x_{k,\nu}, \e_{k,\nu})}$ this is  Corollary \ref{Corol5:notes5}, also making use of  Proposition \ref{Propo5.0}.

This proves (i). The remaining statements follow similarly.
\ep
\begin{proposition}\label{Propo5.4}
There exists $c>0$ such that 
\ba
\frac{n+2}{n-2}\cminfbar\int_{M}\sum_{k=1}^{m}&\uknu^{\frac{4}{n-2}}\w^2\,dv_{g_0}\notag
\leq
(1-c)\left\{\int_{M}\frac{4(n-1)}{n-2}|d\w|_{g_0}^2 dv_{g_0}
+\int_{M}\cez\w^2\,dv_{g_0}\right\}\notag
\end{align}
for all $\nu$ sufficiently large.
\end{proposition}
\bp
Once we have proved Proposition \ref{Propo5.3}, this proof is a contradiction argument similar to \cite[Propostion 5.4]{brendle-flow} and \cite[Proposition 4.6]{almaraz5} and we will omit the details. Assume by contradiction that there is a sequence $\{\tilde w_{\nu}\}$ satisfying 
\begin{equation*}
\int_M \frac{4(n-1)}{n-2}|d\tilde{w}_{\nu}|_{g_0}^2\dv_{g_0}
+\int_{M}\cez\tilde{w}_{\nu}^2\,dv_{g_0}=1
\end{equation*}
and
\begin{equation*}
\lim_{\nu\to\infty}\frac{n+2}{n-2}\cminfbar\int_{M}\sum_{k=1}^{m}\uknu^{\frac{4}{n-2}}\tilde{w}_{\nu}^2\,dv_{g_0}
\geq 1\,.
\end{equation*}
After rescaling around $x_{k,\nu}$, the new sequence obtained converges (weakly in $H^1_{loc}(\Rn)$ if $k\leq m_1$ and in $H^1_{loc}(\R^n)$ if $k\geq m_1+1$) to a certain $\hat w$.  It turns out that one can choose $k\in \{1,...,m\}$ in such way that $\hat w$ satisfies 
\begin{equation*}
\int_{\Rn}\left(\frac{1}{1+|y|^2}\right)^2\hat{w}^2(y)\,dy>0
\end{equation*}
and
\begin{equation*}
\int_{\Rn}|d\hat{w}(y)|^2dy
\leq n(n+2)\int_{\Rn}\left(\frac{1}{1+|y|^2}\right)^2\hat{w}^2(y)\,dy
\end{equation*}
if $k\leq m_1$, or the same two inequalities with $\Rn$ replaced by $\R^n$ if $k\geq m_1+1$. 

On the other hand, if  $k\leq m_1$, due to Proposition \ref{Propo5.3},  $\hat w$ satisfies 
\begin{equation*}
\int_{\Rn}\left(\frac{1}{1+|y|^2}\right)^{\frac{n+2}{2}}\hat{w}(y)\,dy=0\,,
\end{equation*}
\begin{equation*}
\int_{\Rn}\left(\frac{1}{1+|y|^2}\right)^{\frac{n+2}{2}}\frac{1-|y|^2}{1+|y|^2}\hat{w}(y)\,dy=0\,,
\end{equation*}
\begin{equation*}
\int_{\Rn}\left(\frac{1}{1+|y|^2}\right)^{\frac{n+2}{2}}\frac{y_j}{1+|y|^2}\hat{w}(y)\,dy=0\,,
\end{equation*}
where $y=(y_1,...,y_n)$, and $j=1,...,n-1$.  
By considering the corresponding equations on the round hemisphere we obtain a contradiction as in \cite[Proposition 4.6]{almaraz5}.
If $k\geq m_1+1$, $\hat w$ satisfies the same last three equations (with $j=1,...,n$ for the last), but with $\Rn$ replaced by $\R^n$, and the same contradiction is reached by considering corresponding equations on the round sphere instead of the hemisphere.
\ep
\begin{corollary}\label{Corol5.5}
There exists $c>0$ such that 
$$
\frac{n+2}{n-2}\cminfbar\int_{M}v_{\nu}^{\frac{4}{n-2}}w_{\nu}^2\,\dv_{g_0}
\leq
(1-c)\left\{\int_{M}\frac{4(n-1)}{n-2}|d\w|_{g_0}^2\dv_{g_0}
+\int_{M}\cez\w^2\,\dv_{g_0}\right\}
$$
for all $\nu$ sufficiently large.
\end{corollary}
\bp
By the definition of $v_{\nu}$ (equation (\ref{def:v:w:1})), we have
$$
\lim_{\nu\to\infty}\int_{M}\big|v_{\nu}^{\frac{4}{n-2}}-\sum_{k=1}^{m}\bar{u}_{(x_{k,\nu}, \e_{k,\nu})}^{\frac{4}{n-2}}\big|^{n/2}\dv_{g_0}=0\,.
$$
Hence, the assertion follows from Proposition \ref{Propo5.4}.
\ep
\begin{proposition}\label{Propo5.6}
For all $\nu$ sufficiently large, we have $E(v_{\nu})\leq\left(\sum_{k=1}^{m}E(\uk)^{n/2}\right)^{2/n}$.
\end{proposition}
\bp
Choose a permutation $\sigma: \{1,...,m\}$ such that $\e_{\sigma(i),\nu}\leq \e_{\sigma(j),\nu}$ for all $i<j$. During this proof we will omit the symbol $\sigma$, writing $\e_{i,\nu}$ instead of $\e_{\sigma(i),\nu}$, so that $\e_{i,\nu}\leq \e_{j,\nu}$ for all $i<j$.
After calculations similar to the ones in \cite[Proposition 5.6]{brendle-flow} we obtain
\begin{align*}
E(v_\nu)\left(\im v_\nu^\crit\dv_{g_0}\right)^{\frac{n-2}{n}}
\leq& \left(\sum_{k=1}^mE(\uknu)^{\frac{n}{2}}\right)^{\frac 2n}\left(\im v_\nu^\crit\dv_{g_0}\right)^{\frac{n-2}{n}}
-c\sum_{i<j}\left(\frac{\e_{i,\nu}\e_{j,\nu}}{\e^2_{j,\nu}+d_{g_0}(x_{i,\nu},x_{j,\nu})^2}\right)^{\frac{n-2}{2}}\\
&-2\im \sum_{i<j}\a_{i,\nu}\a_{j,\nu}\uinu\left(\ct\Delta_{g_0}\ujnu-\cez\ujnu+\cminfbar\ujnu^{\frac{n+2}{n-2}}\right)\dv_{g_0}\\
&-\frac{8(n-1)}{n-2}\idm\sum_{i<j}\a_{i,\nu}\a_{j,\nu}\uinu\frac{\d\ujnu}{\d\eta_{g_0}}\ds_{g_0}\\
&-2\sum_{i<j}\a_{i,\nu}\a_{j,\nu}(F(\ujnu)-\cminfbar)\im \uinu\ujnu^{\frac{n+2}{n-2}}\dv_{g_0}.
\end{align*}
It is not difficult to see that $F(\ujnu)=\cminfbar+o(1)$. This is more subtle in the case $\ujnu=\bar u_{B;(x_{j,\nu}, \e_{j,\nu})}$, when we make use of Proposition \ref{Propo5.0} and Lemma \ref{Lemma5:notes3}. Then,  because of \cite[Lemma B.4]{brendle-flow}, we have
$$
|F(\ujnu)-\cminfbar|\im \uinu\ujnu^{\frac{n+2}{n-2}}\dv_{g_0}
\leq
o(1)\left(\frac{\e_{i,\nu}\e_{j,\nu}}{\e^2_{j,\nu}+d_{g_0}(x_{i,\nu},x_{j,\nu})^2}\right)^{\frac{n-2}{2}}.
$$
Then, using Corollaries \ref{Corol2:notes5} and \ref{Corol3:notes5}, 
\begin{align*}
E(v_\nu)\left(\im v_\nu^\crit\dv_{g_0}\right)^{\frac{n-2}{n}}
\leq& 
\left(\sum_{k=1}^m E(\uknu )^{\frac{n}{2}}\right)^{\frac{2}{n}}\left(\im v_{\nu}^\crit\right)^{\frac{n-2}{n}}\\
&-\sum_{i<j}(c-C\max\{\rho_A,\rho_B,\rho_C\}^{1/2}-o(1))\left(\frac{\e_{i,\nu}\e_{j,\nu}}{\e^2_{j,\nu}+d_{g_0}(x_{i,\nu},x_{j,\nu})^2}\right)^{\frac{n-2}{2}}.
\end{align*}
Hence, the assertion follows by choosing $\rho_A, \rho_B$ and $\rho_C$ smaller if necessary (see Remark \ref{choosing:bubbles}).

\ep
\begin{corollary}\label{Corol5.7}
Under the hypothesis of Theorem \ref{main:thm},  we have
$$
E(v_{\nu})\leq \cminfbar,\quad\quad
\text{for all $\nu$ sufficiently large.}
$$ 

\end{corollary}
\bp
Using Propositions \ref{Propo:energy:test}, \ref{Propo:energy:test:tubular} and \ref{Propo:energy:test:int}, we obtain $E(\uknu)\leq \Q$ for $k\leq m_1$, and $E(\uknu)\leq \Y$ for $k\geq m_1+1$. Then the result follows from Proposition \ref{Propo5.6} and \eqref{eq:cminbar}.
\ep

\subsection{The case $\uinf >0$}
\begin{proposition}
There exist sequences $\{\psi_a\}_{a\in\N}\subset C^{\infty}(M)$ and $\{\lambda_a\}_{a\in\N}\subset \R$, with $\lambda_a>0$,  satisfying:

(i) For all $a\in\N$, 
\begin{equation}\notag
\begin{cases} 
\ct\Delta_{g_0}\psi_a-\cez\psi_a+\lambda_a\uinf^{\frac{4}{n-2}}\psi_a=0\,,&\text{in}\:M\,,
\\
\frac{\d}{\d\eta_{g_0}}\psi_a=0\,,&\text{on}\:\d M\,.
\end{cases} 
\end{equation}

(ii) For all $a,b\in\N$,
$$
\int_{M}\psi_a\psi_bu_{\infty}^{\frac{4}{n-2}}\dv_{g_0}=
\begin{cases} 
1\,,&\text{if}\:a=b\,,
\\
0\,,&\text{if}\:a\neq b\,.
\end{cases}
$$

(iii) The span of $\{\psi_a\}_{a\in\N}$ is dense in $L^2(M)$.

(iv) We have $\lim_{a\to\infty}\lambda_a=\infty$.
\end{proposition}
\bp
Since we are assuming $R_{g_0}>0$, for each $f\in L^2(M)$ we can define $T(f)=u$, where $u\in H^{1}(M)$ is the unique solution of 
\begin{equation}\notag
\begin{cases} 
\ct\Delta_{g_0}u-\cez u=f\uinf^\frac{4}{n-2}\,,&\text{in}\:M\,,
\\
\frac{\d}{\d\eta_{g_0}}u=0\,,&\text{on}\:\d M\,.
\end{cases} 
\end{equation}
Since $H^{1}(M)$ is compactly embedded in $L^2(M)$, the operator $T:L^2( M)\to L^2(M)$ is compact. Integrating by parts, we see that $T$ is symmetric with respect to the inner product $(\psi_1,\psi_2)\mapsto\int_{ M}\psi_1\psi_2u_{\infty}^{\frac{4}{n-2}}\dv_{g_0}$. Then the result follows from the spectral theorem for compact operators.
\ep
Let $A\subset \N$ be a finite set such that $\l_a>\frac{n+2}{n-2}\cminfbar$ for all $a\notin A$, and define the projection 
$$
\Gamma(f)
=\sum_{a\notin A}\left(\int_{ M}\psi_af\dv_{g_0}\right)\psi_a\uinf^{\frac{4}{n-2}}
=f-\sum_{a\in A}\left(\int_{ M}\psi_af\dv_{g_0}\right)\psi_a\uinf^{\frac{4}{n-2}}\,.
$$
\begin{lemma}\label{Lemma6.4}
There exists $\zetaup>0$ with the following significance: for all $z\in\R^{A}$ with $|z|\leq \zetaup$, there exists a smooth function $\bar{u}_z$ satisfying $\d\bar{u}_z/\d\eta_{g_0}=0$ on $\d M$, 
\begin{equation}\label{Lemma6.4:1}
\int_{M}\uinf^{\frac{4}{n-2}}(\bar{u}_z-\uinf)\psi_a\dv_{g_0}=z_a
\:\:\:\:\text{for all}\:a\in A\,,
\end{equation}
and
\begin{equation}\label{Lemma6.4:2}
\Gamma\left(\ct\Delta_{g_0}\bar{u}_z-\cez \bar{u}_z+\cminfbar \bar{u}_{z}^{\frac{n+2}{n-2}}\right)=0\,.
\end{equation}
Moreover, the mapping $z\mapsto \bar{u}_z$ is real analytic. 
\end{lemma}

\bp
This is just an application of the implicit function theorem.
\ep
\begin{lemma}\label{Lemma6.5}
There exists $0<\gamma<1$ such that
$$
E(\bar{u}_z)-E(\uinf)
\leq
C\sup_{a\in A}\left|
\int_{ M}\psi_a\left(\ct\Delta_{g_0}\bar{u}_z-\cez\,\bar{u}_z+\cminfbar\,\bar{u}_z^{\frac{n+2}{n-2}}
\right)\dv_{g_0}
\right|^{1+\gamma}\,,
$$
if $|z|$ is sufficiently small.
\end{lemma}
\bp
Observe that the function $z\mapsto E(\bar{u}_z)$ is real analytic. According to results of Lojasiewicz (see equation (2.4) in \cite[p.538]{simon}), there exists $0<\gamma<1$ such that
$$
|E(\bar{u}_z)-E(\uinf)|\leq \sup_{a\in A}\left|\frac{\d}{\d z_a}E(\bar{u}_z)\right|^{1+\gamma}\,,
$$
if $|z|$ is sufficiently small.
Now we can follow the lines in \cite[Lemma 6.5]{brendle-flow} to obtain the result.
\ep
We set 
\ba
\mathcal{A}_{\nu}=\Big{\{}(z, (x_k,\e_k,\a_k)_{k=1,...,m})\in &\,\R^A\times (M\times\R_+\times\R_+)^m\,,\:\text{such that}\notag
\\
&x_k\in \d M\:\text{if}\:k\leq m_1\,,\:x_k\in M\backslash\d M\:\text{if}\:k\geq m_1+1,\notag
\\
&\:\:|z|\leq\zetaup,\: d_{g_0}(x_{k}, x^*_{k,\nu})\leq \e^*_{k,\nu}\,,\:\frac{1}{2}\leq\frac{\e_k}{\e^*_{k,\nu}}\leq 2
\,,\:\frac{1}{2}\leq\a_k\leq 2\Big{\}}\,.\notag
\end{align}

For each $\nu$, we can choose a pair $(z_{\nu}, (x_{k,\nu},\e_{k,\nu},\a_{k,\nu})_{k=1,...,m})\in\mathcal{A}_{\nu}$ such that
\ba
&\int_M\frac{4(n-1)}{n-2}\big| d(u_{\nu}-\bar{u}_{z_{\nu}}-\sum_{k=1}^{m}\a_{k,\nu}\uknu )\big|_{g_0}^2\dv_{g_0}+\int_{M}\cez\big(u_{\nu}-\bar{u}_{z_{\nu}}-\sum_{k=1}^{m}\a_{k,\nu}\uknu \big)^2\dv_{g_0}\notag
\\
&\hspace{0.5cm}\leq\int_M\frac{4(n-1)}{n-2}\big{|}d(u_{\nu}-\bar{u}_{z}-\sum_{k=1}^{m}\a_{k}\uk)\big{|}_{g_0}^2\dv_{g_0}+\int_{M}\cez\big(u_{\nu}-\bar{u}_{z}-\sum_{k=1}^{m}\a_{k}\uk\big)^2\dv_{g_0}\notag
\end{align}
for all $(z, (x_k,\e_k,\a_k)_{k=1,...,m})\in\mathcal{A}_{\nu}$. Here, $\bar{u}_{(x_{k,\nu},\e_{k,\nu})}=\bar{u}_{A;(x_{k,\nu},\e_{k,\nu})}$ and $\bar{u}_{(x_{k},\e_{k})}=\bar{u}_{A;(x_{k},\e_{k})}$  if $k\leq m_1$,  $\bar{u}_{(x_{k,\nu},\e_{k,\nu})}=\bar{u}_{B;(x_{k,\nu},\e_{k,\nu})}$ and $\bar{u}_{(x_{k},\e_{k})}=\bar{u}_{B;(x_{k},\e_{k})}$ if $m_1+1\leq k\leq m_2$, and $\bar{u}_{(x_{k,\nu},\e_{k,\nu})}=\bar{u}_{C;(x_{k,\nu},\e_{k,\nu})}$ and $\bar{u}_{(x_{k},\e_{k})}=\bar{u}_{C;(x_{k},\e_{k})}$ if $k\geq m_2+1$; see (\ref{def:test:func:u}), (\ref{def:test:func:u:tubular}) and (\ref{def:test:func:u:int}).

The proofs of the next three propositions are similar to Propositions \ref{Propo5.0}, \ref{Propo5.1} and \ref{Propo5.2}.
\begin{proposition}\label{Propo6.0}
If $k\geq m_1+1$, then $\lim_{\nu\to\infty} d_{g_0}(x_{k,\nu}, \d M)/\e_{k, \nu}=\infty$.
\end{proposition}
\begin{proposition}\label{Propo6.6}
We have:

(i) For all $i\neq j$, 
$$
\lim_{\nu\to\infty}\left\{
\frac{\e_{i,\nu}}{\e_{j,\nu}}+\frac{\e_{j,\nu}}{\e_{i,\nu}}
+\frac{d_{g_0}(x_{i,\nu},x_{j,\nu})^2}{\e_{i,\nu}\e_{j,\nu}}\right\}=\infty\,.
$$

(ii) We have
$$
\lim_{\nu\to\infty}
\big{\|}\u-\bar{u}_{z_{\nu}}-\sum_{k=1}^{m}\a_{k,\nu}\uknu\big{\|}_{H^1(M)}=0\,.
$$
\end{proposition}

\begin{proposition}\label{Propo6.7}
We have $|z_{\nu}|=o(1)$,
and
$$
d_{g_0}(x_{k,\nu}, x^*_{k,\nu})\leq o(1)\,\e^*_{k,\nu}\,,\:\:\:
\frac{\e_{k,\nu}}{\e^*_{k,\nu}}=1+o(1)\,,\:\:\:
\text{and}\:\:\:\a_{k,\nu}=1+o(1)\,,
$$
for all $k=1,...,m$. In particular, $(z_{\nu},(x_{k, \nu},\e_{k, \nu},\a_{k, \nu})_{k=1,...,m})$ is an interior point of $\mathcal{A}_{\nu}$ for $\nu$ sufficiently large.
\end{proposition}
\begin{notation}
We write $u_{\nu}=v_{\nu}+w_{\nu}$, where
\begin{equation}\label{def:v:w:2}
v_{\nu}=\bar{u}_{z_{\nu}}+\sum_{k=1}^{m}\a_{k,\nu}\uknu\:\:\:\:\text{and}\:\:\:\:
w_{\nu}=u_{\nu}-\bar{u}_{z_{\nu}}-\sum_{k=1}^{m}\a_{k,\nu}\uknu\,.
\end{equation}
\end{notation}

Observe that by Proposition \ref{Propo6.6} we have
\begin{equation}\label{propr:w:2}
\int_M\frac{4(n-1)}{n-2}|dw_{\nu}|_{g_0}^2\dv_{g_0}
+\int_{M}\cez w_{\nu}^2\dv_{g_0}=o(1)\,.
\end{equation}
Set 
$$
C_{\nu}=\left(\int_{\d M}|\w|^{\critbordo}\ds_{g_0}\right)^{\frac{n-2}{2(n-1)}}
+\left(\int_{M}|\w|^{\crit}\dv_{g_0}\right)^{\frac{n-2}{2n}}\,,
$$
\begin{proposition}\label{Propo6.8}
Fix $\rho\leq P_0$. Let $\psi_{k,\nu}:\Omega_{k,\nu}=B^+_{\rho}(0)\subset \Rn\to M$ be Fermi coordinates centered at $x_{k,\nu}$ if $1\leq k\leq m_1$, and let $\psi_{k,\nu}:\Omega_{k,\nu}=\tilde{B}_{x_{k,\nu}\rho}\subset R^n\to M$ be normal coordinates centered at $x_{k,\nu}$ if $m_1+1\leq k\leq m$ (see Definitions \ref{def:fermi} and \ref{def:normal}). We have:
\\
(i) $\displaystyle\big{|} \int_{M} \uinf^{\frac{4}{n-2}}\psi_a\,\w\,\dv_{g_0}\big{|}
\leq o(1)\im|w_\nu|\dv_{g_0},\:\:\:\:\text{for}\:a\in A.$
\\
(ii) $\displaystyle\big{|} \int_{M} \uknu^{\frac{n+2}{n-2}}\,\w\,\dv_{g_0}\big{|}
\leq o(1)\,C_{\nu}$\,.
\\
(iii) $\displaystyle\big{|} \int_{\Omega_{k,\nu}} \uknu^{\frac{n+2}{n-2}}
\frac{\e_{k,\nu}^2-|\psi_{k,\nu}^{-1}(x)|^2}{\e_{k,\nu}^2+|\psi_{k,\nu}^{-1}(x)|^2}\,\w\,\dv_{g_0}\big{|}
\leq o(1)\,C_{\nu}$\,.
\\
(iv) $\displaystyle\big{|} \int_{\Omega_{k,\nu}} \uknu^{\frac{n+2}{n-2}}
\frac{\e_{k,\nu}\psi_{k,\nu}^{-1}(x)}{\e_{k,\nu}^2+|\psi_{k,\nu}^{-1}(x)|^2}\,\w\,\dv_{g_0}\big{|}
\leq o(1)\,C_{\nu},\:\:\:\:\text{if}\:m_1+1\leq k\leq m$,

\vspace{0.1cm}
and
$\displaystyle\big{|} \int_{\Omega_{k,\nu}} \uknu^{\frac{n+2}{n-2}}
\frac{\e_{k,\nu}\overline{\psi_{k,\nu}^{-1}(x)}}{\e_{k,\nu}^2+|\psi_{k,\nu}^{-1}(x)|^2}\,\w\,\dv_{g_0}\big{|}
\leq o(1)\,C_{\nu},\:\:\:\:\text{if}\: k\leq m_1,$
\\
where we are denoting $\bar{y}=(y_1,...,y_{n-1})$ for any $y=(y_1,...,y_n)\in \R^n$.
\end{proposition}
\bp
(i) Set $\tilde{\psi}_{a,z}=\d \bar{u}_z/\d z_a$. It follows from the identities (\ref{Lemma6.4:1}) and (\ref{Lemma6.4:2}) that $\tilde{\psi}_{a,0}=\psi_a$ for all $a\in A$.
By the definition of $(z_{\nu}, (x_{k,\nu}, \e_{k,\nu}, \a_{k,\nu})_{1\leq k\leq m})$, we have
$$
\int_M\ct\langle d\tilde{\psi}_{a,z_{\nu}},w_{\nu}\rangle_{g_0}\dv_{g_0}
+\int_{ M}\cez\tilde{\psi}_{a,z_{\nu}}w_{\nu}\,\dv_{g_0}=0\,.
$$
Hence,
\ba
\l_a\int_{ M}\uinf^{\frac{4}{n-2}}\psi_aw_{\nu}\,\dv_{g_0}&=-\int_{M}\left(\ct\Delta_{g_0}\psi_a-\cez\psi_a\right)w_{\nu}\,\dv_{g_0}\notag
\\
&=\int_{ M}\left(\ct\Delta_{g_0}(\tilde{\psi}_{a,z_{\nu}}-\psi_a)
-\cez(\tilde{\psi}_{a,z_{\nu}}-\psi_a)\right)w_{\nu}\,\dv_{g_0}
+\idm \frac{\d\tilde{\psi}_{a,z_{\nu}}}{\d\eta_{g_0}}w_\nu\ds_{g_0}\,.\notag
\end{align}
However, we know that $\d\tilde{\psi}_{a,z_{\nu}}/\d\eta_{g_0}=0$ on $\d M$.
Then, since $\l_a>0$ and $|z_{\nu}|\to 0$ as $\nu\to\infty$, we conclude that the assertion (i) follows.

The proofs of (ii), (iii), and (iv) are similar to Proposition \ref{Propo5.3}.
\ep
\begin{proposition}\label{Propo6.9}
There exists $c>0$ such that 
\ba
\frac{n+2}{n-2}&\cminfbar\int_{ M}\Big(\uinf^{\frac{4}{n-2}}+\sum_{k=1}^{m}\uknu^{\frac{4}{n-2}}\Big)\w^2\,\dv_{g_0}\leq
(1-c)\int_{M}\Big(\frac{4(n-1)}{n-2}|d\w|_{g_0}^2+\cez\w^2\Big)\,\dv_{g_0}\notag
\end{align}
for all $\nu$ sufficiently large.
\end{proposition}
\bp
As in Proposition \ref{Propo5.4}, once Proposition \ref{Propo6.8} is established, this proof is a contradiction argument similar to \cite[Proposition 6.8]{brendle-flow} and \cite[Proposition 4.18]{almaraz5}.
\ep
\begin{corollary}\label{Corol6.10}
There exists $c>0$ such that 
$$
\frac{n+2}{n-2}\cminfbar\int_{ M}v_{\nu}^{\frac{4}{n-2}}w_{\nu}^2\,\dv_{g_0}
\leq
(1-c)\int_{M}\Big(\frac{4(n-1)}{n-2}|d\w|_{g_0}^2
+\cez\w^2\Big)\,\dv_{g_0}
$$
for all $\nu$ sufficiently large.
\end{corollary}
\bp
By the definition of $v_{\nu}$ (see (\ref{def:v:w:2})), we have
$$
\lim_{\nu\to\infty}\int_{ M}\big|v_{\nu}^{\frac{4}{n-2}}-\uinf^{\frac{4}{n-2}}-\sum_{k=1}^{m}\bar{u}_{(x_{k,\nu}, \e_{k,\nu})}^{\frac{4}{n-2}}\big|^{\frac{n}{2}}\dv_{g_0}=0\,.
$$
Hence, the assertion follows from Proposition \ref{Propo6.9}.
\ep
The next two propositions are similar to Propositions 6.14 and 6.15 of \cite{brendle-flow} and we will just outline their proofs.
\begin{proposition}\label{Propo6.14}
There exist $C>0$ and $0<\gamma<1$ such that
\ba
E(\bar{u}_{z_{\nu}})&-E(\uinf)\leq
C\left\{
\int_{ M}u_{\nu}^{\crit}|R_{g_{\nu}}-\cminfbar|^{\frac{2n}{n+2}}\dv_{g_0}
\right\}^{\frac{n+2}{2n}(1+\gamma)}
+C\sum_{k=1}^{m}\e_{k,\nu}^{\frac{n-2}{2}(1+\gamma)}\notag
\end{align}
if $\nu$ is sufficiently large.
\end{proposition}
\bp
As in \cite[Lemmas 6.11 and 6.12]{brendle-flow}, because $\d u_{\nu}/\d\eta_{g_0}=\d \bar{u}_{z_{\nu}}/\d\eta_{g_0}=0$ on $\d M$, we can show that there exists $C>0$ such that
\begin{equation}\label{Lemma6.11}
\|u_{\nu}-\bar{u}_{z_{\nu}}\|_{L^{\frac{n+2}{n-2}}( M)}^{\frac{n+2}{n-2}}
\leq C\|u_{\nu}^{\frac{n+2}{n-2}}(R_{g_{\nu}}-\cminfbar)\|_{L^{\frac{2n}{n+2}}(M)}^{\frac{n+2}{n-2}}
+C\sum_{k=1}^{m}\e_{k,\nu}^{\frac{n-2}{2}}
\end{equation}
and
\begin{equation}\label{Lemma6.12}
\|u_{\nu}-\bar{u}_{z_{\nu}}\|_{L^1(M)}
\leq C\|u_{\nu}^{\frac{n+2}{n-2}}(R_{g_{\nu}}-\cminfbar)\|_{L^{\frac{2n}{n+2}}(M)}
+C\sum_{k=1}^{m}\e_{k,\nu}^{\frac{n-2}{2}}\,,
\end{equation}
for $\nu$ sufficiently large.

We will prove the estimate
\ba\label{Lemma6.13}
\sup_{a\in A}&\left|
\int_{ M}\psi_a\left(
\ct\Delta_{g_0}\bar{u}_{z_{\nu}}-\cez\bar{u}_{z_{\nu}}+\cminfbar\bar{u}_{z_{\nu}}^{\frac{n+2}{n-2}}
\right)\dv_{g_0}
\right|
\\
&\leq
C\left\{
\int_{M}u_{\nu}^{\crit}|R_{g_{\nu}}-\cminfbar|^{\frac{2n}{n+2}}\dv_{g_0}
\right\}^{\frac{n+2}{2n}}
+C\sum_{k=1}^{m}\e_{k,\nu}^{\frac{n-2}{2}}\notag
\end{align}
for $\nu$ is sufficiently large.

Integrating by parts, we obtain
\ba
\int_{ M}&\psi_a\left(
\ct\Delta_{g_0}\bar{u}_{z_{\nu}}-\cez\bar{u}_{z_{\nu}}+\cminfbar\bar{u}_{z_{\nu}}^{\frac{n+2}{n-2}}
\right)\dv_{g_0}\notag
\\
&=\int_{ M}\psi_a\left(
\ct\Delta_{g_0}u_{\nu}-\cez u_{\nu}+\cminfbar u_{\nu}^{\frac{n+2}{n-2}}
\right)\dv_{g_0}\notag
\\
&\hspace{0.5cm}+\l_a\int_{M}\uinf^{\frac{4}{n-2}}\psi_a(u_{\nu}-\bar{u}_{z_{\nu}})\,\dv_{g_0}
-\cminfbar\int_{ M}\psi_a(u_{\nu}^{\frac{n+2}{n-2}}-\bar{u}_{z_{\nu}}^{\frac{n+2}{n-2}})\,\dv_{g_0}\,.\notag
\end{align}
Using the fact that $\ct\Delta_{g_0}u_{\nu}-\cez u_{\nu}+\cminfbar u_{\nu}^{\frac{n+2}{n-2}}=-(R_{g_{\nu}}-\cminfbar)u_{\nu}^{\frac{n+2}{n-2}}$  and the pointwise estimate
$$
|u_{\nu}^{\frac{n+2}{n-2}}-\bar{u}_{z_{\nu}}^{\frac{n+2}{n-2}}|
\leq C\bar{u}_{z_{\nu}}^{\frac{4}{n-2}}|u_{\nu}-\bar{u}_{z_{\nu}}|
+C|u_{\nu}-\bar{u}_{z_{\nu}}|^{\frac{n+2}{n-2}}\,,
$$
we obtain
\ba
\sup_{a\in A}&\left|
\int_{ M}\psi_a\left(
\ct\Delta_{g_0}\bar{u}_{z_{\nu}}-\cez\bar{u}_{z_{\nu}}+\cminfbar\bar{u}_{z_{\nu}}^{\frac{n+2}{n-2}}
\right)\dv_{g_0}
\right|\notag
\\
&\leq
C\|u_{\nu}^{\frac{n+2}{n-2}}(R_{g_{\nu}}-\cminfbar)\|_{L^{\frac{2n}{n+2}}(M)}
+C\|u_{\nu}-\bar{u}_{z_{\nu}}\|_{L^1( M)}
+C\|u_{\nu}-\bar{u}_{z_{\nu}}\|_{L^{\frac{n+2}{n-2}}(M)}^{\frac{n+2}{n-2}}\,.\notag
\end{align}
Then it follows from (\ref{Lemma6.11}) and (\ref{Lemma6.12}) that
\ba\label{Lemma6.13:1}
\sup_{a\in A}&\left|
\int_{M}\psi_a\left(
\ct\Delta_{g_0}\bar{u}_{z_{\nu}}-\cez\bar{u}_{z_{\nu}}+\cminfbar\bar{u}_{z_{\nu}}^{\frac{n+2}{n-2}}
\right)\dv_{g_0}
\right|
\\
&\leq
C\|u_{\nu}^{\frac{n+2}{n-2}}(R_{g_{\nu}}-\cminfbar)\|_{L^{\frac{2n}{n+2}}(M)}^{\frac{n+2}{n-2}}
+C\|u_{\nu}^{\frac{n+2}{n-2}}(R_{g_{\nu}}-\cminfbar)\|_{L^{\frac{2n}{n+2}}( M)}
+C\sum_{k=1}^{m}\e_{k,\nu}^{\frac{n-2}{2}}\,.\notag
\end{align}

On the other hand,  by Corollary \ref{Corol3.2} we can assume 
\begin{equation}\label{Lemma6.12:4}
\|u_{\nu}^{\frac{n+2}{n-2}}(R_{g_{\nu}}-\cminfbar)\|_{L^{\frac{2n}{n+2}}(M)}
=\left(\int_{ M}|R_{g_{\nu}}-\cminfbar|^{\frac{2n}{n+2}}dv_{g_{\nu}}\right)^{\frac{n+2}{2n}}
<1.
\end{equation}

The estimate (\ref{Lemma6.13}) now follows using the inequality (\ref{Lemma6.12:4}) in (\ref{Lemma6.13:1}).
Proposition \ref{Propo6.14} is a consequence of  Lemma \ref{Lemma6.5} and the estimate (\ref{Lemma6.13}).
\ep
\begin{proposition}\label{Propo6.15}
There exists $c>0$ such that
$$
E(v_{\nu})\leq \left(E(\bar{u}_{z_{\nu}})^{\frac{n}{2}}+\sum_{k=1}^{m}E(\bar{u}_{x_k,\e_{k,\nu}})^{\frac{n}{2}}\right)^{\frac 2n}
-c\sum_{k=1}^{m}\e_{k,\nu}^{\frac{n-2}{2}}
$$
if $\nu$ is sufficiently large.
\end{proposition}
\bp
Choose a permutation $\sigma: \{1,...,m\}$ such that $\e_{\sigma(i),\nu}\leq \e_{\sigma(j),\nu}$ for all $i<j$. During this proof we will omit the symbol $\sigma$, writing $\e_{i,\nu}$ instead of $\e_{\sigma(i),\nu}$, so that $\e_{i,\nu}\leq \e_{j,\nu}$ for all $i<j$.
After calculations similar to the ones in \cite[Proposition 6.15]{brendle-flow}, we obtain
\ba
E(v_{\nu})&\left(\int_{M}v_{\nu}^{\crit}\dv_{g_0}\right)^{\frac{n-2}{n}}\notag
\\
&\leq
\left(E(\bar{u}_{z_{\nu}})^{\frac n2}+\sum_{k=1}^{m}E(\uknu)^{\frac n2}\right)^{\frac 2n}\left(\int_{ M}v_{\nu}^{\crit}\dv_{g_0}\right)^{\frac{n-2}{n}}\notag
\\
&\hspace{0.5cm}-\sum_{k=1}^{m}2\a_{k,\nu}\int_{ M}\Big(\ct\Delta_{g_0}\bar{u}_{z_{\nu}}-\cez\bar{u}_{z_{\nu}}+F(\bar{u}_{z_{\nu}})\bar{u}_{z_{\nu}}^{\frac{n+2}{n-2}}\Big)\,\uknu\dv_{g_0}\notag
\\
&\hspace{0.5cm}-\sum_{i<j}2\a_{i,\nu}\a_{j,\nu}\int_M\ct\frac{\d\ujnu}{\d\eta_{g_0}}\uinu\dv_{g_0}\notag
\\
&\hspace{0.5cm}-\sum_{i<j}2\a_{i,\nu}\a_{j,\nu}\int_{ M}\Big(\ct\Delta_{g_0}\ujnu-\cez\ujnu+F(\ujnu)\ujnu^{\frac{n+2}{n-2}}\Big)\,\uinu\dv_{g_0}\notag
\\
&\hspace{0.5cm}-c\sum_{k=1}^{m}\e_{k,\nu}^{\frac{n-2}{2}}-c\sum_{i<j}\left(\frac{\e_{i,\nu}\e_{j,\nu}}{\e_{j,\nu}^2+d_{g_0}(x_{i,\nu},x_{j,\nu})^2}\right)^{\frac{n-2}{2}}.\notag
\end{align}

Since $F(\bar{u}_{z_{\nu}})\to F(\uinf)=\cminfbar$ as $\nu\to\infty$, we have the estimate
\begin{equation*}
\int_{M}\left|\ct\Delta_{g_0}\bar{u}_{z_{\nu}}-\cez\bar{u}_{z_{\nu}}
+F(\bar{u}_{z_{\nu}})\bar{u}_{z_{\nu}}^{\frac{n+2}{n-2}}\right|\uknu\dv_{g_0}
\leq o(1)\e_{k,\nu}^{\frac{n-2}{2}}\,.
\end{equation*}

Now the assertion follows as in the proof of Proposition \ref{Propo5.6}.
\ep
\begin{corollary}\label{Corol6.16}
Under the hypothesis of Theorem \ref{main:thm}, 
there exist $C>0$ and $0<\gamma<1$ such that
$$
E(v_{\nu})\,\leq\, \cminfbar
+C\left(
\int_{M}u_{\nu}^{\crit}|R_{g_{\nu}}-\cminfbar|^{\frac{2n}{n+2}}\dv_{g_0}
\right)^{\frac{n+2}{2n}(1+\gamma)},
$$
if $\nu$ is sufficiently large.
\end{corollary}
\bp
Using Propositions \ref{Propo:energy:test}, \ref{Propo:energy:test:tubular} and \ref{Propo:energy:test:int}, we obtain $E(\uknu)\leq \Q$ for all $k=1,...,m_1$ and $E(\uknu)\leq \Y$ for all $k=m_1+1,...,m$. Then the result follows from  Propositions \ref{Propo6.14} and \ref{Propo6.15} and \eqref{eq:cminbar}.
\ep

\section{Proof of the main theorem}\label{sec:mainthm}

As in Sections 3 and 7 of \cite{brendle-flow}, the proof of Theorem \ref{main:thm} is carried out in several propositions, whose proofs will be only sketched in what follows.

Let $u(t)$, $t\geq 0$, be the solution of (\ref{eq:evol:u}) obtained in Section \ref{sec:prelim}. The next proposition, which is analogous to \cite[Proposition 3.3]{brendle-flow}, is a crucial step in the argument. 
\begin{proposition}\label{Propo3.3}
Let $\{t_{\nu}\}_{\nu=1}^{\infty}$ be a sequence such that $\lim_{\nu\to\infty}t_{\nu}=\infty$.  Then we can choose $0<\gamma<1$ and $C>0$ such that, after passing to a subsequence, we have
$$
\overline{R}_{g(t_{\nu})}-\cminfbar 
\leq
C\left\{\int_{ M} u(t_{\nu})^{\crit}|R_{g(t_{\nu})}-\cminfbar|
^{\frac{2n}{n+2}}\dv_{g_0}\right\}^{\frac{n+2}{2n}(1+\gamma)}
$$
for all $\nu$. 
\end{proposition}
\bp
It is a long computation using Corollaries \ref{Corol5.5}, \ref{Corol5.7}, \ref{Corol6.10} and \ref{Corol6.16}; see \cite[Section 7]{brendle-flow}.
\ep
\begin{proposition}\label{Propo3.5}
There exists $C>0$ such that
$$
\int_{0}^{\infty}\left\{\int_{ M}u(t)^{\crit}(\cesc-\cescbar)^2\dv_{g_{0}}\right\}^{\frac{1}{2}}dt
\leq C
$$
for all $t\geq 0$.
\end{proposition}
\bp
A simple contradiction argument using Corollary \ref{Corol3.2} and Proposition \ref{Propo3.3} (see \cite[Proposition 3.4]{brendle-flow}) shows that  
there exist $0<\gamma<1$ and $t_0>0$ such that
\begin{equation*}
\cescbar-\cminfbar 
\leq C\left\{\int_{M} u(t)^{\crit}|\cesc-\cminfbar|
^{\frac{2n}{n+2}}\dv_{g_0}\right\}^{\frac{n+2}{2n}(1+\gamma)}
\end{equation*}
for all $t\geq t_0$.
Then it follows that
\ba
\cescbar-\cminfbar
\leq\, &C\left\{\int_{ M} u(t)^{\crit}|\cesc-\cescbar|
^{\frac{2n}{n+2}}\dv
_{g_0}\right\}^{\frac{n+2}{2n}(1+\gamma)}
+C(\cescbar-\cminfbar)^{1+\gamma}\,,\notag
\end{align} 
hence
\begin{equation}\label{Corol3.4}
\cescbar-\cminfbar 
\leq C\left\{\int_{M} u(t)^{\crit}|\cesc-\cescbar|
^{\frac{2n}{n+2}}\dv_{g_0}\right\}^{\frac{n+2}{2n}(1+\gamma)}
\end{equation}
for $t>0$ sufficiently large.
By (\ref{eq:evol:Rbar}) and  \eqref{Corol3.4}, there exists $c>0$ such that 
\ba
\frac{d}{dt}(\cescbar-\cminfbar)
&=-\frac{n-2}{2}\int_{ M}(\cesc-\cescbar)^2\,u(t)^{\crit}\dv_{g_0}\notag
\\
&\leq -\frac{n-2}{2}\left\{
\int_{ M}\big{|}\cesc-\cescbar\big{|}^{\frac{2n}{n+2}}u(t)^{\crit}\dv_{g_0}
\right\}^{\frac{n+2}{n}}\leq -c(\cescbar-\cminfbar)^{\frac{2}{1+\gamma}}\notag
\end{align}
for $t>0$ sufficiently large. Hence, 
$\frac{d}{dt}(\cescbar-\cminfbar)^{-\frac{1-\gamma}{1+\gamma}}\geq c$, which implies
$$
\cescbar-\cminfbar\leq Ct^{-\frac{1+\gamma}{1-\gamma}},\:\:\:\:\text{for}\:t>0\:\text{sufficiently large}.
$$
Then using H\"{o}lder's inequality and the equation (\ref{eq:evol:Rbar}) we obtain
\ba
\int_{T}^{2T}\left(\int_{ M}(\cesc-\cescbar)^2u(t)^{\crit}\dv_{g_0}\right)^{\frac{1}{2}}dt
&\leq
\left(\int_{T}^{2T}dt\right)^{\frac{1}{2}}
\left(\int_{T}^{2T}\int_{ M}(\cesc-\cescbar)^2u(t)^{\crit}\dv_{g_0}\,dt\right)^{\frac{1}{2}}\notag
\\
&=\left\{\frac{2}{n-2}T(\overline{R}_{g(T)}-\overline{R}_{g(2T)})\right\}^{\frac{1}{2}}
\leq CT^{-\frac{\gamma}{1-\gamma}}\notag
\end{align}
for $T$ sufficiently large. This implies
\ba
\int_{0}^{\infty}&\left(\int_{ M}(\cesc-\cescbar)^2u(t)^{\crit}\dv_{g_0}\right)^{\frac{1}{2}}dt\notag
\\
&=\int_{0}^{1}\left(\int_{ M}(\cesc-\cescbar)^2u(t)^{\crit}\dv_{g_0}\right)^{\frac{1}{2}}dt+\sum_{k=0}^{\infty}\int_{2^k}^{2^{k+1}}
\left(\int_{ M}(\cesc-\cescbar)^2u(t)^{\crit}\dv_{g_0}\right)^{\frac{1}{2}}dt\notag
\\
&\leq C\sum_{k=0}^{\infty}2^{-\frac{\gamma}{1-\gamma}k}\leq C\,,\notag
\end{align}
which concludes the proof.
\ep
\begin{proposition}\label{Propo3.7}
There exist $C,c>0$ such that 
\begin{equation}\label{Propo3.7:1}
\sup_{M} u(t)\leq C
\:\:\:\:\text{and}\:\:\:\:
\inf_{M} u(t)\geq c
\,,\:\:\:\:\:\text{for all}\:t\geq 0\,.
\end{equation}
\end{proposition}
\bp
We first claim that, given $\gamma_0>0$, there exists $r>0$ such that
\begin{equation}\label{Propo3.6}
\int_{B_r(x)}u(t)^{\crit}\dv_{g_{0}}\leq \gamma_0,\:\:\:\:\:\:\text{for all}\:t\geq 0,\,x\in M\,.
\end{equation}
Indeed, we can make use of  Proposition \ref{Propo3.5} as in \cite[Proposition 3.6]{brendle-flow} to obtain the above inequality.

Fix $n/2<q<p<(n+2)/2$. According to Corollary \ref{Corol3.2} there is $C_2>0$ such that
$$
\int_{M}|\cesc|^p\dv_{g(t)}\leq C_2\,,
\:\:\:\:\:\text{for all}\:t\geq 0\,.
$$
Set $\gamma_0=\gamma_1^{\frac{p}{p-q}}C_2^{-\frac{q}{p-q}}$, where $\gamma_1$ is the constant obtained in Proposition \ref{PropoA.1}. By \eqref{Propo3.6}, there is $r>0$ such that
$$
\int_{B_r(x)}\dv_{g(t)}\leq \gamma_0\,,\:\:\:\:\:\:\text{for all}\:t\geq 0,\,x\in M\,.
$$
Then
$$
\int_{B_r(x)}|\cesc|^q\dv_{g(t)}
\leq 
\left\{\int_{B_r(x)}\dv_{g(t)}\right\}^{\frac{p-q}{p}}
\left\{\int_{B_r(x)}|\cesc|^p\dv_{g(t)}\right\}^{\frac{q}{p}}
\leq \gamma_1\,.
$$
Hence, the first assertion of (\ref{Propo3.7:1}) follows from Proposition \ref{PropoA.1}. The second one follows exactly as in the proof of the second estimate of (\ref{Propo2.4:1}). 
\ep
\bp[Proof of Theorem \ref{main:thm}]
Once we have proved Proposition \ref{Propo3.7}, it follows as in \cite[p.229]{brendle-flow} that all higher order derivatives of $u$ are uniformly bounded. The uniqueness of the asymptotic limit of $\cesc$ follows from Proposition \ref{Propo3.5}.
\ep


\begin{appendices}
\renewcommand{\theequation}{A-\arabic{equation}}
\setcounter{equation}{0}
\renewcommand{\thetheorem}{A-\arabic{theorem}}
\setcounter{theorem}{0}
\section{Some elliptic estimates}

Let $(M^n,g)$ be a complete  Riemannian manifold with boundary $\partial M$ and dimension $n\geq 3$, and let $\eta_g$ be its unit normal vector pointing inwards.
\begin{definition}
We say that $u\in H^1(M)$ is a {\it{subsolution}} (resp. {\it{supersolution}}) of 
 \begin{align}\label{system}
  \begin{cases}
   \Delta_gu+Pu= f\,,&\text{in}\:M\,,\\
   \partial u/\partial \eta_g+\bar Pu= \bar f\,,&\text{on }\d M\,.
  \end{cases}
 \end{align}
if, for all $0\leq v\in C_c^1(M)$, the following quantity is nonpositive (resp. nonnegative)
$$
\int_{M}(\langle du, dv\rangle_g - Puv+f v)dv_g+\int_{\d M}(-\bar Puv+\bar f v)d\sigma_g.
$$
\end{definition}

The next proposition is similar to \cite[Theorems 8.17 and 8.18]{gilbarg-trudinger}; see also \cite[Lemma A.1]{han-li}.
\begin{proposition}\label{Propo:estim:Lp}
Let $q>n$, $s>n-1$ and $P\in L^{q/2}(M)$, $\bar P\in L^{s}(\d M)$ with $||P||_{L^{q/2}}(M)+||\bar P||_{L^{s}}(\d M)\leq \Lambda$.
 
   (a) For any $p>1$, there exists $C=C(n,p,q,s,g,\Lambda)$ and $r_0=r_0(M,g)$ such that
  $$
\sup_{B^+_r(x)}u\leq Cr^{-\frac{n}{p}}||u||_{L^p(B^+_{2r}(x))}
+Cr^{2-\frac{2n}{q}}||f||_{L^{q/2}(B^+_{4r}(x))}
+Cr^{1-\frac{n-1}{s}}||\bar f||_{L^s(D_{4r}(x))}
$$
  for any $x\in\d M$, $r<r_0$ and $0\leq u\in H^1(M)$ subsolution of \eqref{system}.
 
 (b) If $1\leq p<\frac{n}{n-2}$, there exists $C=C(n,p,q,s,g,\Lambda)$ and $r_0=r_0(M,g)$ such that
 $$
r^{-\frac{n}{p}}||u||_{L^p(B^+_{2r}(x))}\leq C\inf_{B^+_r(x)}u
+Cr^{2-\frac{2n}{q}}||f||_{L^{q/2}(B^+_{4r}(x))}
+Cr^{1-\frac{n-1}{s}}||\bar f||_{L^s(D_{4r}(x))}
$$
 for any $x\in\d M$, $r<r_0$ and $0\leq u\in H^1(M)$  supersolution of \eqref{system}.
\end{proposition}
\bp
After rescaling we can assume $r=1$.
Let $\b\neq 0$, $k=||f||_{L^{q/2}(B_4^+)}+||\bar f||_{L^{s}(D_4)}$ and $0\leq \chi\in C_c^1(B_4^+)$. We will assume that $k>0$. The general case will follow by tending $k$ to zero. Set  $\bar u=u+k$.

If $u$ is a subsolution, by definition we have
$$
\int_{M}\langle du, d(\chi^2\bar u^{\b})\rangle_gdv_g\leq  \int_{M}( Pu-f)\chi^2\bar u^{\b}dv_g+\int_{\d M}(\bar Pu-\bar f)\chi^2\bar u^{\b}d\sigma_g,
$$
and we have the opposite inequality in case $u$ is a supersolution. Choosing $\b>0$ should $u$ be a subsolution and $\b<0$ should $u$ be a supersolution, 
 in both cases we obtain
\begin{align}\label{Propo:estim:Lp:1}
\int_{M}\chi^2\bar u^{\b-1}&|d\bar u|^2_gdv_g
\leq  
|\b|^{-1}\int_{M}2\chi\bar u^{\b}|d\chi|_g |d\bar u|_g\,dv_g
\\
&+|\b|^{-1}\int_{M}\chi^2( |P|+k^{-1}|f|)\bar u^{\b+1}dv_g+|\b|^{-1}\int_{\d M}\chi^2(|\bar P|+k^{-1}|\bar f|)\bar u^{\b+1}d\sigma_g\notag
\end{align}
by means of $\langle du, d(\chi^2\bar u^{\b})\rangle_g=2\chi\bar u^{\b}\langle d\chi, d\bar u\rangle_g+\b\chi^2\bar u^{\b-1}|d\bar u|^2_g$.
Applying Young's inequality to the last term of \eqref{Propo:estim:Lp:1} we arrive at
\begin{align}\label{Propo:estim:Lp:2}
\int_{M}\chi^2\bar u^{\b-1}&|d\bar u|^2_gdv_g
\leq  
C|\b|^{-2}\int_{M}|d\chi|_g^2\bar u^{\b+1} \,dv_g
\\
&+C|\b|^{-1}\int_{M}\chi^2( |P|+k^{-1}|f|)\bar u^{\b+1}dv_g+C|\b|^{-1}\int_{\d M}\chi^2(|\bar P|+k^{-1}|\bar f|)\bar u^{\b+1}d\sigma_g.\notag
\end{align}
Set $h=|P|+k^{-1}|f|$, $\bar h=|\bar P|+k^{-1}|\bar f|$ and
$$
w=
\begin{cases}
\bar u^{\frac{\b+1}{2}}&\text{if}\:\b\neq -1,
\\
\log \bar u&\text{if}\:\b= -1.
\end{cases}
$$
Then \eqref{Propo:estim:Lp:2} can be rewritten as 
\begin{align}\label{Propo:estim:Lp:3}
\int_{M}\chi^2&|dw|^2_gdv_g
\leq  
C\frac{(\b+1)^2}{|\b|^2}\int_{M}|d\chi|_g^2 w^2 \,dv_g
\\
&+C\frac{(\b+1)^2}{|\b|}\int_{M}\chi^2 h w^2dv_g+C\frac{(\b+1)^2}{|\b|}\int_{\d M}\chi^2\bar h w^2 d\sigma_g\notag
\end{align}
if $\b\neq -1$ and
\begin{align}\label{Propo:estim:Lp:4}
\int_{M}\chi^2|dw|^2_gdv_g
&\leq  
C\int_{M}|d\chi|_g^2 \,dv_g
+C\int_{M}\chi^2 h dv_g+C\int_{\d M}\chi^2\bar h d\sigma_g
\end{align}
if $\b=-1$. 
It follows from $\chi^2|dw|_g^2\geq \frac{1}{2}|d(\chi w)|_g^2-w^2|d\chi|_g^2$ and Sobolev inequalities that 
\begin{equation}\label{Propo:estim:Lp:5}
\left(\int_M(\chi w)^{\frac{2n}{n-2}}dv_g\right)^{\frac{n-2}{n}}-C\int_M|d\chi|_g^2w^2dv_g\leq C\int_{M}\chi^2|dw|^2_gdv_g
\end{equation}
In order to handle the right hand side of \eqref{Propo:estim:Lp:3} we use H\"older's and interpolation inequalities to get
\begin{align}\label{Propo:estim:Lp:6}
\int_{M}\chi^2  h w^2dv_g
&\leq 
\|h\|_{L^{q/2}(B_4^+)}\|\chi w\|^2_{L^{2q/(q-2)}(B_4^+)}
\\
&\leq 
\|h\|_{L^{q/2}(B_4^+)}(\e^{1/2}\|\chi w\|_{L^{2n/(n-2)}(B_4^+)}+\e^{-\mu_1/2}\|\chi w\|_{L^2(B_4^+)})^2\notag
\\
&\leq 
2\|h\|_{L^{q/2}(B_4^+)}(\e\|\chi w\|^2_{L^{2n/(n-2)}(B_4^+)}+\e^{-\mu_1}\|\chi w\|^2_{L^2(B_4^+)})\notag
\end{align}
where $\mu_1=n/(q-n)$, and
\begin{align}\label{Propo:estim:Lp:7}
\int_{\d M}\chi^2 \bar h w^2d\s_g
&\leq 
\|\bar h\|_{L^{s}(D_4)}\|\chi w\|^2_{L^{2s/(s-1)}(D_4)}
\\
&\leq 
\|\bar h\|_{L^{s}(D_4)}(\e^{1/2}\|\chi w\|_{L^{2(n-1)/(n-2)}(D_4)}+\e^{-\mu_2/2}\|\chi w\|_{L^2(D_4)})^2\notag
\\
&\leq 
2\|\bar h\|_{L^{s}(D_4)}(\e\|\chi w\|^2_{L^{2(n-1)/(n-2)}(D_4)}+\e^{-\mu_2}\|\chi w\|^2_{L^2(D_4)})\notag
\end{align}
where $\mu_2=(n-1)/(s+1-n)$. It follows from the Sobolev embedding theorems that 
\begin{equation*}
\e^{-\mu_2}\int_{D_4}(\chi w)^2d\s_g\leq 
\e\int_{B_4^+}|d(\chi w)|_g^2dv_g+\e^{-2\mu_2-1}\int_{B_4^+}(\chi w)^2dv_g
\end{equation*}
and
\begin{equation*}
\Big(\int_{D_4}(\chi w)^{\frac{2(n-1)}{n-2}}d\s_g\Big)^{\frac{n-2}{n-1}}
\leq 
C\int_{B_4^+}|d(\chi w)|_g^2dv_g.
\end{equation*}
Then the inequality \eqref{Propo:estim:Lp:7} becomes
\begin{align}\label{Propo:estim:Lp:8}
\int_{\d M}\chi^2 \bar h w^2d\s_g
&\leq 
C\e\|\bar h\|_{L^{s}(D_4)}\int_{B_4^+}|d(\chi w)|_g^2dv_g
+C\e^{-2\mu_2-1}\|\bar h\|_{L^{s}(D_4)}\int_{B_4^+}(\chi w)^2dv_g.
\end{align}

Choosing $\e=c|\b|(\b+1)^{-2}\Lambda^{-1}$ with $c>0$ small, we can make use of the inequalities \eqref{Propo:estim:Lp:5}, \eqref{Propo:estim:Lp:6}, \eqref{Propo:estim:Lp:7} and \eqref{Propo:estim:Lp:8} in \eqref{Propo:estim:Lp:3} to obtain
\begin{equation}\label{Propo:estim:Lp:9}
\Big(\int_{B_4^+}(\chi w)^{\frac{2n}{n-2}}dv_g\Big)^{\frac{n-2}{n}}
\leq
C(1+|\gamma|)^{2\mu}\int_{B_4^+}(|d\chi|_g^2+\chi^2)w^2dv_g.
\end{equation}
Here,  $\gamma=\b+1$, $\mu=max\{\mu_1+1,2\mu_2+2\}$, and $C$ depends on $\Lambda$ and is bounded when $|\b|$ is bounded away from zero.

For any $1\leq r_a\leq r_b\leq 3$ we choose $\chi$ as a cut-off function satisfying $0\leq \chi\leq 1$, $|d\chi|\leq 2/(r_b-r_a)$ and 
\begin{align*}
\begin{cases}
\chi\equiv 1&\text{in}\:B_{r_a}^+,
\\
\chi\equiv 0&\text{in}\:B_4^+\backslash B_{r_b}^+.
\end{cases}
\end{align*}
Using this in \eqref{Propo:estim:Lp:9} we obtain
\begin{equation}\label{Propo:estim:Lp:10}
\Big(\int_{B_{r_a}^+}\bar u^{\frac{\gamma n}{n-2}}dv_g\Big)^{\frac{n-2}{n}}
\leq
\frac{C(1+|\gamma|)^{2\mu}}{r_b-r_a}\int_{B_{r_b}^+}\bar u^{\gamma}dv_g.
\end{equation}
If we set $\Phi(e,r)=\Big(\int_{B_r^+}\bar u^edv_g\Big)^{1/e}$ and $\delta=n/(n-2)$, the estimate \eqref{Propo:estim:Lp:10} becomes
\begin{align}\label{Propo:estim:Lp:11} 
\begin{cases}
\displaystyle
\Phi(\delta\gamma,r_a)\leq \left(\frac{C(1+|\gamma|)^{\mu}}{r_b-r_a}\right)^{\frac{2}{|\gamma|}}\Phi(\gamma,r_b)&\text{if}\:\:\gamma>0,
\\
\displaystyle
\Phi(\gamma,r_b)\leq \left(\frac{C(1+|\gamma|)^{\mu}}{r_b-r_a}\right)^{\frac{2}{|\gamma|}}\Phi(\delta\gamma,r_a)&\text{if}\:\:\gamma<0.
\end{cases}
\end{align}
It is well known that $\lim_{e\to\infty}\Phi(e,r)=\sup_{B_r^+}\bar u$ and $\lim_{e\to -\infty}\Phi(e,r)=\inf_{B_r^+}\bar u$.
The rest of the proof follows as in \cite[p.197-198]{gilbarg-trudinger} by iterating the first  inequality in \eqref{Propo:estim:Lp:11} to prove (a), and by using \eqref{Propo:estim:Lp:4} and iterating the second inequality in \eqref{Propo:estim:Lp:11} to prove (b).
\ep

Once we have established Proposition \ref{Propo:estim:Lp}(a), the proof of the next proposition is similar to \cite[Proposition A.3]{almaraz5}.
\begin{proposition}\label{PropoA.1}
Let $(M^n,g_0)$ be a compact Riemannian manifold with boundary $\d M$ and with dimension $n\geq 3$. 
For each $q>n/2$ we can find positive constants $\gamma_1=\gamma_1(M,g_0,q)$ and $C=C(M,g_0,q)$ with the following significance: if $g=u^{\frac{4}{n-2}}g_0$ is a conformal metric satisfying 
$$
\int_{M}\dv_{g}\leq 1
\:\:\:\:\:\:
\text{and}
\:\:\:\:\:\:
\int_{B_r(x)}|R_g|^q\,\dv_g\leq \gamma_1
$$
for $x\in M$, then we have
$$
u(x)\leq Cr^{-\frac{n-2}{2}}\left(\int_{B_r(x)}\dv_g\right)^{\frac{n-2}{2n}}\,.
$$
\end{proposition}
Using Proposition \ref{Propo:estim:Lp}(b) and interior Harnack estimates for elliptic linear equations (see \cite[Theorem 8.18]{gilbarg-trudinger}), one can prove the next proposition by adapting the arguments in \cite[Proposition A.2]{brendle-flow}.
\begin{proposition}\label{CorolA.3}
 Let $(M,g_0)$ be a Riemannian manifold with boundary $\partial M$, $P$ a smooth function on $M$, and suppose $u$ that satisfies
 \begin{equation*}
  \begin{cases}
   \displaystyle -\Delta_{g_0}u(t)+Pu\geq 0\,,&\text{in}\:M\,,\\
   \displaystyle \frac{\d}{\d \eta_{g_0}}u=0\,,&\text{on}\:\d M\,.
  \end{cases}
 \end{equation*}
Then there exists $C=C(P,g_0)$ such that
$$
C\inf_Mu\geq \int_M u\dv_{g_0}.
$$
In particular,
$$
\int_{M}u^{\frac{2n}{n-2}}dv_{g_0}\leq C\inf_Mu\left(\sup_Mu\right)^{\frac{n+2}{n-2}}.
$$
\end{proposition}

\renewcommand{\theequation}{B-\arabic{equation}}
\setcounter{equation}{0}
\renewcommand{\thetheorem}{B-\arabic{theorem}}
\setcounter{theorem}{0}
\section{Construction of the Green function on manifolds with boundary}

In this section, we prove the existence of the Green function used in this paper and
some of its properties. The construction performed here extends the one in \cite[Proposition B-2]{almaraz5}; see also \cite[p.201]{druet-hebey-robert} and  \cite[p.106]{aubin}.
\begin{lemma}\label{lemma:holder}
Let $(M,g)$ be a connected Riemannian manifold of dimension $n\geq 2$ and fix $x\in M$ and $\a\in \R$.
Let $u:M\backslash\{x\}\to \R$ be a function satisfying
$$
|u(y)|\leq C_0d_g(x,y)^{\a}
\:\:\:\:\text{and}\:\:\:\:
|\nabla_gu(y)|_g\leq C_0d_g(x,y)^{\a-1}\,,
$$
for any $y\in M$, with $x\neq y$. Then, for any $0<\theta\leq 1$, there exists $C_1=C_1(M,g,C_0,\a)$ such that
$$
|u(y)-u(z)|\leq C_1d_g(y,z)^{\theta}(d_g(x,y)^{\a-\theta}+d_g(x,z)^{\a-\theta})
$$
for any $y,z\in M$, with $y\neq x\neq z$.
\end{lemma}
This is \cite[Lemma B.1]{almaraz5}. For the reader's convenience, we provide the proof here.
\bp
Let $y\neq x$ and $z\neq x$.

\vspace{0.1cm}\noindent
{\underline{1st case:}} $d_g(y,z)\leq \frac{1}{2}d_g(x,y)$.
Let $\gamma:[0,1]\to M$ be a smooth curve such that $\gamma(0)=y$, $\gamma(1)=z$, and
$\int_0^1|\gamma'(t)|_gdt\leq\frac{3}{2}d_g(y,z)$.

\vspace{0.1cm}\noindent
{\it{Claim.}} We have
$\frac{1}{4}d_g(x,y)\leq d_g(\gamma(t),x)\leq\frac{7}{4}d_g(x,y)$.

Indeed, since $d_g(y,\gamma(t))\leq\frac{3}{2} d_g(y,z)\leq\frac{3}{4}d_g(x,y)$, we have
$$
d_g(x,\gamma(t))\geq d_g(x,y)-d_g(\gamma(t),y)\geq
d_g(x,y)-\frac{3}{4}d_g(x,y)=\frac{1}{4}d_g(x,y)\,.
$$
Moreover,
$$
d_g(\gamma(t),x)\leq d_g(\gamma(t),y)+d_g(y,x)\leq
\frac{3}{4}d_g(x,y)+d_g(x,y)=\frac{7}{4}d_g(x,y)\,.
$$
This  proves the claim.

Observe that $u(z)-u(y)=\int_0^1g(\nabla_gu(\gamma(t)),\gamma'(t))\,dt$. Thus, 
\ba 
|u(y)-u(z)| 
&\leq \sup_{t\in
[0,1]}|\nabla_gu(\gamma(t))|_g\int_0^1|\gamma'(t)|_gdt\leq C\sup_{t\in
[0,1]}d_g(\gamma(t),x)^{\a-1}\frac{3}{2}d_g(y,z)\notag
\\
&\leq C(\a)d_g(x,y)^{\a-1}d_g(y,z)\leq C(\a)d_g(x,y)^{\a-\theta}d_g(y,z)^{\theta}\,.\notag
\end{align}
\noindent
{\underline{2nd case:}} $d_g(y,z)> \frac{1}{2}d_g(x,y)$.
In this case, we have
\ba
|u(y)-u(z)|
&\leq |u(y)|+|u(z)|\leq Cd_g(y,x)^{\a}+Cd_g(z,x)^{\a}\notag
\\
&\leq
Cd_g(y,x)^{\a-\theta}d_g(z,y)^{\theta}+Cd_g(z,x)^{\a-\theta}(d_g(x,y)+d_g(y,z))^{\theta}\notag
\\
&\leq Cd_g(y,z)^{\theta}(d_g(x,y)^{\a-\theta}+d_g(x,z)^{\a-\theta})\,.\notag
\end{align}
\ep
Let  $(M,g)$ be a compact Riemannian manifold with boundary $\d M$, dimension
$n\geq 3$, and positive Sobolev quotient $Q(M)$.
\begin{notation}
We denote by $L_g$ the conformal Laplacian
$\Delta_g-\frac{n-2}{4(n-1)}R_g$, and by $B_g$ the boundary conformal operator
$\frac{\d}{\d\eta_g}-\frac{n-2}{2(n-1)}H_g$, where $\eta_g$ is the inward unit normal
vector to $\d M$.
\end{notation}

Set $d(x)=d_g(x,\d M)$ for $x\in M$, and $M_{\rho}=\{x\in M\,;\:d(x)<\rho\}$ for $\rho>0$.
Choose $\tilde{\rho}_0=\tilde{\rho}_0(M,g)>0$ small such that the function
\ba
M_{2\tilde{\rho}_0}&\to \d M\notag
\\
x&\mapsto \bar{x}\notag
\end{align}
is well defined and smooth, where $\bar{x}$ is defined by $d_g(x,\bar{x})=d_g(x,\d M)$, and $\tilde{\rho}_0/4$ is smaller than the injectivity radius of $M$. Then, for any $0<t<2\tilde{\rho}_0$, the set $\d_tM=\{x\in M\,;\:d(x)=t\}$ is a smooth embedded $(n-1)$-submanifold of $M$. For each $x\in M_{\tilde{\rho}_0}$, define the function
\ba
M_{2\tilde{\rho}_0}&\to \d_{d(x)} M\notag
\\
y&\mapsto y_x\,,\notag
\end{align}
where $y_x$ is defined by $d_g(y,y_x)=d_g(y,\d_{d(x)} M)$. 

For any $x\in M_{\rho_0}$ and $\rho_0\in(0,\tilde{\rho}_0)$, we define the local coordinates $\psi_x(y)=(y_1,...,y_n)$ on $M_{2\rho_0}$, where $y_n=d(y)$, and $(y_1,...,y_{n-1})$ are normal coordinates of $y_x$, centered at $x$, with respect to the submanifold $\d_{d(x)}M$ . Then $(x,y)\mapsto \psi_x(y)$ is locally defined and smooth.
Observe that $\psi_x(x)=(0,...,0,d(x))$ for any $x\in M_{{\rho}_0}$, and that $\psi_x$ are Fermi coordinates if $x\in\d M$.
Moreover, in those coordinates we have $g_{an}\equiv \delta_{an}$ and $g_{ab}(x)=\delta_{ab}$, for $a,b=1,...,n$, and the inward normal unit vector to $\d M$ is $d\psi_x^{-1}(\d/\d y_n)$, see figure 2. Choosing $\tilde{\rho}_0$ possibly smaller,  we can assume that, for any $x\in M_{\tilde{\rho}_0}$, $\psi_x(y)=(y_1,...,y_n)$ is defined for $0\leq y_n <2\tilde{\rho}_0$ and $|(y_1,...,y_{n-1})|<\tilde{\rho}_0$.
\begin{figure}\label{fig:coord}
\begin{center}
 \begin{tikzpicture}
\begin{scope}
\draw (0,0) to [out=20,in=190] (2,.6) to [out=10,in=170] (4,.5) to [out=-10,in=190](6,.8);
 \draw (0,-2) to [out=20,in=190] (2,-1.4) to [out=10,in=170] (4,-1.5) to [out=-10,in=190](6,-1);
 \draw (2,.6) to [out=280,in=100](2,-1.4);
 \draw (4,0.5) to [out=260,in=80] (4,-1.5);
 \node at (2,.9) {$y_x$};
 \node at (2,-1.7) {$\bar y$};
 \draw[black,fill=red] (2.03,-.2) circle (.2ex);
 \node at (1.7, -.3) {$y$};
 \node at (4,.8) {$x$};
 \node at (3.9,-1.8) {$\bar x$};
 \node at (6.7,0.6) {$\partial_{d(x)}M$};
 \node at (6.7,-1.3) {$\partial M$};
\draw[decoration={brace,mirror,raise=5pt},decorate] (4,-1.5)-- node[right=6pt] {$d(x)$} (4, .4);
\draw[decoration={brace,mirror,raise=5pt},decorate] (2,-1.3)-- node[right=6pt] {$d(y)$} (2, -0.2);
 \end{scope}
 \end{tikzpicture}
\end{center}
\caption{Illustration of the notations.}
\end{figure}
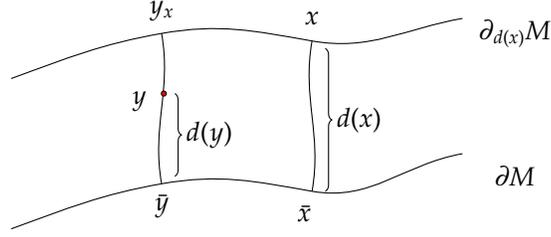

\begin{proposition}\label{green:point}
Let  $\rho_0\in(0,\tilde{\rho}_0)$, $x_0\in M$ and $d=\left[\frac{n-2}{2}\right]$.
Suppose that one of the following conditions holds:
\\
(a)  $x_0\in \d M$ and there exist $C=C(M,g)$ and $N$ sufficiently large such that
\begin{equation}
H_g(y)\leq Cd_g(x_0,y)^{N}\,,\:\:\:\:\text{for all}\:y\in \d M\,;
\end{equation}
(b) $x_0\in M_{\rho_0/2}$ and $H_g\equiv 0$ on $\d M$;
\\
(c) $x_0\in M\backslash M_{2\rho_0}$.

\vspace{0.1cm}
Then there exists a positive $G_{x_0}\in C^{\infty}(M\backslash \{x_0\})$ satisfying
\begin{equation}\label{eq:G}
\begin{cases}
L_{g}G_{x_0}=0\,,&\text{in}\:M\backslash \{x_0\}\,,
\\
B_{g}G_{x_0}=0\,,&\text{on}\:\d M\backslash \{x_0\}\,,
\end{cases}
\end{equation}
\begin{equation}\label{green:formula:0}
\phi(x_0)=-\int_M G_{x_0}(y)L_g\phi(y)\dv_g(y)-\int_{\d M}G_{x_0}(y)B_g\phi(y)\ds_g(y)
\end{equation}
for any $\phi\in C^2(M)$. Moreover, the following properties hold:

\vspace{0.2cm}\noindent
(P1) There exists $C=C(M,g)$ such that, for any $y\in M$ with $y\neq x_0$,
$$
|G_{x_0}(y)|\leq Cd_g(x_0,y)^{2-n}\:\:\:\:\text{and}\:\:\:\:
|\nabla_g G_{x_0}(y)|\leq Cd_g(x_0,y)^{1-n}\,.
$$
(P2) If $x_0\in \d M$ consider Fermi coordinates $y=(y_1,...,y_n)$ centered at that point. In those coordinates, write $g_{ab}=\exp(h_{ab})$, $a,b=1,...,n$, where 
\begin{equation}\label{eq:h}
\Big|h_{ab}(y)-\sum_{|\a|=1}^{d}h_{ab,\a}y^{\a}\Big|\leq C(M,g)|y|^{d+1},
\end{equation}
where $h_{ab,\a}\in \R$ and each $\a$ stands for a multi-index.
Then there exists $C=C(M,g,\rho_0)$ such that \footnote{The $\log$ term in dimensions $3$ and $4$ should also be included in \cite[Proposition B-1]{almaraz5}. However, that term does not affect the results in \cite{almaraz5} as observed in the footnote in Proposition \ref{Propo1:notes5} above.}
\begin{align}\label{est:G}
&\big|G_{x_0}(y)-\frac{2|y|^{2-n}}{(n-2)\sigma_{n-1}}\big|\leq
C\sum_{a,b=1}^{n-1}\sum_{|\a|=1}^{d}|h_{ab,\a}|d_g(x_0,y)^{|\a|+2-n}+
\begin{cases}
Cd_g(x_0,y)^{d+3-n}&\text{if}\:n\geq 5,
\\
C(1+|\log d_g(x_0,y)|)&\text{if}\:n=3,4,
\end{cases}
\\
&\big|\nabla_g(
G_{x_0}(y)-\frac{2|y|^{2-n}}{(n-2)\sigma_{n-1}}
)\big|\leq
C\sum_{a,b=1}^{n-1}\sum_{|\a|=1}^{d}|h_{ab,\a}|d_g(x_0,y)^{|\a|+1-n}+Cd_g(x_0,y)^{d+2-n}.\notag
\end{align}

\vspace{0.2cm}\noindent
(P3) If $x_0\in M_{\rho_0/2}$ consider the coordinate system $\psi_{x_0}$ defined above. Then there exists $C=C(M, g, \rho_0)$ such that 
$$
\big|G_{x_0}(y)-\frac{1}{(n-2)\sigma_{n-1}}[|(y_1,...,y_{n-1}, y_n-d(x_0))|^{2-n}+|(y_1,...,y_{n-1}, y_n+d(x_0))|^{2-n}]\big|
\leq Cd_g(x_0,y)^{3-n},
$$
$$
\big|\nabla_g\big(G_{x_0}(y)-\frac{1}{(n-2)\sigma_{n-1}}[|(y_1,...,y_{n-1}, y_n-d(x_0)))|^{2-n}+|(y_1,...,y_{n-1}, y_n+d(x_0))|^{2-n}]\big)\big|
\leq Cd_g(x_0,y)^{2-n},
$$
if $n\geq 4$ and
$$
\big|G_{x_0}(y)-\frac{1}{(n-2)\sigma_{n-1}}[|(y_1,...,y_{n-1}, y_n-d(x_0))|^{2-n}+|(y_1,...,y_{n-1}, y_n+d(x_0))|^{2-n}]\big|
\leq C(1+|\log d_g(x_0,y)|)\,,
$$
$$
\big|\nabla_g\big(G_{x_0}(y)-\frac{1}{(n-2)\sigma_{n-1}}[|(y_1,...,y_{n-1}, y_n-d(x_0))|^{2-n}+|(y_1,...,y_{n-1}, y_n+d(x_0))|^{2-n}]\big)\big|
\leq Cd_g(x_0,y)^{-1}\,,
$$
if $n=3$.

\vspace{0.2cm}\noindent
(P4) If $x_0\in M\backslash M_{2\rho_0}$ consider normal coordinates $y=(y_1,...,y_n)$ centered at that point. As in (P2), write $g_{ab}=\exp(h_{ab})$ where $h_{ab}$ satisfies (\ref{eq:h}).
Then there exists $C=C(M,g,\rho_0)$ such that the estimates (\ref{est:G}) hold. (Observe that in this case the sums range from $|\a|=2$ to $d$ instead of from $|\a|=1$ to $d$.)
\end{proposition}
\begin{remark}
The indentity \eqref{green:formula:0} and the estimates in $(P2)$ and $(P3)$ may change according to the normalization chosen for $G_{x_0}$. Notice that different ones have been used in the rest of the paper.
\end{remark}
\bp
Let $\chiup:\R_+\to [0,1]$ be a smooth cutoff function satisfying $\chiup(t)=1$ for $t<\rho_0/2$, and $\chiup(t)=0$ for $t\geq \rho_0$. For each $x\in M_{\rho_0}$, set
\ba
K_1(x,y)=\frac{\chiup(y_n/2)\chiup(|(y_1,...,y_{n-1})|)}{(n-2)\sigma_{n-1}}\cdot\left\{|(y_1,...,y_{n-1},y_n-d(x))|^{2-n}+|(y_1,...,y_{n-1},y_n+d(x))|^{2-n}\right\}\,,\notag
\end{align}
where we are using the coordinates $\psi_x(y)=(y_1,...,y_n)$. Observe that
$$
\sum_ {a=1}^n\frac{\d^2}{\d y_ a^2}K_1(x,y)=0 \,,\:\:\:
\text{for}\:\:|(y_1,...,y_{n-1})|<\rho_0/2\,,\:0\leq y_n<\rho_0\,,\:\text{and}\:x\neq y\,.
$$
Moreover,  $\d K_1/\d y_n (x,y)=0$ if $y\in \d M$ with $x\neq y$.

For each $x\in M\backslash M_{\rho_0/2}$, set
$$
K_2(x,y)=\frac{\chiup(4d_g(y,x))}{(n-2)\sigma_{n-1}}d_g(y,x)^{2-n}\,,\:\:\:\:\text{if}\:\:0<d_g(y,x)<\rho_0/4\,.
$$
If we express $y\mapsto K_2(x,y)$ in normal coordinates $(y_1,...,y_n)$ centered at $x$, we have $K_2(x,y)=\chiup(4|(y_1,...,y_n)|)|(y_1,...,y_n)|^{2-n}$, and thus
$$
\sum_ {a=1}^n\frac{\d^2}{\d y_ a^2}K_2(x,y)=0 \,,\:\:\:\:
\text{for}\:\:0<d_g(y,x)<\rho_0/8\,.
$$

Define $K:M\times M\backslash D_M\to \R$ by the expression
$$
K(x,y)=\chiup(d(x))K_1(x,y)+(1-\chiup(d(x)))K_2(x,y)\,,
$$
where $D_M=\{(x,x)\in M\times M\,;\:x\in M\}$.
Thus, $K(x,y)=K_1(x,y)$ if $x\in M_{\rho_0/2}$, and $K(x,y)=K_2(x,y)$ if $x\in M\backslash M_{\rho_0}$.
Observe that $\d K/\d\eta_{g,y}(x,y)=0$ if $y\in\d M$ with $y\neq x$.

Expressing  $y\mapsto K_1(x,y)$ and $y\mapsto K_2(x,y)$ in their respective coordinate systems (as
described above)  one can check that there exists  $C=C(M,g,\rho_0)$ such that
$$
|L_{g,y}K(x,y)|\leq Cd_g(x,y)^{1-n}\,.
$$

For any $\phi\in C^2(M)$ and $x\in M$, we have
\begin{align}\label{form:green:H}
\phi(x)&=\int_M\big( \Delta_{g,y}K(x,y)\phi(y)-K(x,y)\Delta_g\phi(y)\Big)\dv_g(y)-\int_{\d M} K(x,y)\frac{\d}{\d\eta_g}\phi(y)\ds_g(y)\,.
\end{align}
Indeed, this expression holds with $K_1(x,y)$ replacing $K(x,y)$ when $x\in M_{\rho_0/2}$, and with $K_2(x,y)$ replacing $K(x,y)$ when $x\in M\backslash M_{\rho_0}$.

We define $\Gamma_{k}:M\times M\backslash D_M\to \R$ inductively by setting
$$
\Gamma_1(x,y)=L_{g,y}K(x,y)
$$
and
$$
\Gamma_{k+1}(x,y)=\int_M\Gamma_{k}(x,z)\Gamma_1(z,y)\dv_g(z)\,.
$$
According to  \cite[Proposition 4.12]{aubin}, which is a result due to Giraud (\cite[p.50]{giraud}), we have
\begin{equation}\label{estim:gamma}
|\Gamma_k(x,y)|\leq
\begin{cases}
Cd_g(x,y)^{k-n}\,,\:\:&\text{if}\:k<n\,,
\\
C(1+|\log d_g(x,y)|)\,,\:\:&\text{if}\:k=n\,,
\\
C\,,\:\:&\text{if}\:k>n\,,
\end{cases}
\end{equation}
for some $C=C(M,g,\rho_0)$.
Moreover, $\Gamma_{k}$ is continuous on $M\times M$ for $k>n$, and on $M\times M\backslash D_M$ for $k\leq n$.

If (a) or (b) holds we can refine the estimate (\ref{estim:gamma}) around the point $x_0$, using the expansion $g_{ab}=\exp(h_{ab})$. 
Since $K(x,y)=K_1(x,y)$ for $x\in M_{\rho_0/2}$ and $K(x,y)=K_2(x,y)$ for $x\in M\backslash M_{\rho_0}$, one can see that
$$
|L_{g,y}K(x_0,y)|\leq C\sum_{a,b=1}^n\sum_{|\a|=1}^{d}|h_{ab,\a}|d_g(x_0,y)^{|\a|-n}+Cd_g(x_0,y)^{d+1-n},
$$
for some $C=C(M,g,\rho_0)$, if (a) or (b) holds.
Then Giraud's result implies
\begin{equation}\label{estim:gamma:h}
|\Gamma_k(x_0,y)|\leq  C
\sum_{a,b=1}^n\sum_{|\a|=1}^{d}|h_{ab,\a}|d_g(x_0,y)^{k-1+|\a|-n}+d_g(x_0,y)^{k+d-n}\,,\:\:\text{if}\:k<n-d\,.
\end{equation}

\vspace{0.2cm}\noindent
{\it{Claim 1.}} Given $0<\theta<1$, there exists $C=C(M,g,\rho_0,\theta)$ such that
\begin{equation}\label{estim:claim1}
|\Gamma_{n+1}(x,y)-\Gamma_{n+1}(x,y')|\leq Cd_g(y,y')^{\theta}\,, \:\:\:\: \text{for any}\:y\neq
x\neq y'\,.
\end{equation}
In particular, $\Gamma_{n+1}(x_0,\cdot)\in C^{0,\theta}(M )$.

Indeed, observe that
$
|\Gamma_1(x,y)-\Gamma_1(x,y')|\leq Cd_g(y,y')^{\theta}(d_g(x,y)^{1-\theta-n}+d_g(x,y')^{1-\theta-n})\,,
$
according to Lemma \ref{lemma:holder}. So, Claim 1 follows from the estimates (\ref{estim:gamma}) and Giraud's result.

Set
$$
F_{k}(x,y)=K(x,y)+\sum_{j=1}^{k}\int_M \Gamma_j(x,z)K(z,y)\dv_g(z)\,.
$$
{\it{Claim 2.}} 
For any $\phi\in C^2(M)$ and  $x\in M$, and for all $k=1,2,...$, we have
\ba\label{claim3}
\phi(x)=
&-\int_M F_k(x,y)L_g\phi(y)\dv_g(y)-\int_{\d M} F_k(x,y)B_g\phi(y)\ds_g(y)
\\
&+\int_M\Gamma_{k+1}(x,y)\phi(y)\dv_g(y)-\int_{\d M} \frac{n-2}{2(n-1)}H_g(y)F_k(x,y)\phi(y)\ds_g(y)\,.\notag
\end{align}

Claim 2 can be proved by induction on $k$.

\vspace{0.2cm}\noindent
{\it{Claim 3.}} 
For any $x\in M$ and $0<\theta<1$, the function $y\mapsto F_{n}(x,y)$ is in $C^{1,\theta}(M\backslash \{x\})$ and satisfies
\begin{equation}\label{claim2:1}
|F_{n}(x,y)|\leq Cd_g(x,y)^{2-n}\,,
\:\:\:\:
|\nabla_{g,y}F_{n}(x,y)|_g\leq Cd_g(x,y)^{1-n}\,,
\end{equation}
and
\begin{equation}\label{claim2:2}
\frac{|\nabla_{g,y}F_{n}(x,y)-\nabla_{g,y'}F_{n}(x,y')|_g}{d_g(y,y')^{\theta}}\leq Cd_g(x,y)^{1-\theta-n}+Cd_g(x,y')^{1-\theta-n}\,,
\end{equation}
for some $C=C(M,g,\rho_0)$.
In particular, for any $x\in \d M$, $y\mapsto \d F_{n}/\d \eta_{g,y}(x,y)$ defines a continuous function on $\d M\backslash \{x\}$.

As a consequence of Claim 3, if $x_0\in \d M$ we can choose $N$ large enough in the hypothesis (a)  such that $y\mapsto H_g(y) F_{n}(x_0,y)$ is in $C^{1,\theta}(\d M)$ for $0<\theta<1$ and satisfies
\begin{equation}\label{estim:HF}
\|H_g(\cdot) F_{n}(x_0,\cdot)\|_{C^{1,\theta}(\d M)}\leq C(M,g,\rho_0,\theta)\,.
\end{equation}
It is clear that \eqref{estim:HF} also holds if $x_0\in M\backslash M_{\rho_0}$ with no assumptions on $H_g$, and that its left side vanishes under the hypothesis (b). In particular \eqref{estim:HF} holds should (a), (b) or (c) holds.

Let us prove Claim 3.  Choose $y\neq x$ and a smooth curve $y_t$ such that $y_0=y$. Then, for any $r>0$,
$$
\frac{d}{dt}\int_{M\backslash B_r(y)}\Gamma_j(x,z)K(z,y_t)\dv_g(z)=\int_{M\backslash B_r(y)}\Gamma_j(x,z)\frac{d}{dt}K(z,y_t)\dv_g(z)
$$

For any $r>0$ such that $2r<d_g(x,y)$ and $t$ small,  we have 
\ba
\int_{B_r(y)}\Gamma_j(x,z)\Big|\frac{K(z,y_t)-K(z,y)}{t}\Big|\dv_g(z)
&\leq
C\int_{B_r(y)}d_g(x,z)^{1-n}(d_g(z,y_t)^{1-n}+d_g(z,y)^{1-n})\dv_g(z)\notag
\\
&\leq
C2^{n-1}d_g(x,y)^{1-n}\int_{B_r(y)}(d_g(z,y_t)^{1-n}+d_g(z,y)^{1-n})\dv_g(z)\notag
\end{align}
and the right-hand side goes to $0$ as $r\to 0$. Here, $B_r(y)$ stands for the geodesic ball centered at $y$.
Hence,
\begin{equation}\label{eq:deriv:H}
\frac{d}{dt}\int_{M}\Gamma_j(x,z)K(z,y_t)\dv_g(z)=\int_{M}\Gamma_j(x,z)\frac{d}{dt}K(z,y_t)\dv_g(z)
\end{equation}
and the estimates in (\ref{claim2:1}) follow from Giraud's result.

Now,
\ba
\frac{1}{d_g(y,y')^{\theta}}
&\Big|
\int_{M}\Gamma_j(x,z)\frac{\d}{\d y_i}K(z,y)\dv_g(z)-\int_{M}\Gamma_j(x,z)\frac{\d}{\d y_i}K(z,y')\dv_g(z)
\Big|\notag
\\
&\leq
\int_{M}\Gamma_j(x,z)
\Big|
\frac
{\frac{\d}{\d y_i}K(z,y)-\frac{\d}{\d y_i}K(z,y')}
{d_g(y,y')^{\theta}}
\Big|
\dv_g(z)\notag
\\
&\leq
C\int_{M}d_g(x,z)^{1-n}(d_g(z,y)^{1-\theta-n}+d_g(z,y')^{1-\theta-n})\dv_g(z)\notag
\\
&\leq
C(d_g(x,y)^{2-\theta-n}+d_g(x,y')^{2-\theta-n})\,,\notag
\end{align}
where we used Lemma \ref{lemma:holder} in the second inequality, and Giraud's result in the last one.

This proves Claim 3.

Using the  hypothesis $Q(M)>0$, we define $u_{x_0}\in C^{2,\theta}(M)$ as the unique solution of
\ba\label{eq:ux}
\begin{cases}
L_{g}u_{x_0}(y)=-\Gamma_{n+1}(x_0,y)\,,&\text{in}\:M\,,
\\
B_{g}u_{x_0}(y)=\frac{n-2}{2(n-1)}H_g(y)F_{n}(x_0,y)\,,&\text{on}\:\d M\,.
\end{cases}
\end{align}
It satisfies
\ba\label{estim:C2:ux}
\|u_{x_0}\|_{C^{2,\theta}(M)}
\leq C\|u_{x_0}\|_{C^{0}(M)}
&+C\|\Gamma_{n+1}(x_0,\cdot)\|_{C^{0,\theta}(M)}
+C\|H_g(\cdot)F_{n}(x_0,\cdot)\|_{C^{1,\theta}(\d M)}\,
\end{align}
where $C=C(M,g,\rho_0,\theta)$  (see  \cite[Theorems 6.30 and 6.31]{gilbarg-trudinger}.

\vspace{0.2cm}\noindent
{\it{Claim 4.}} There exists $C=C(M,g,\rho_0,\theta)$ such that $\|u_{x_0}\|_{C^{2,\theta}(M)}\leq C$.

Indeed,  using  (\ref{claim3}) with $k=n$ and any
$\phi\in C^{2}(M)$, one can see that
$$
\sup_M|\phi|\leq C\sup_M|L_g\phi|+C\sup_{\d M}|B_g\phi|+C\|\phi\|_{L^2(M)}+C\|\phi\|_{L^2(\d M)}\,.
$$
Since $Q(M)>0$, there exists $C=C(M,g)$ such that
$$
\int_M\phi^2\dv_g+\int_{\d M}\phi^2\ds_g
\leq
C\int_M |L_g(\phi)\phi|\dv_g+C\int_{\d M} |B_g(\phi)\phi|\ds_g\,.
$$
Thus, the Young's inequality implies
$$
\int_M\phi^2\dv_g+\int_{\d M}\phi^2\ds_g
\leq
C\int_M L_g(\phi)^2\dv_g+C\int_{\d M} B_g(\phi)^2\ds_g\,.
$$
Hence,
$
\|\phi\|_{C^0(M)}\leq
C\|L_g\phi\|_{C^0(M)}+C\|B_g\phi\|_{C^0(\d M)}\,.
$
Setting $\phi=u_{x_0}$ and using the equations (\ref{eq:ux}), we see that
\begin{equation}\label{estim:C0:ux}
\|u_{x_0}\|_{C^0(M)}\leq
C\|\Gamma_{n+1}(x_0,\cdot)\|_{C^0(M)}+C\|H_g(\cdot)F_{n}(x_0,\cdot)\|_{C^0(\d M)}\,.
\end{equation}
Claim 4 follows from the estimates (\ref{estim:gamma}), (\ref{estim:claim1}),
(\ref{estim:HF}), (\ref{estim:C2:ux}), and (\ref{estim:C0:ux}).

We define the function $G_{x_0}\in C^{1,\theta}(M\backslash \{x_0\})$ by
$$
G_{x_0}(y)=K(x_0,y)+\sum_{k=1}^{n}\int_M\Gamma_{i}(x_0,z)K(z,y)\dv_g(z)+u_{x_0}(y)\,.
$$
One can check that the formula (\ref{green:formula:0}) holds.

\vspace{0.2cm}\noindent
{\it{Claim 5.}} We have $G_{x_0}\in C^{\infty}(M\backslash\{x_0\})$ and \eqref{eq:G}.

In order to prove Claim 5, we rewrite (\ref{form:green:H}) as
\begin{align}\label{form:green:H:L}
\int_M K(x,y)&L_g\phi(y)\dv_g(y)
+\int_{\d M}K(x,y)B_g\phi(y)\ds_g(y)
\\
&=\int_M L_{g,y}K(x,y)\phi(y)\dv_g(y)-\phi(x)
-\int_{\d M}\frac{n-2}{2(n-1)}H_g(y)K(x,y)\phi(y)\ds_g(y)\,.\notag
\end{align}
Thus,
\begin{align}
&\int_M\left\{\int_M\Gamma_j(x,z)K(z,y)\dv_g(z)\right\}L_g\phi(y)\dv_g(y)+\int_{\d M}\left\{\int_M\Gamma_j(x,z)K(z,y)\dv_g(z)\right\}B_g\phi(y)\ds_g(y)\notag
\\
&=\int_M \Gamma_j(x,z)\left\{
\int_M K(z,y)L_g\phi(y)\dv_g(y)+\int_{\d M}K(z,y)B_g\phi(y)\ds_g(y)
\right\}\dv_g(z)\notag
\\
&=\int_M \Gamma_j(x,z)\int_M L_{g,y}K(z,y)\phi(y)\dv_g(y)\dv_g(z)\notag
\\
&\hspace{1cm}-\int_M \Gamma_j(x,z)\left\{\int_{\d M}\frac{n-2}{2(n-1)}H_g(y)K(z,y)\phi(y)\ds_g(y)+\phi(z)\right\}\dv_g(z)\notag
\\
&=\int_M \left\{\int_M \Gamma_j(x,z)L_{g,y}K(z,y)\dv_g(z)-\Gamma_j(x,y)\right\}\phi(y)\dv_g(y)\notag
\\
&\hspace{1cm}-\int_{\d M}\left\{\int_M \Gamma_j(x,z)K(z,y)\dv_g(z)\right\}\frac{n-2}{2(n-1)}H_g(y)\phi(y)\ds_g(y)\,,\notag
\end{align}
where we used (\ref{form:green:H:L}) in the second equality. Hence, we proved that the equations
\begin{equation}\notag
\begin{cases}
L_{g,y}\int_M\Gamma_j(x,z)K(z,y)\dv_g(z)=\Gamma_{j+1}(x,y)-\Gamma_{j}(x,y)\,,&\text{in}\:M\,,
\\
B_{g,y}\int_M\Gamma_j(x,z)K(z,y)\dv_g(z)=-\frac{n-2}{2(n-1)}H_g(y)\int_M\Gamma_j(x,z)K(z,y)\dv_g(z)\,,&\text{on}\:\d M\,,
\end{cases}
\end{equation}
hold in the sense of distributions. Then it is easy to check that the equations (\ref{eq:G}) hold in the sense of distributions.
Since $G_{x_0}\in C^{1,\theta}(M\backslash\{x_0\})$, elliptic regularity arguments  imply that $G_{x_0}\in C^{\infty}(M\backslash\{x_0\})$. 
This proves Claim 5.

The property (P1) follows from  (\ref{claim2:1}) and Claim 4.  In order to
prove (P2),(P3) and (P4), we use (\ref{estim:gamma}), (\ref{estim:gamma:h}), (\ref{eq:deriv:H}) and Claim 4.

\vspace{0.2cm}\noindent
{\it{Claim 6.}} The function $G_{x_0}$ is positive on $M\backslash\{x_0\}$.

Let us prove Claim 6.
Let
$$
G_{x_0}^ -=
\begin{cases}
-G_{x_0}\,,&\text{if}\:G_{x_0}<0\,,
\\
0\,,&\text{if}\:G_{x_0}\geq 0\,.
\end{cases}
$$
Since $G_{x_0}^-$ has support in $M\backslash\{x_0\}$, one has
\ba
0&=-\int_MG_{x_0}^-L_gG_{x_0}\dv_g-\int_{\d M}G_{x_0}^-B_gG_{x_0}\ds_g\notag
\\
&=\int_M\left(|\nabla_gG_{x_0}^-|_g^2+\frac{n-2}{4(n-1)}R_g(G_{x_0}^-)^2\right)\dv_g+\int_{\d M}\frac{n-2}{2(n-1)}H_g(G_{x_0}^-)^2\ds_g\,.\notag
\end{align}
By the hypothesis $Q(M)>0$, we have $G_{x_0}^-\equiv 0$ which implies  $G_{x_0}\geq 0$.

We now change the metric by a conformal positive factor $u\in C^{\infty}(M)$ such that $\tilde{g}=u^{\frac{4}{n-2}}g$ satisfies $R_{\tilde{g}}>0$ in $M$ and $H_{\tilde{g}}\equiv0$ on $\d M$ (see \cite{escobar3}). Observing the conformal properties (\ref{propr:L}) and (\ref{propr:B}), we see that $\tilde{G}=u^{-1}G_{x_0}\geq 0$ satisfies $L_{\tilde{g}}\tilde{G}=0$ in $M\backslash \{x_0\}$ and $B_{\tilde{g}}\tilde{G}=0$ on $\d M\backslash \{x_0\}$. Then the strong maximum principle implies $\tilde{G}>0$, proving Claim 6.

This finishes the proof of Proposition \ref{green:point}.
\ep

Let $(M,g_0)$ be a Riemannian manifold with $Q(M)>0$ and $H_{g_0}\equiv 0$. Let $g_{x_0}=f_{x_0}^{\frac{4}{n-2}}g_0$ be a conformal metric satisfying 
$$
|f_{x_0}(x)-1|\leq C(M,g_0)d_{g_0}(x,x_0).
$$
\begin{notation}
For a Riemannian metric $g$ we set $M_{t, g}=\{x\in M:d_{g}(x,\pa M)<t\}$ and $\partial_{t, g} M=\{x\in M: d_g(x,\pa M)=t\}$.
\end{notation}
\begin{proposition}\label{green:tubular}
If $\rho_0$ is sufficiently small and $x_0\in M_{\rho_0, g_{x_0}}\backslash \d M$, then there  exists a positive $G_{x_0}\in C^\infty(M\backslash\{x_0\})$ satisfying 
\begin{align}\label{green-gx0}
\begin{cases}
L_{g_{x_0}}\Gx=0,\quad \text{in}\:M\backslash\{x_0\},\\
B_{g_{x_0}}\Gx=0,\quad \text{on}\:\pa M,
\end{cases}
\end{align}
and there exists $C=C(M,g_0,\rho_0)$ such that 
\begin{align}\label{goal-1}
|\Gx(y)-|\phi_0(y)|^{2-n}|\leq 
\begin{cases}
C|\phi_0(y)|^{3-n}+Cd_{g_{x_0}}(x_0,\partial M)|\phi_0(y)|^{1-n}&n\geq 4,
\\
C(1+|\log(|\phi_0(y)|)|)+Cd_{g_{x_0}}(x_0,\partial M)|\phi_0(y)|^{1-n}&n=3,
\end{cases}
\end{align}
\begin{equation}\label{goal-2}
|\nabla_{g_{x_0}}(\Gx(y)-|\phi_0(y)|^{2-n})|\leq C|\phi_0(y)|^{1-n}+Cd_{g_{x_0}}(x_0,\partial M)|\phi_0(y)|^{-n},
\end{equation}
where $\phi_0(y)=(y_1,...,y_n)$ are $g_{x_0}$- normal coordinates centered at $x_0$.
\end{proposition}
\begin{proof}
We will use the notation $d(x)=d_{g_0}(x,\partial M)$. 
Let us define the coordinate system $\psi_0(y)=(y_1,...,y_n)$ on $M_{\rho_0, g_0}$ where $(y_1,\cdots,y_{n-1})$ are normal coordinates of $y_{x_0}$ on $\pa_{d(x_0),g_0} M$ centered at ${x_0}$, with respect to the metric induced by $g_0$, and $y_n=d(y)-d(x_0)$. Here, $y_{x_0}\in \partial_{d(x_0),g_0}M$ is such that $d_{g_0}(y,y_{x_0})=d_{g_0}(y,\partial_{d(x_0),g_0}M)$. This differs from $\psi_{x_0}$ defined above by a translation in the last coordinate.


According to Proposition \ref{green:point}, multiplying it by some constant, one can construct a function $G_0$, satisfying \begin{equation*}
\begin{cases}
L_{g_0}G_{0}=0\,,&\text{in}\:M\backslash \{x_0\}\,,
\\
B_{g_0}G_{0}=0\,,&\text{on}\:\d M\,,
\end{cases}
\end{equation*} 
\begin{align*}
\Big|G_{0}(y)-&\frac{1}{2}\big(|(y_1,...,y_n)|^{2-n}+|(y_1,...,y_{n-1}, y_n+2d(x_0))|^{2-n}\big)\Big|
\leq 
\begin{cases}
Cd_{g_0}(y,x_0)^{3-n}&n\geq 4,
\\
C(1+|\log d_{g_0}(y,x_0)|)&n=3,
\end{cases}
\end{align*}
and
$$
\Big|\nabla_{g_0}\big(G_{0}(y)-\frac{1}{2}\big(|(y_1,...,y_n)|^{2-n}+|(y_1,...,y_{n-1}, y_n+2d(x_0))|^{2-n}\big)\big)\Big|
\leq Cd_{g_0}(y,x_0)^{2-n}.
$$ 
for some $C=C(M,g_0,\rho_0)$.
Using $|(y_1,...,y_{n-1}, y_n+2d(x_0))|\geq |(y_1,...,y_n)|$ and Lemma \ref{lemma:holder} we have
\begin{align*}
\left||(y_1,...,y_n)|^{2-n}-|(y_1,...,y_{n-1}, y_n+2d(x_0))|^{2-n}\right|\leq Cd(x_0)|(y_1,...,y_n)|^{1-n},\\
\left|\nabla|(y_1,...,y_n)|^{2-n}-\nabla|(y_1,...,y_{n-1}, y_n+2d(x_0))|^{2-n}\right|\leq Cd(x_0)|(y_1,...,y_n)|^{-n}.
\end{align*}
Then
\begin{equation}\label{G0-1}
|G_0(y)-|\psi_0(y)|^{2-n}|
\leq  
\begin{cases}
Cd_{g_0}(y,x_0)^{3-n}+Cd(x_0)d_{g_0}(y,x_0)^{1-n}&n\geq 4,
\\
C(1+|\log d_{g_0}(y,x_0)|)+Cd(x_0)d_{g_0}(y,x_0)^{1-n}&n=3,
\end{cases}
\end{equation}
\begin{equation}\label{G0-2}
|\nabla_{g_0}(G_0(y)-|\psi_0(y)|^{2-n})|
\leq Cd_{g_0}(y,x_0)^{2-n}+Cd(x_0)d_{g_0}(y,x_0)^{-n}.
\end{equation}

Now we change this to the conformal metric $g_{x_0}$. Let $\phi_0(y)=(y_1,...,y_n)$ be $g_{x_0}$-conformal normal coordinates centered at $x_0$. 
By the definition of $\phi_0$ and $\psi_0$ one can check that $\xi=\phi_0\circ \psi_0^{-1}$  satisfies $\xi(0)=0$ and $d\xi(0)=id_{\mathbb{R}^n}$. Since $M$ is compact, one can find $C=C(M,g_0)$ uniform in $x_0$ such that 
\begin{align}\label{trans-relation}
|\xi(y_1,...,y_n)-(y_1,...,y_n)|\leq C|(y_1,...,y_n)|^2.
\end{align} 

The function $G_{x_0}=f_{x_0}^{-1}G_0$  satisfies \eqref{green-gx0}, so we shall prove \eqref{goal-1} and \eqref{goal-2}.
Observe that 
\begin{equation}\label{Gx0}
|G_{x_0}(y)-G_0(y)|
\leq Cd_{g_0}(y,x_0)|G_{x_0}(y)|
\leq Cd_{g_0}(y,x_0)^{3-n}.
\end{equation}

Combining \eqref{G0-1}, \eqref{trans-relation} and \eqref{Gx0}, one gets \eqref{goal-1} from the following steps:
\begin{align*}
|G_{x_0}(y)-|\phi_0(y)|^{2-n}|
&\leq |G_{x_0}(y)-G_0(y)|+|G_0(y)-|\psi_0(y)|^{2-n}|+||\psi_0(y)|^{2-n}-|\xi\circ \psi_0(y)|^{2-n}|
\\
&\leq Cd_{g_0}(y,x_0)^{3-n}+Cd(x_0)d_{g_0}(y,x_0)^{1-n}+C|\psi_0(y)|^{3-n}
\\
&\leq Cd_{g_0}(y,x_0)^{3-n}+Cd_{g_{x_0}}(x_0,\partial M)(x_0)d_{g_0}(y,x_0)^{1-n}
\end{align*}
for $n\geq 4$ and with obvious modifications for $n=3$.
Similarly, using \eqref{G0-2}, \eqref{trans-relation} and \eqref{Gx0}, one gets \eqref{goal-2}.
\end{proof}
\end{appendices}

\noindent
\textsc{Instituto de Matem\'{a}tica, Universidade Federal Fluminense (UFF), 
\\Rua Prof. Marcos Waldemar de Freitas S/N, Niter\'{o}i, RJ, 24210-201, Brazil.}
\\{\bf{almaraz@vm.uff.br}}

\vspace{0.5cm}
\noindent
\textsc{Johns Hopkins University, Mathematics Department\\3400 N. Charles St., 222 Krieger Hall, Baltimore, MD 21218, USA.}
\\{\bf{lsun@math.jhu.edu}}

\end{document}